\documentclass[12pt]{amsart}%
\usepackage{amsmath}
\usepackage{amsthm}
\usepackage{amssymb}
\usepackage{amscd}
\usepackage{amsfonts}
\usepackage{amsbsy}
\usepackage{epsfig,afterpage}
\usepackage[dvips]{psfrag}
\usepackage{subfigure}
\usepackage{epsfig}
\usepackage{amsmath}
\usepackage{srcltx}
\usepackage{amsfonts}
\usepackage{amssymb}
\usepackage{psfrag}
\usepackage{verbatim}
\usepackage{graphicx}
\usepackage{amsmath}
\usepackage{amsthm}
\usepackage{amssymb}
\usepackage{amscd}
\usepackage{amsfonts}
\usepackage{amsbsy}
\usepackage{epsfig}
\usepackage[dvips]{psfrag}
\usepackage{multirow}
\usepackage{graphics}
\usepackage{epstopdf}
\usepackage{overpic}
\usepackage{float}

\usepackage{graphicx}

\def\a{\alpha}
\def\b{\beta}

\def\g{\gamma}
\newcommand{\s}{\ensuremath{\mathbb{S}}}

\newcommand{\bbox}{\ \hfill\rule[-1mm]{2mm}{3.2mm}}

\newtheorem {theorem} {Theorem}%[section]
\newtheorem {proposition} [theorem]{Proposition}

\newtheorem {lemma}  [theorem]{Lemma}

\newcommand{\R}{\mathbb{R}}

\newcommand{\sss}{\ensuremath{\mathbb{S}}}

\newcommand{\D}{\ensuremath{\mathbb{D}}}

\begin{document}

\title[Riccati Quadratic Polynomial Differential System]
{Phase Portraits of the Riccati Quadratic Polynomial  Differential Systems}

\author[ J. Llibre, B.D. Lopes and P.R. da Silva]
{Jaume Llibre$^1$, Bruno D. Lopes $^2$ and Paulo R. da Silva$^3$}

\address{$^1$ Departament de Matem\`{a}tiques,
Universitat Aut\`{o}noma de Barcelona, 08193 Bellaterra, Barcelona,
Catalonia, Spain.}

\address{$^2$ IMECC--UNICAMP, CEP 13081--970, Campinas,
S\~ao Paulo, Brazil.}

\address{$^3$  Departamento de Matem{\'a}tica -- IBILCE--UNESP,
Rua C. Colombo, 2265, CEP 15054--000 S. J. Rio Preto, S\~ao Paulo,
Brazil.}

\email{jllibre@mat.uab.cat}

\email{brunodomicianolopes@gmail.com}

\email{paulo.r.silva@unesp.br}

\thanks{$^*$ The first author is partially supported by the Ministerio de Econom\'ia,
Industria y Competitividad, Agencia Estatal de Investigaci\'on grant
MTM2016-77278-P (FEDER), the Ag\`encia de Gesti\'o d'Ajuts
Universitaris i de Recerca grant 2017SGR1617, and the H2020 European
Research Council grant MSCA-RISE-2017-777911. The second author is supported by PNPD/CAPES-IMECC/UNICAMP.
The  third author was partially supported by CAPES, CNPq and FAPESP.  All the authors are
supported by FP7-PEOPLE-2012-IRSES-316338.}


\date{}
\dedicatory{}

\maketitle
\begin{abstract}
In this paper we characterize the phase portrait of the Riccati quadratic polynomial differential systems
$$\dot{x}= \a_2(x),\quad\dot{y} = ky^2+\b_1(x) y + \g_2(x), $$
 with $(x,y)\in\R^2$,  $\g_2(x)$ non-zero (otherwise the system is a Bernoulli differential system),
 $k\neq0$ (otherwise the system is a Lienard differential system),
$\b_ 1(x)$ a polynomial of degree at most $1$, $\a_ 2(x)$ and $\g_ 2(x)$ polynomials of degree at most 2, and
the maximum of the degrees of $ \a_2(x)$  and $k y^2+\b_1(x) y + \g_2(x)$ is 2.
We give the complete description of their phase portraits in the Poincar\'{e} disk  (i.e. in the
compactification of $\R^2$ adding the circle $\mathbb {S}^1$ of the
infinity) modulo topological equivalence.\end{abstract}

\section{Introduction and statement of the main results}   \label{s1}

Numerous problems of applied mathematics are modeled by quadratic polynomial differential systems, see for instance \cite{MTMN}.
Excluding linear systems, such systems are the ones with the lowest degree of complexity,
and the large bibliography on the subject  proves its relevance.
We refer for example to the books of Ye Yanqian et al. \cite{Ye}, Reyn \cite{Re},
and Artes, Llibre, Schlomiuk,  Vulpe \cite{Ar}, and the surveys of Coppel \cite{Co}, and Chicone and
Jinghuang \cite{CJ}  are excellent introductory readings to the quadratic polynomial differential systems.\\

In this paper we characterize the phase portraits of the Riccati quadratic differential systems
\begin{equation}\dot{x}= \a_2(x),\quad\dot{y} = ky^2+\b_1(x) y + \g_2(x), \label{eq1}\end{equation}
 with $(x,y)\in\R^2$,  $\g_2(x)$ non-zero (otherwise the system is a Bernoulli differential system),
 $k\neq0$ (otherwise the system is a Lienard differential system),
$\b_ 1(x)$ a polynomial of degree at most $1$, $\a_ 2(x)$ and $\g_ 2(x)$ polynomials of degree at most 2, and
the maximum of the degrees of $ \a_2(x)$  and $k y^2+\b_1(x) y + \g_2(x)$ is 2. In \eqref{eq1}  the dot denotes
derivative with respect to the time.

\begin{proposition}\label{p0} A Riccati quadratic differential system \eqref{eq1}
is topologically equivalent to  one of the following systems:
\[\begin{array}{clll}
(i)\quad &\dot{x}=x(x+1),& \dot{y}&=y^2+(ax+b) y + cx^2+ dx + e;\\
(ii)\quad &\dot{x}=x^2,&\dot{y}&=y^2+(ax+b) y + cx^2+ dx + e;\\
(iii)\quad &\dot{x}=x,&\dot{y}&=y^2+(ax+b) y + cx^2+ dx + e;\\
(iv)\quad &\dot{x}=1, &\dot{y}&=y^2+(ax+b) y + cx^2+ dx + e;\\
(v)\quad &\dot{x}=x^2+1, &\dot{y}&=y^2+(ax+b) y + cx^2+ dx + e.
\end{array}
\]
with $c^2+d^2+e^2\neq 0$ in all these systems.

\end{proposition}
We note that the Riccati systems have no periodic orbits. In fact,
the equilibrium points of  systems (i), (ii) and (iii) are on invariant straight lines and
 systems (iv) and (v) do not have equilibrium points, and consequently they do not have limit cycles, because
 it is well known that  a periodic orbit in the plane must surrounds at least one equilibrium point.

The objective of this work is to classify the phase portraits of the Riccati quadratic polynomial
 differential systems \eqref{eq1} in the Poincar\'{e} disk modulo topological equivalence.
As any polynomial differential system, system \eqref{eq1}
can be extended to an analytic system on a closed disk of radius
one, whose interior is diffeomorphic to $\R^2$ and its  boundary,
the circle $\s^1,$ plays the role of the infinity.
This closed disk is denoted by $\D^2$ and called the
\emph{Poincar\'{e} disk}, because the technique for doing such an
extension is precisely the {\it Poincar\'{e} compactification} for a
polynomial differential system  in $\R^2$, which is described in
details in chapter 5 of \cite{DLA}.  In this paper we shall use the notation of that chapter.
By using this compactification technique the dynamics of
system \eqref{eq1} in a neighborhood of the  infinity can be studied and we have the
following result.

\begin{theorem}\label{mainteo} The phase portraits of the Riccati  system \eqref{eq1}
in the Poincar\'{e} disk are topologically equivalent to one of the 74 phase portraits
presented in Figures $1$, $2$ and $3$. The phase portraits of the   systems of Proposition 1
 are provided in Tables $1$, $2$, $3$, $4$ and $5$ where
\begin{equation}\begin{array}{lll}&\Delta_{F_1} =b^2-4 e,  &\Delta_{F_2} =(b-a)^2-4  (c-d+e),\\
&\Delta_{I_1}= (a-1)^2-4 c, &\Delta_{I_2}= a^2-4  c.\end{array}\label{deltas}\end{equation}
\end{theorem}

Three papers on generalizations of Riccati differential equations can be found in \cite{FP, LOV, Wu}.

This paper is organized as follows. In section \ref{s2} we prove Proposition \ref{p0},
and study the finite equilibria. In section \ref{s3} we study the  infinite
equilibria.  Finally in section \ref{s4}  we prove Theorem \ref{mainteo}.
\newpage
\begin{table}[!htb]
    \begin{center}
        \begin{tabular}{|c|c|c|}
            \hline
            \small{Phase Portraits of systems (i)} & \small{conditions}  \\
            \hline
            \small {$ P1, P2, P3, P4, P5  $}&   \small {$ \Delta_{I_1}>0,  \Delta_{F_1}>0,  \Delta_{F_2}>0 $}\\
            \hline
                \small {$P6, P7, P8, P9 $}& \small {$ \Delta_{I_1}>0,  \Delta_{F_1}>0,  \Delta_{F_2}=0$}\\
            \hline
            \small {    $P10$}&     \small {$ \Delta_{I_1}>0,   \Delta_{F_1}>0,  \Delta_{F_2}<0$}\\
            \hline
                \small {$P11, P12, P13, P14$}&  \small { $ \Delta_{I_1}>0,  \Delta_{F_1}=0,  \Delta_{F_2}>0$}\\
            \hline
                \small {$P15, P16 $}&   \small { $ \Delta_{I_1}>0,  \Delta_{F_1}=0,  \Delta_{F_2}=0 $}\\
            \hline
            \small {    $P17$}& \small {$\Delta_{I_1}>0,  \Delta_{F_1}=0,  \Delta_{F_2}<0$}\\
            \hline
                \small {$P18$}& \small{$ \Delta_{I_1}>0,  \Delta_{F_1}<0,  \Delta_{F_2}>0$}\\
            \hline
                \small {$ P19$}&    \small { $ \Delta_{I_1}>0,  \Delta_{F_1}<0,  \Delta_{F_2}=0 $}\\
            \hline
                \small {$ P20 $}&   \small {$ \Delta_{I_1}>0,  \Delta_{F_1}<0,  \Delta_{F_2}<0 $}\\
            \hline
            \small {    $P21, P22, P23 $}&  \small {$ \Delta_{I_1}=0,  \Delta_{F_1}>0,  \Delta_{F_2}>0 $}\\
            \hline
                \small {$P24, P25$}&    \small {$ \Delta_{I_1}=0,  \Delta_{F_1}>0,  \Delta_{F_2}=0$}\\
            \hline
            \small {    $P26 $}&    \small {$ \Delta_{I_1}=0,  \Delta_{F_1}>0,  \Delta_{F_2}<0$}\\
            \hline
            \small {    $P27 $}&    \small { $\Delta_{I_1}=0,  \Delta_{F_1}=0,  \Delta_{F_2}>0$}\\
            \hline
                \small {$P28 $}&    \small { $ \Delta_{I_1}=0,  \Delta_{F_1}=0,  \Delta_{F_2}=0 $}\\
            \hline
            \small {    $P29$}&\small {$ \Delta_{I_1}=0,  \Delta_{F_1}=0,  \Delta_{F_2}<0$}\\
            \hline
                \small {$P30 $}&    \small {$\Delta_{I_1}=0,  \Delta_{F_1}<0,  \Delta_{F_2}>0$}\\
            \hline
            \small {    $P31$}&     \small {$ \Delta_{I_1}=0,  \Delta_{F_1}<0,  \Delta_{F_2}=0 $}\\
            \hline
            \small {    $P32 $}&    \small { $ \Delta_{I_1}=0,  \Delta_{F_1}<0,  \Delta_{F_2}<0 $}\\
            \hline
                \small {$P33 $}&    \small {$ \Delta_{I_1}<0,  \Delta_{F_1}>0,  \Delta_{F_2}>0 $}\\
            \hline
            \small {    $P34 $}&    \small {$ \Delta_{I_1}<0,  \Delta_{F_1}>0,  \Delta_{F_2}=0$}\\
            \hline
            \small {    $P35$}& \small { $ \Delta_{I_1}<0,  \Delta_{F_1}>0,  \Delta_{F_2}<0$}\\
            \hline
                \small {$P36 $}&    \small {  $ \Delta_{I_1}<0,  \Delta_{F_1}=0,  \Delta_{F_2}>0$}\\
            \hline
            \small {    $P37$}& \small { $ \Delta_{I_1}<0,  \Delta_{F_1}=0,  \Delta_{F_2}=0 $}\\
            \hline
            \small {    $P38$}& \small {$\Delta_{I_1}<0,  \Delta_{F_2}=0,  \Delta_{I_1}<0$}\\
            \hline
                \small {$P39$}& \small { $ \Delta_{I_1}<0,  \Delta_{F_1}<0,  \Delta_{F_2}>0$}\\
            \hline
                \small {$P40$}&     \small {$ \Delta_{I_1}<0,  \Delta_{F_1}<0,  \Delta_{F_2}=0$}\\
            \hline
                \small {$P41$}& \small { $\Delta_{I_1}<0,  \Delta_{F_1}<0,  \Delta_{F_2}<0$}\\
            \hline
        \end{tabular}
    \end{center}
    \caption{\small The phase portraits of systems (i). }
    \label{t6}
\end{table}

\begin{table}[!htb]
    \begin{center}
        \begin{tabular}{|c|c|c|}
            \hline
            \small{Phase Portraits of systems (ii)} & \small{conditions}  \\
            \hline
            \small{$P42, P43, P44 $}&   \small{ $  \Delta_{I_1}>0,  \Delta_{F_1}>0$}\\
            \hline
                \small{$P45, P46, P47 $}&   \small{ $ \Delta_{I_1}>0,  \Delta_{F_1}=0$}\\
            \hline
                \small{$P48$}&      \small{$ \Delta_{I_1}>0,  \Delta_{F_1}<0$}\\
            \hline
            \small{ $P49, P50, P51  $}& \small{ $  \Delta_{I_1}=0,  \Delta_{F_1}>0$}\\
            \hline
                \small{$P52, P53, P54 $}&   \small{$ \Delta_{I_1}=0,  \Delta_{F_1}=0$}\\
            \hline
                \small{$P55$}&  \small{  $ \Delta_{I_1}=0,  \Delta_{F_1}<0$}\\
            \hline
            \small{ $P56 $}&    \small{$  \Delta_{I_1}<0,  \Delta_{F_1}>0$}\\
            \hline
            \small{ $P57, P58 $}&   \small{ $ \Delta_{I_1}<0,  \Delta_{F_1}=0$}\\
            \hline
            \small{ $P41$}&     \small{$ \Delta_{I_1}<0,  \Delta_{F_1}<0$}\\
            \hline
        \end{tabular}
    \end{center}
    \caption{\small The phase portraits of systems (ii).}
    \label{t2}
\end{table}

\begin{table}[!htb]
    \begin{center}
        \begin{tabular}{|c|c|c|}
            \hline
            \small{Phase Portraits of systems (iii)} & \small{conditions}  \\
            \hline
        \small{ $ P59, P60, P61 $}&     \small{$  \Delta_{I_2}>0,  \Delta_{F_1}>0$}\\
            \hline
            \small{ $P62, P63, P64 $}&  \small{ $ \Delta_{I_2}>0,  \Delta_{F_1}=0$}\\
            \hline
            \small{ $P65$}& \small{  $ \Delta_{I_2}>0,  \Delta_{F_1}<0$}\\
            \hline
                \small{$P66, P67 $}&    \small{ $  \Delta_{I_2}=0,  \Delta_{F_1}>0$}\\
            \hline
            \small{ $P68, P69$}&    \small{$ \Delta_{I_2}=0,  \Delta_{F_1}=0$}\\
            \hline
                \small{$P32 $}&     \small{$ \Delta_{I_2}=0,  \Delta_{F_1}<0$}\\
            \hline
            \small{ $P35$}&     \small{$  \Delta_{I_2}<0,  \Delta_{F_1}>0$}\\
            \hline
            \small{ $P38$}&     \small{$ \Delta_{I_2}<0,  \Delta_{F_1}=0$}\\
            \hline
                \small{$P41$}&      \small{$ \Delta_{I_2}<0,  \Delta_{F_1}<0$}\\

            \hline
        \end{tabular}
    \end{center}
    \caption{\small The phase portraits of systems (iii).}
    \label{t3}
\end{table}

\begin{table}[!htb]
    \begin{center}
        \begin{tabular}{|c|c|c|}
            \hline
        \small  {Phase Portrait  of systems (iv) } & \small {conditions}  \\
            \hline
                \small  {$P70, P71 $} & \small  { $  \Delta_{I_2}>0$}\\
            \hline
            \small  {   $P72, P73, P74 $} & \small  { $   \Delta_{I_2}=0$} \\
            \hline
                \small  {$P41$} &   \small  { $  \Delta_{I_2}<0$ }\\
            \hline
        \end{tabular}
    \end{center}
    \caption{\small The phase portraits of systems (iv).}
    \label{t4}
\end{table}

\begin{table}[!htb]
    \begin{center}
        \begin{tabular}{|c|c|c|}
            \hline
            \small{Phase Portraits of systems (v)} & \small{conditions}  \\
    \hline
            \small  {$P70, P71 $}&  \small  {$  \Delta_{I_1}>0$}\\
            \hline
            \small  {   $P72, P73, P74$}&   \small  { $  \Delta_{I_1}=0$}\\
            \hline
            \small  {   $P41$}&     \small  {$  \Delta_{I_1}<0$}\\
            \hline
        \end{tabular}
    \end{center}
    \caption{\small The phase portraits of systems (v).}
    \label{t5}
\end{table}

\newpage

\begin{figure}
    \begin{minipage}[t]{2.7cm}\psfrag{a}{$e$}\centering\includegraphics[scale=.31]{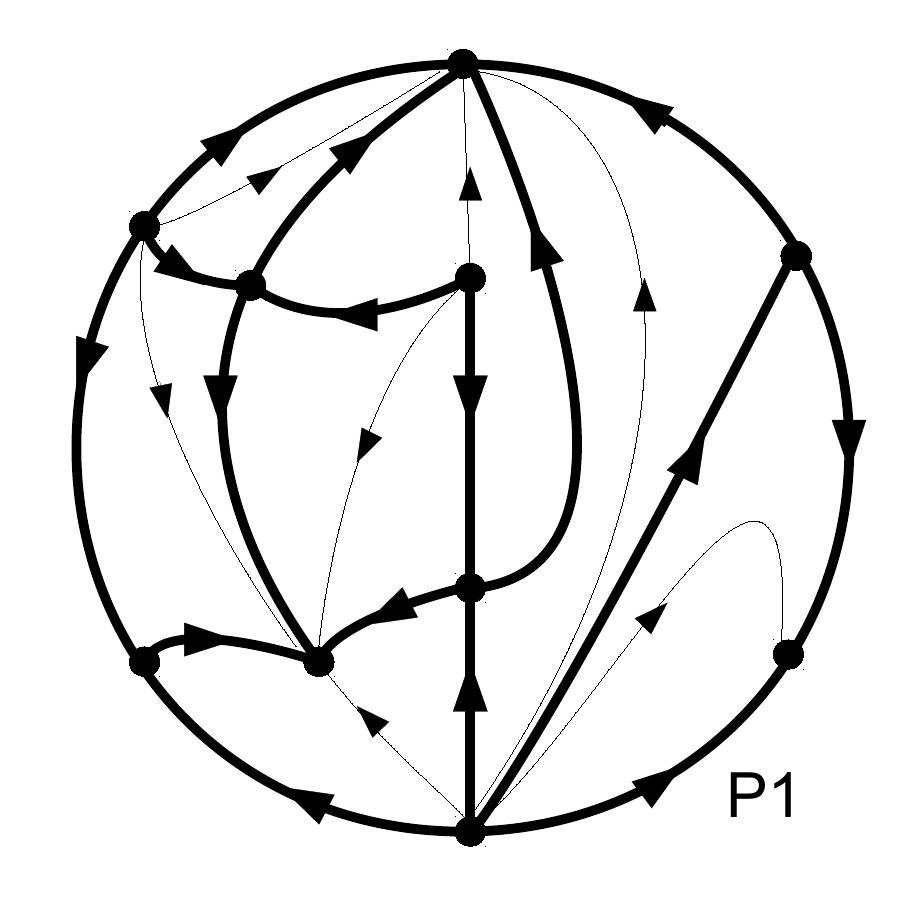}\end{minipage}
    \begin{minipage}[t]{2.7cm}\psfrag{b}{$b$}\centering\includegraphics[scale=.31]{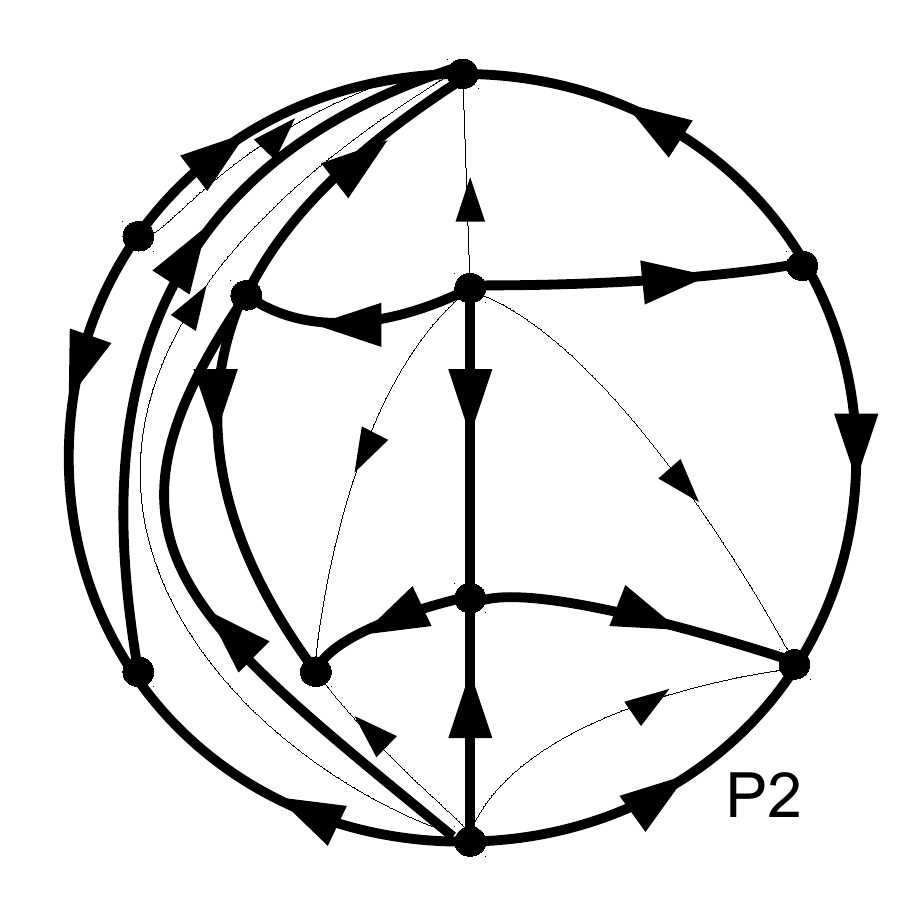}\end{minipage}
    \begin{minipage}[t]{2.7cm}\psfrag{c}{$c$}\centering\includegraphics[scale=.31]{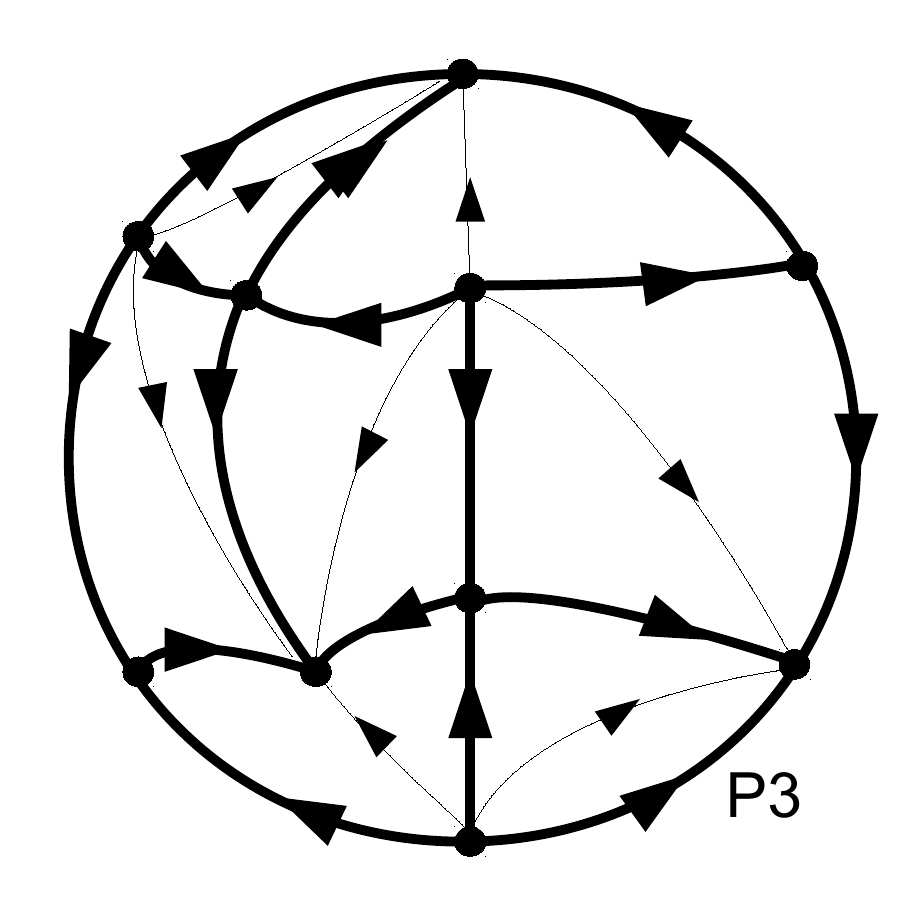} \end{minipage}
    \begin{minipage}[t]{2.7cm}\psfrag{d}{$d$}\centering\includegraphics[scale=.31]{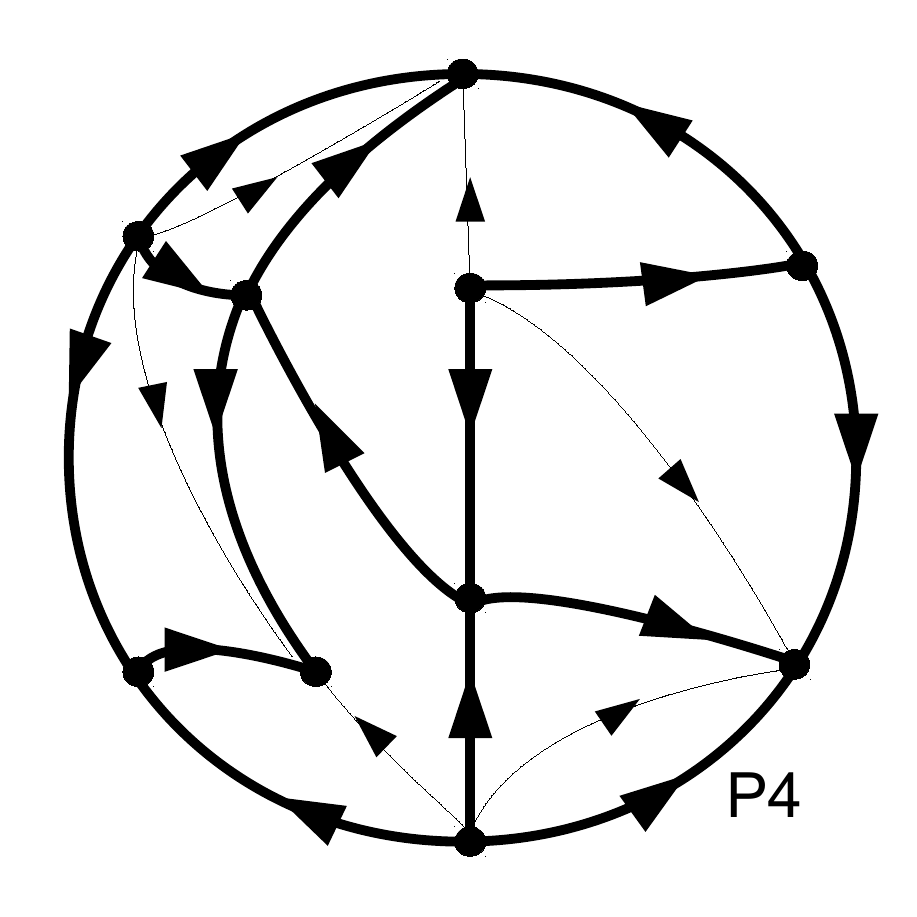}\end{minipage}
    %\begin{minipage}[t]{2.7cm}\psfrag{d}{$d$}\centering\includegraphics[scale=.31]{f27e.pdf}\end{minipage}

    \begin{minipage}[t]{2.7cm} \psfrag{a}{$a$}\centering\includegraphics[scale=.31]{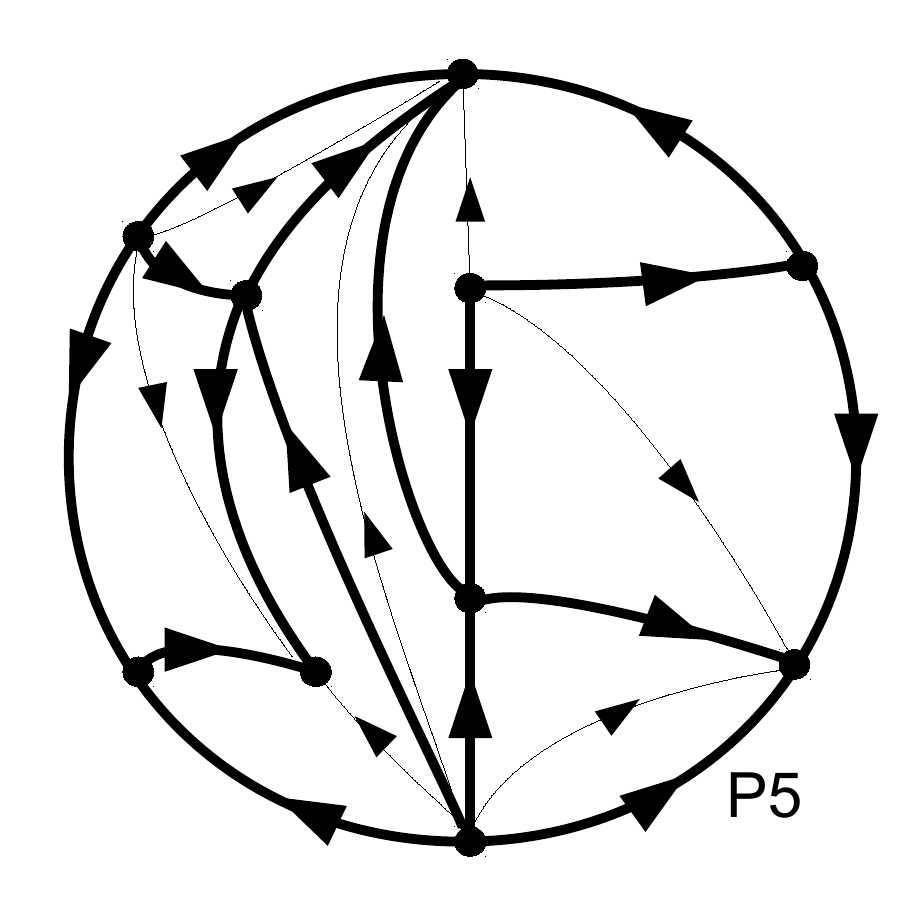}\end{minipage}
    \begin{minipage}[t]{2.7cm}\psfrag{b}{$b$}\centering\includegraphics[scale=.31]{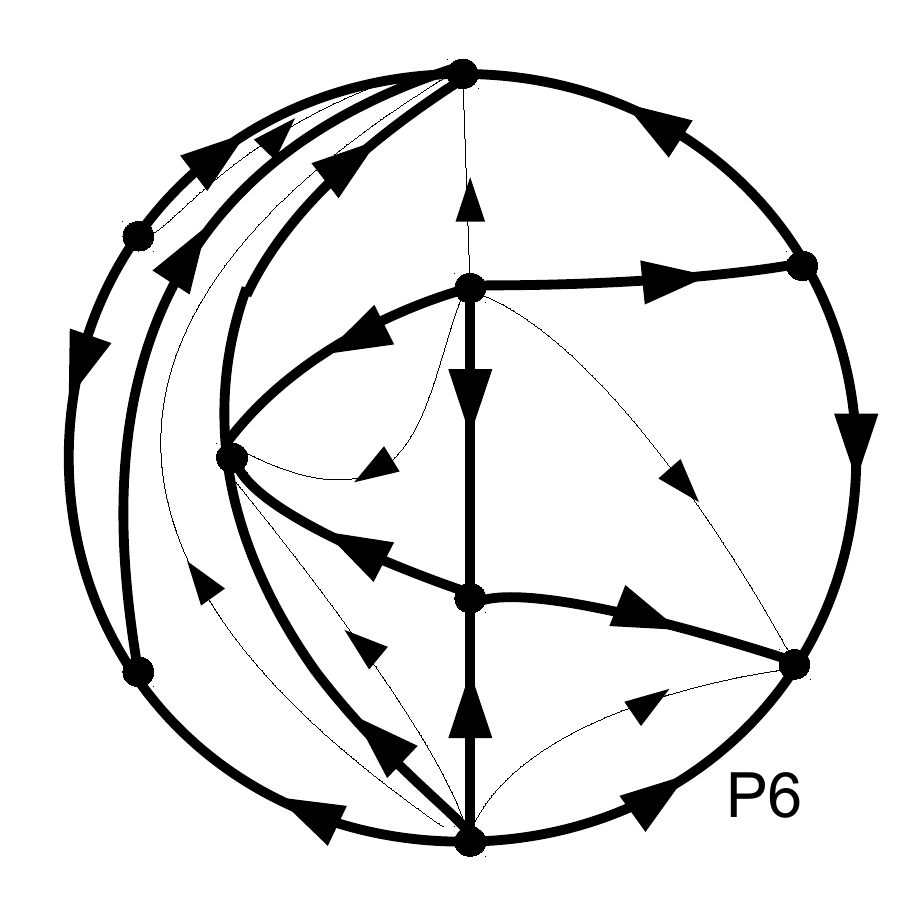}\end{minipage}
    \begin{minipage}[t]{2.7cm}\psfrag{c}{$c$}\centering\includegraphics[scale=.31]{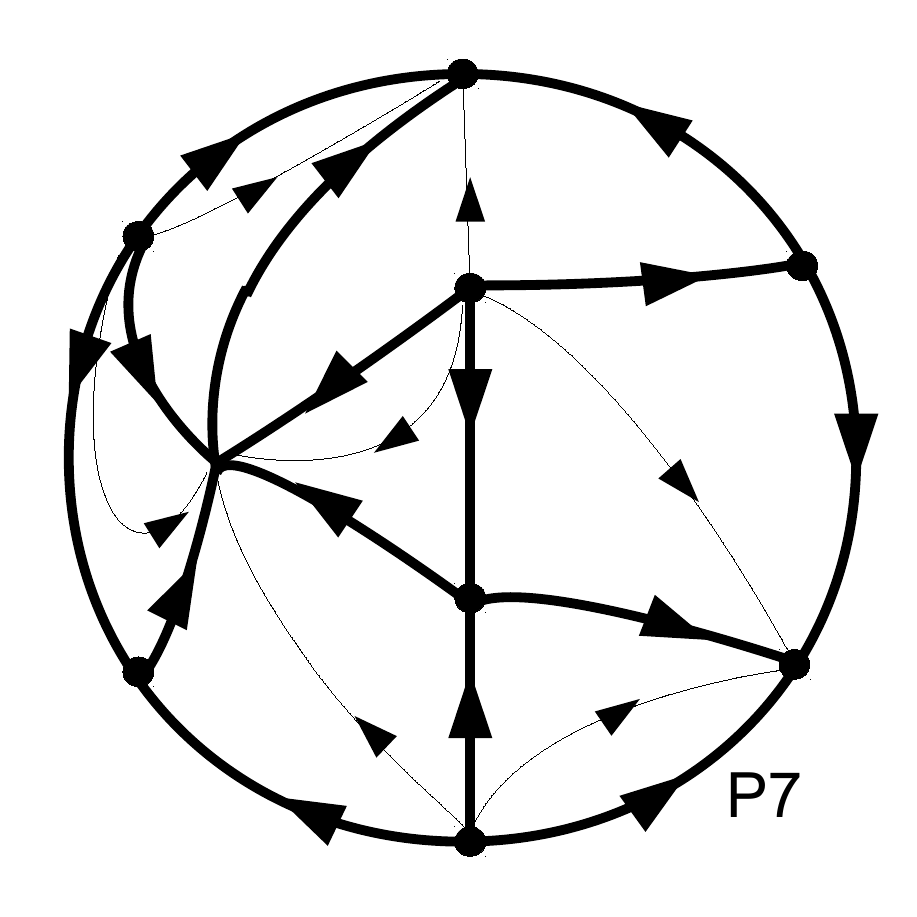}\end{minipage}
    \begin{minipage}[t]{2.7cm}\psfrag{d}{$d$}\centering\includegraphics[scale=.31]{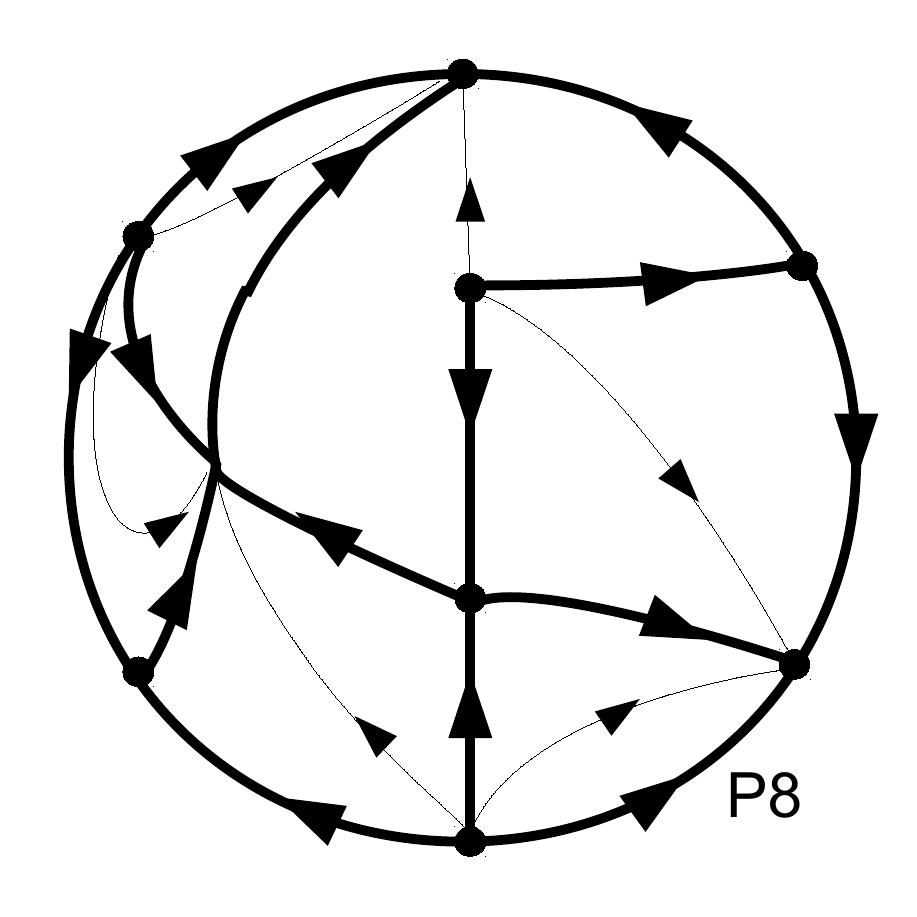}\end{minipage}

    \begin{minipage}[t]{2.7cm} \psfrag{a}{$a$}\centering\includegraphics[scale=.31]{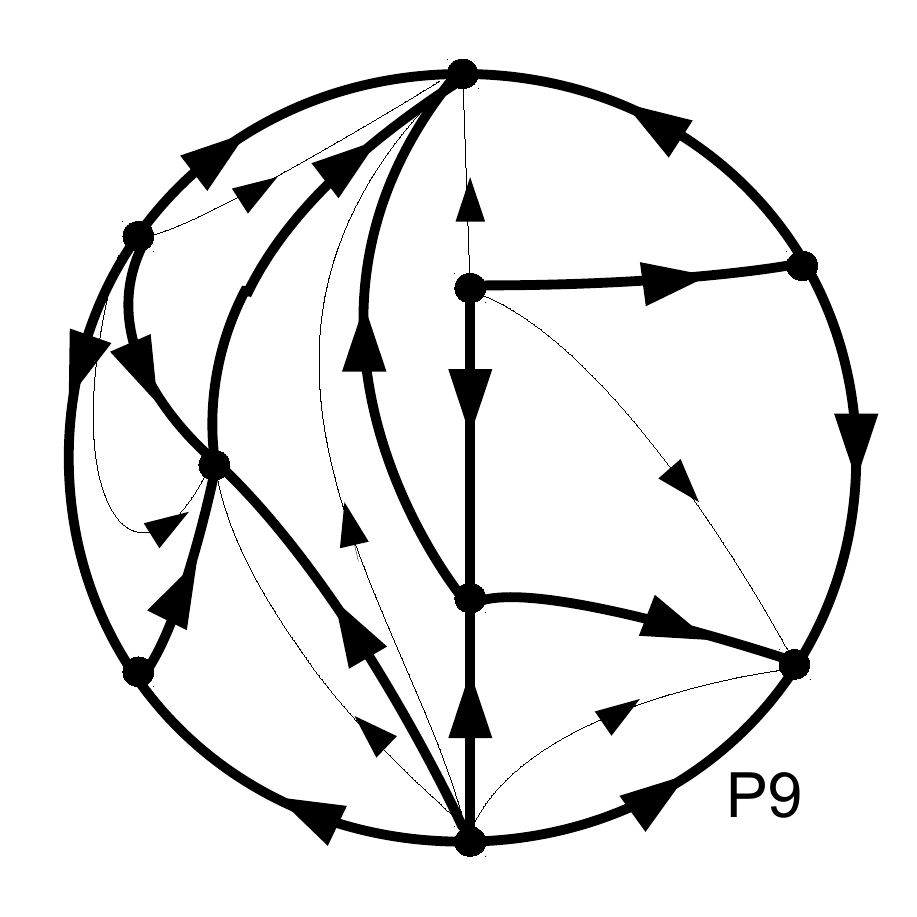}\end{minipage}
    \begin{minipage}[t]{2.7cm}\psfrag{b}{$b$}\centering\includegraphics[scale=.31]{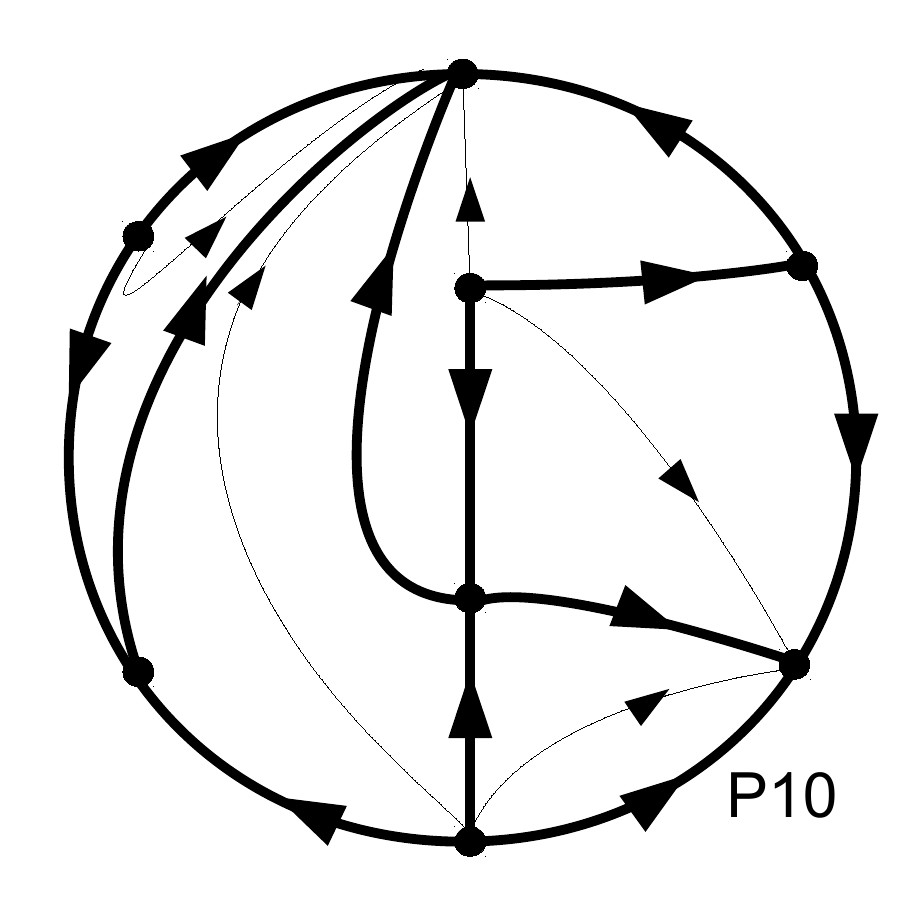}\end{minipage}
    \begin{minipage}[t]{2.7cm}\psfrag{c}{$c$}\centering\includegraphics[scale=.31]{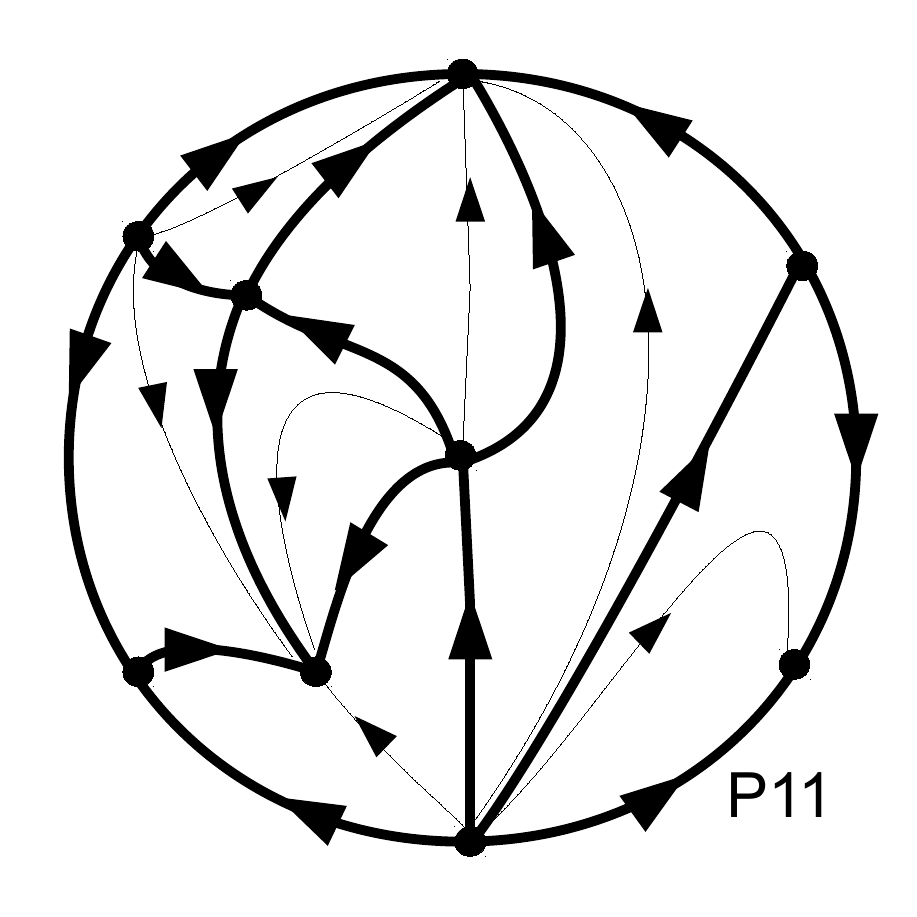}\end{minipage}
    \begin{minipage}[t]{2.7cm}\psfrag{d}{$d$}\centering\includegraphics[scale=.31]{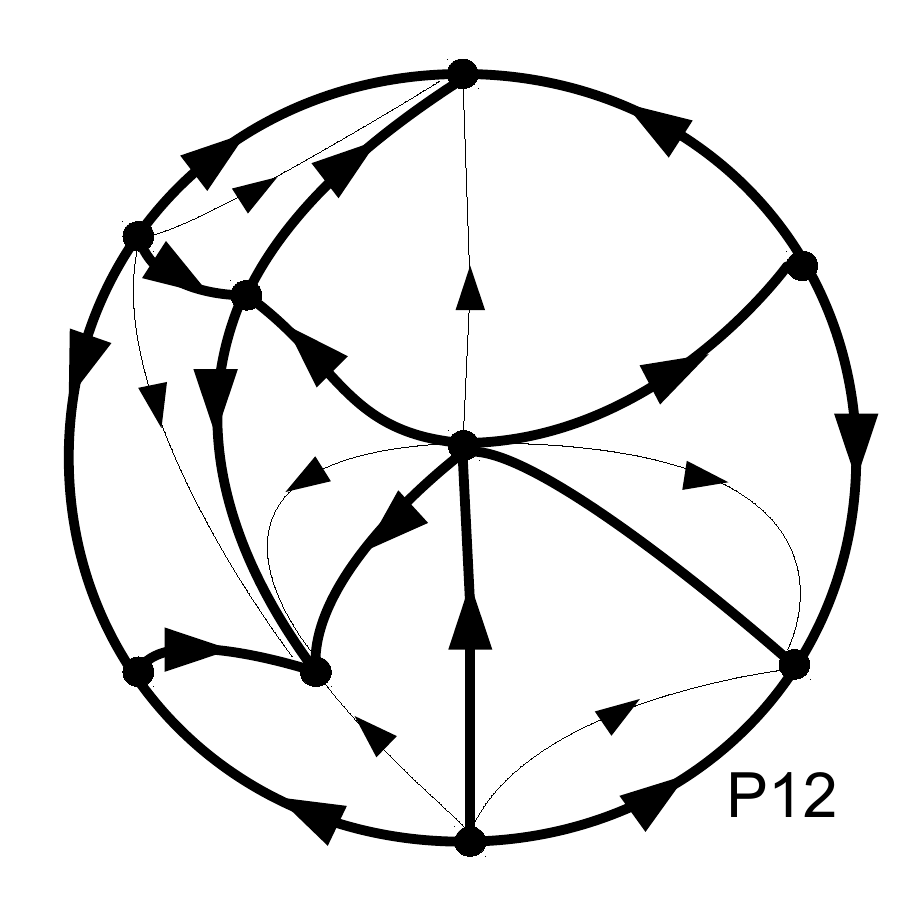}\end{minipage}

    \begin{minipage}[t]{2.7cm} \psfrag{a}{$a$}\centering\includegraphics[scale=.31]{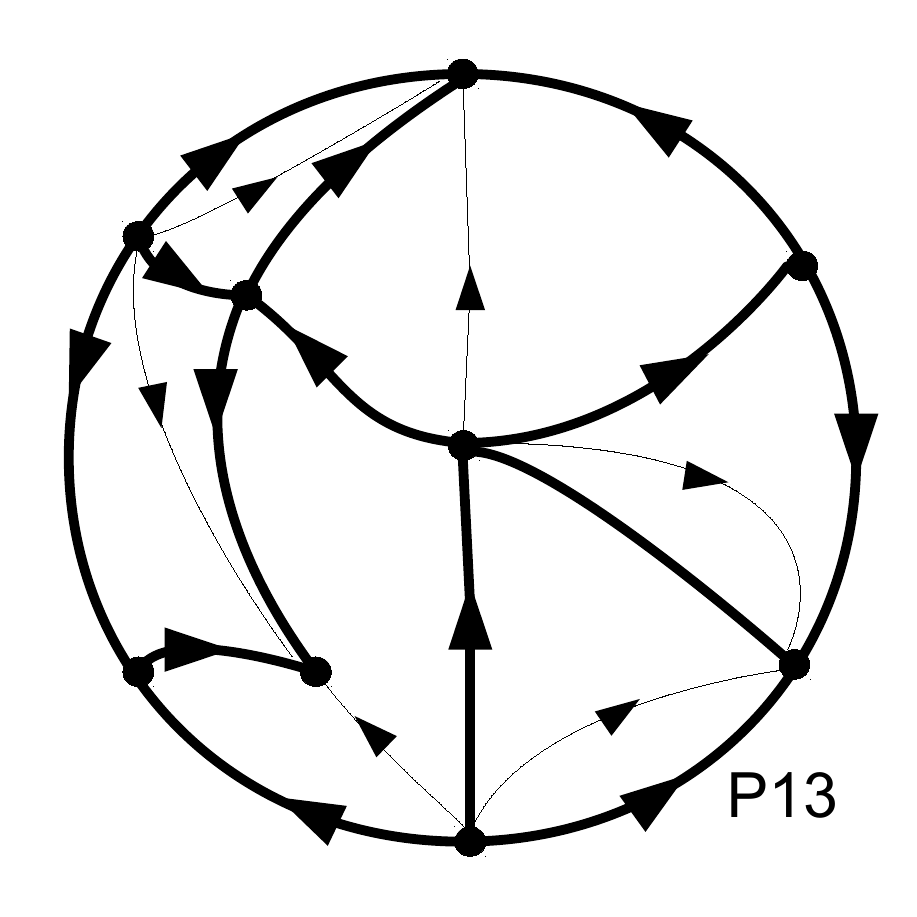}\end{minipage}
    \begin{minipage}[t]{2.7cm}\psfrag{b}{$b$}\centering\includegraphics[scale=.31]{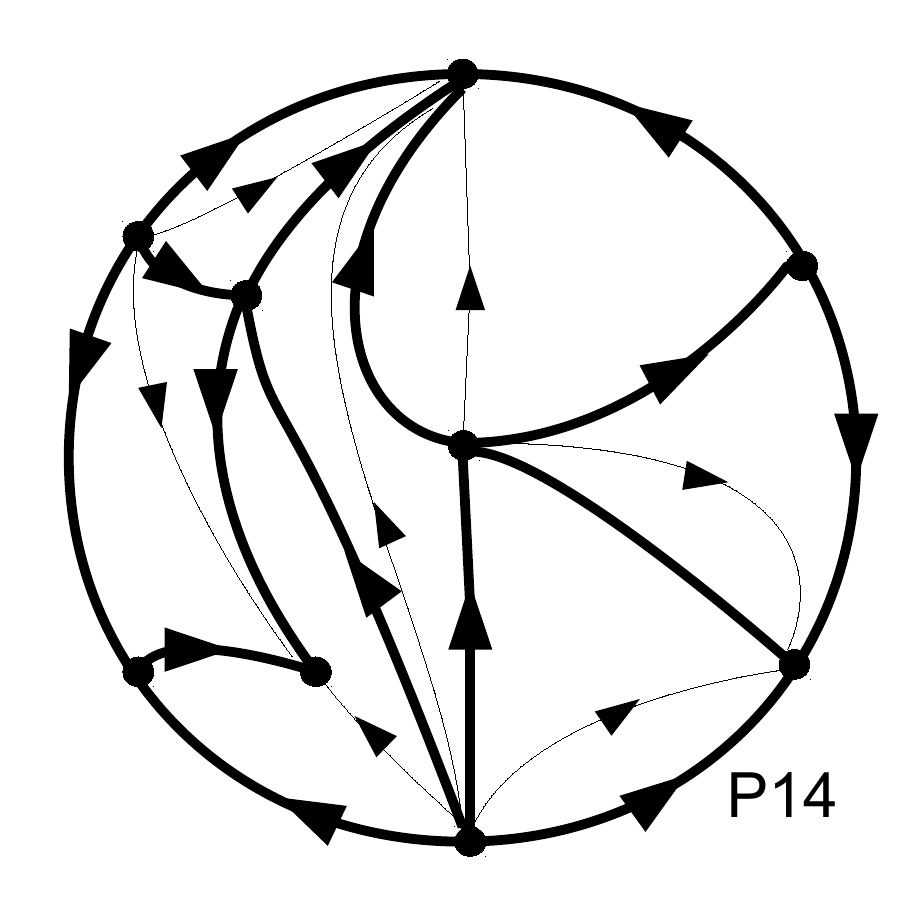}\end{minipage}
    \begin{minipage}[t]{2.7cm}\psfrag{c}{$c$}\centering\includegraphics[scale=.31]{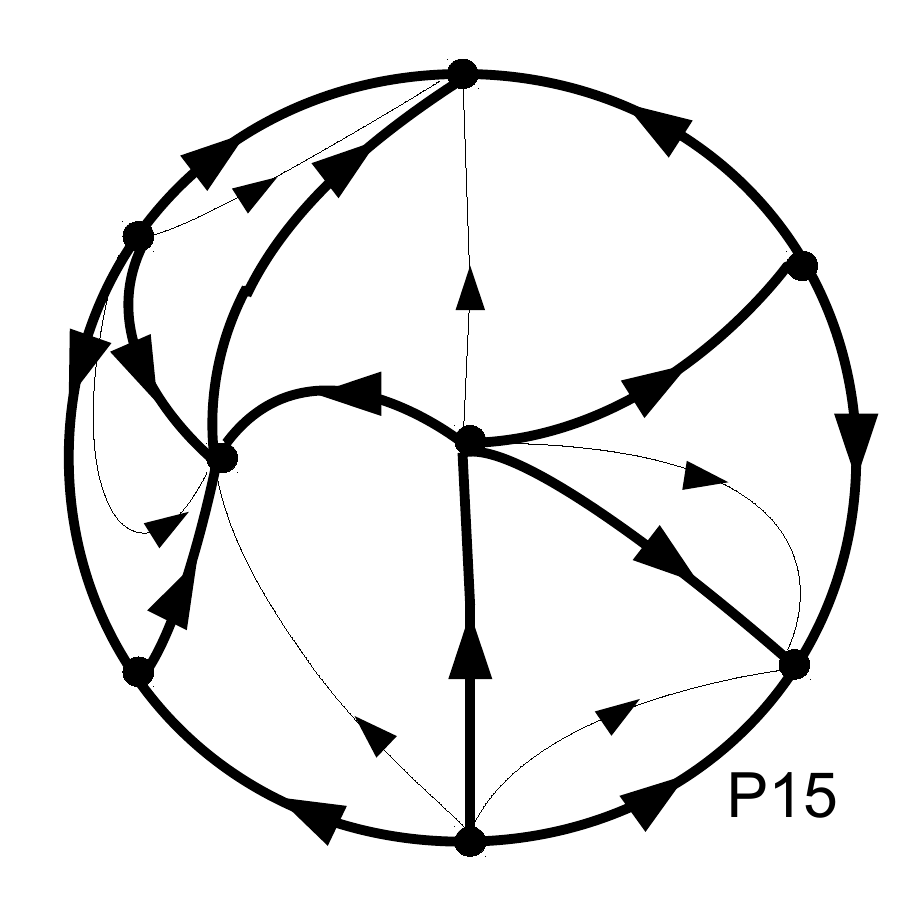}\end{minipage}
    \begin{minipage}[t]{2.7cm}\psfrag{d}{$d$}\centering\includegraphics[scale=.31]{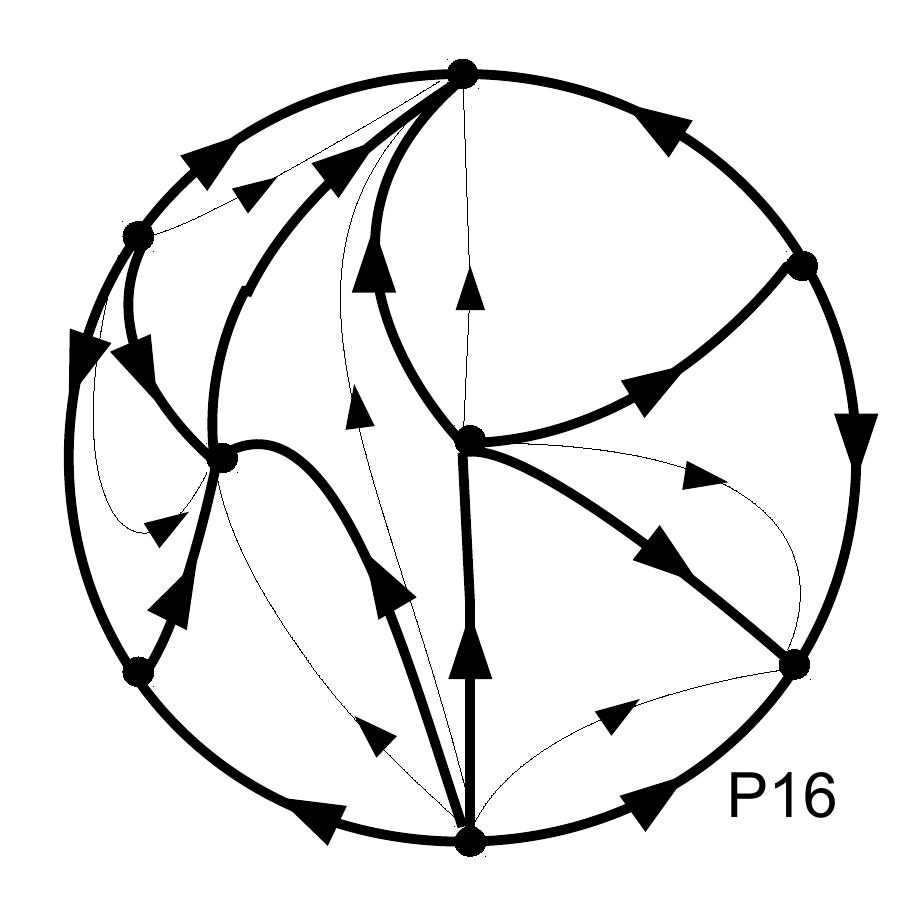}\end{minipage}

    \begin{minipage}[t]{2.7cm} \psfrag{a}{$a$}\centering\includegraphics[scale=.31]{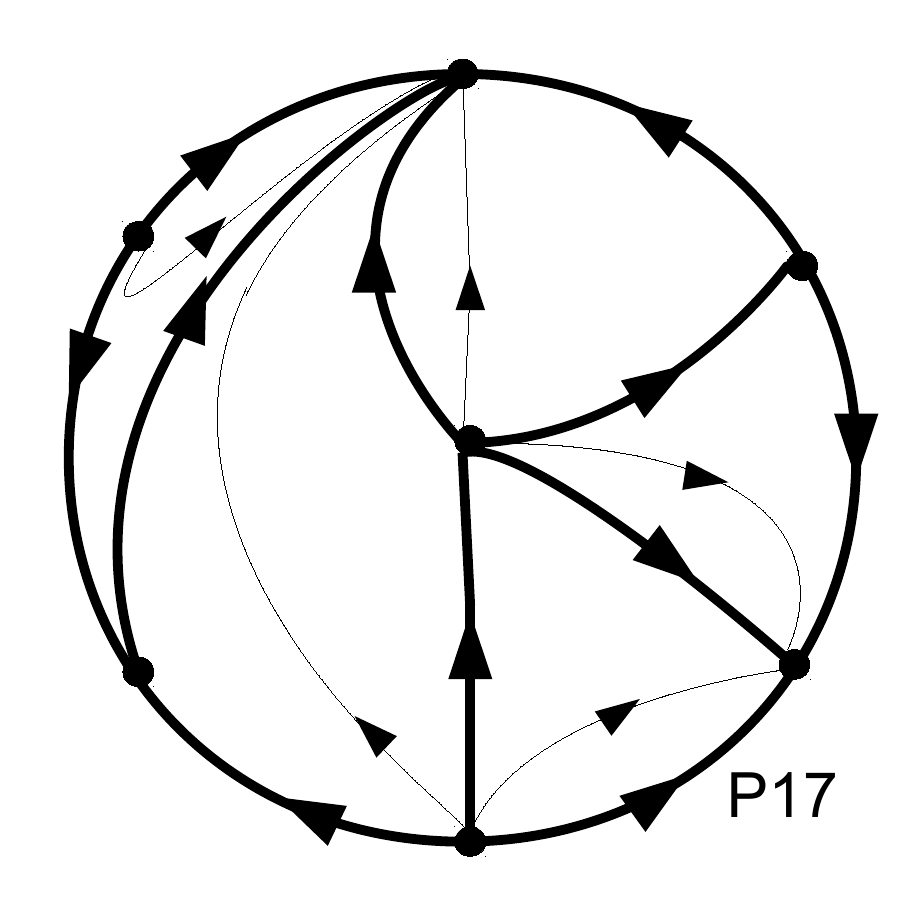}\end{minipage}
    \begin{minipage}[t]{2.7cm}\psfrag{b}{$b$}\centering\includegraphics[scale=.31]{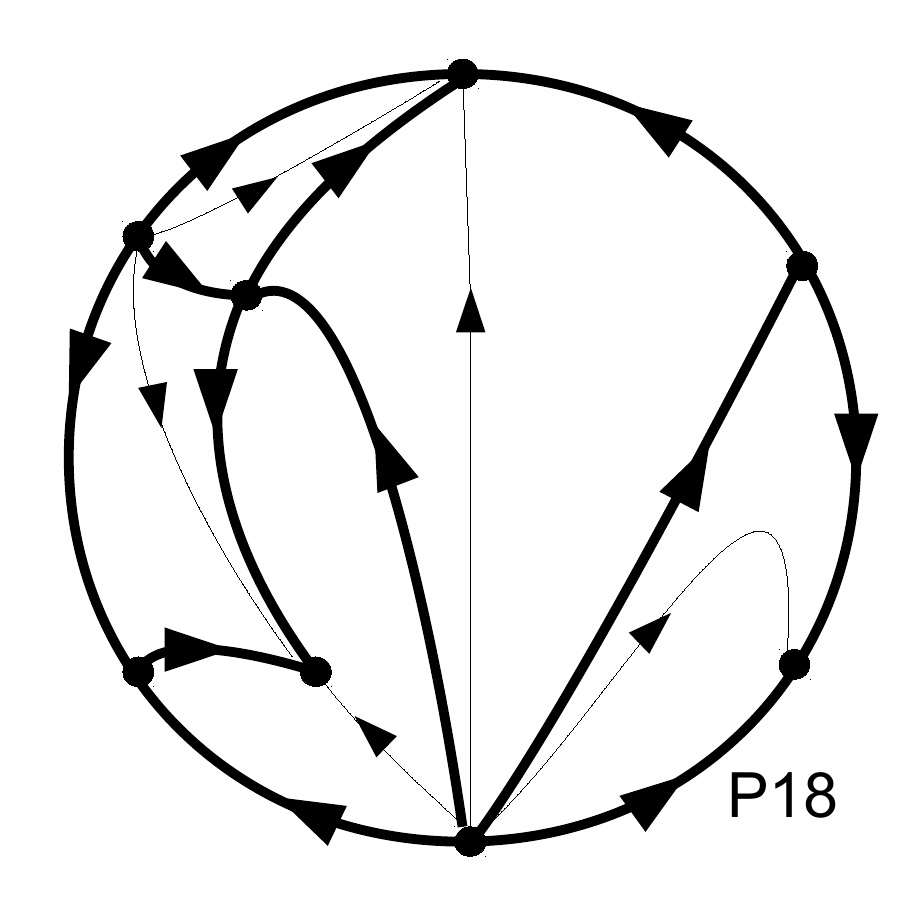}\end{minipage}
    \begin{minipage}[t]{2.7cm}\psfrag{c}{$c$}\centering\includegraphics[scale=.31]{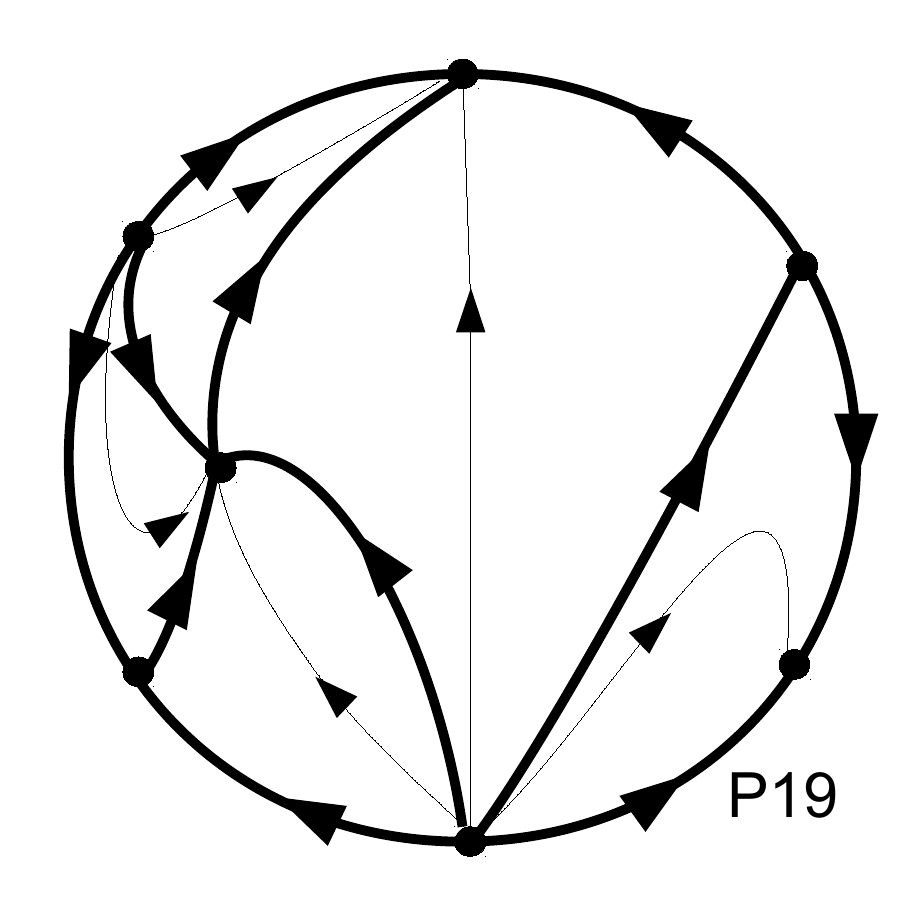}\end{minipage}
    \begin{minipage}[t]{2.7cm}\psfrag{d}{$d$}\centering\includegraphics[scale=.31]{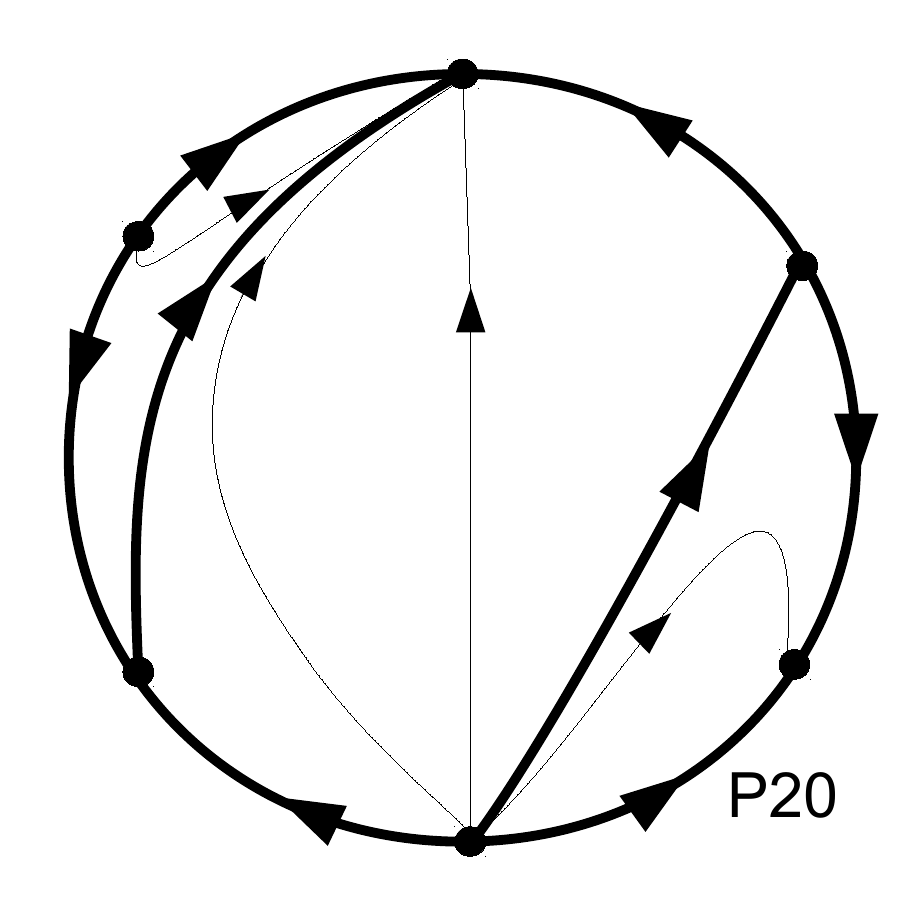}\end{minipage}

    \begin{minipage}[t]{2.7cm} \psfrag{a}{$a$}\centering\includegraphics[scale=.31]{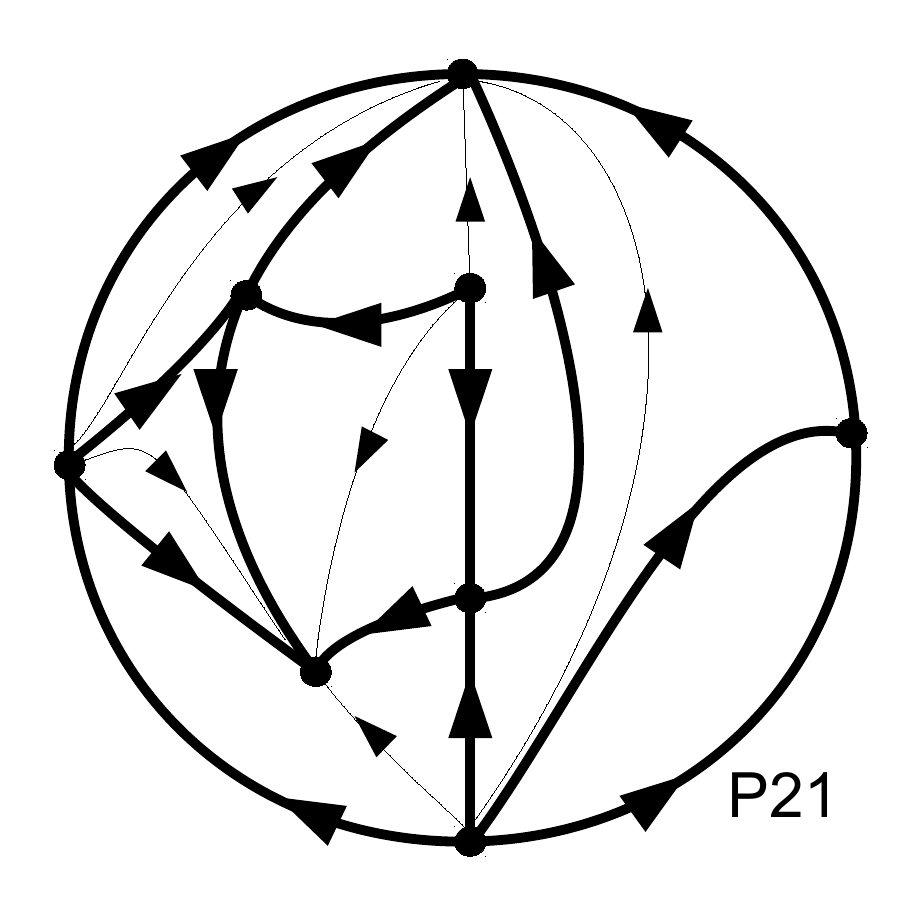}\end{minipage}
    \begin{minipage}[t]{2.7cm}\psfrag{b}{$b$}\centering\includegraphics[scale=.31]{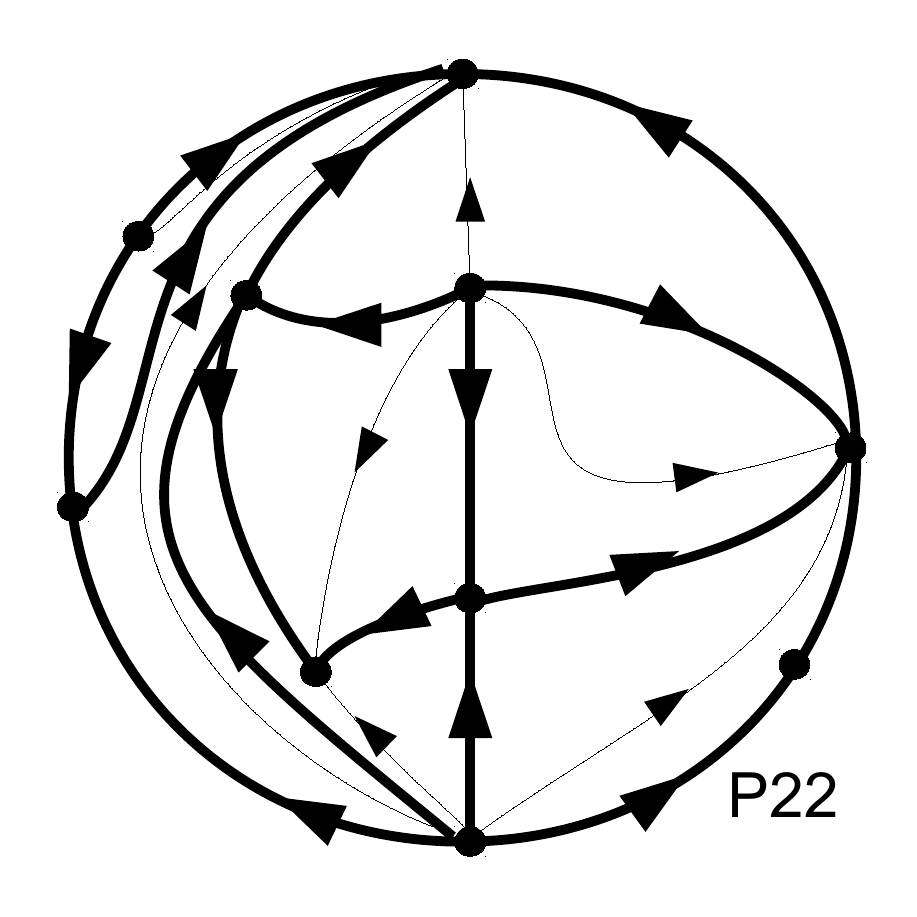}\end{minipage}
    \begin{minipage}[t]{2.7cm}\psfrag{c}{$c$}\centering\includegraphics[scale=.31]{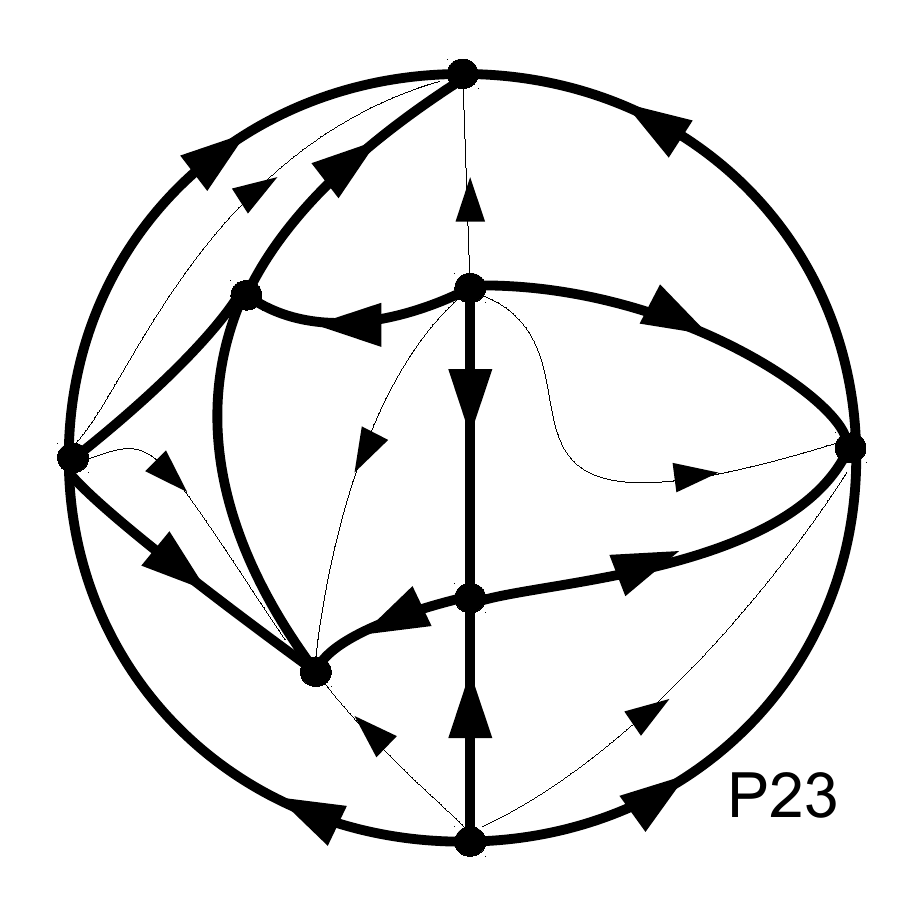}\end{minipage}
    \begin{minipage}[t]{2.7cm}\psfrag{d}{$d$}\centering\includegraphics[scale=.31]{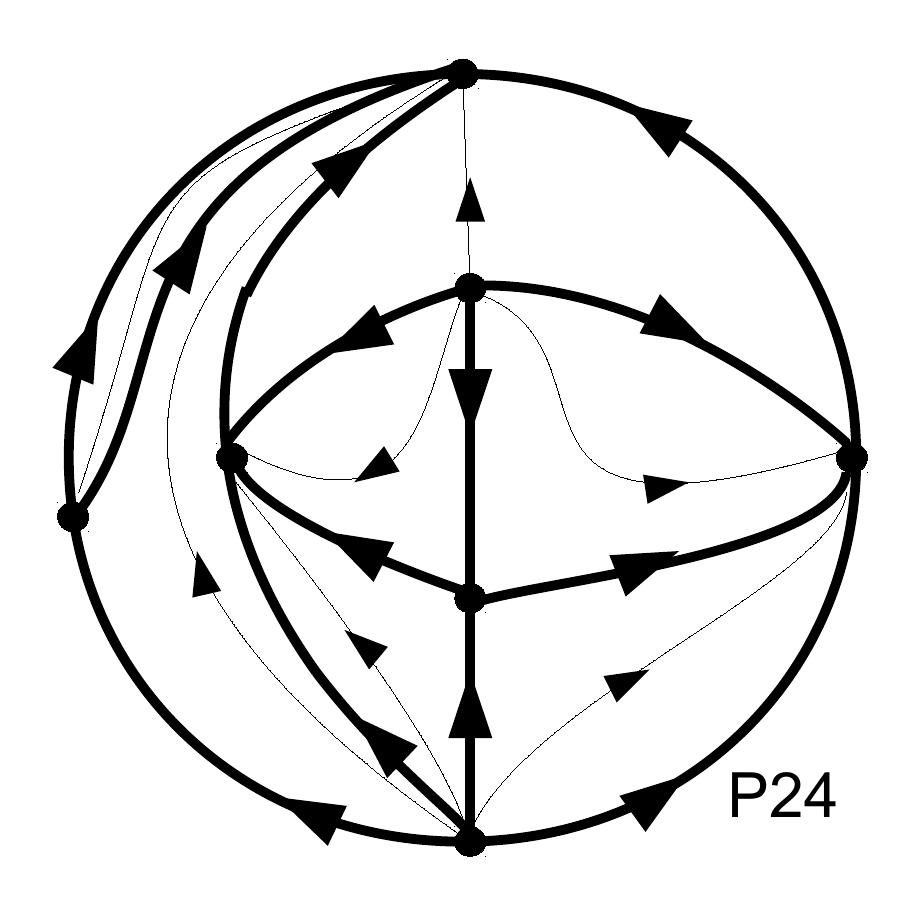}\end{minipage}

    \begin{minipage}[t]{2.7cm} \psfrag{a}{$a$}\centering\includegraphics[scale=.31]{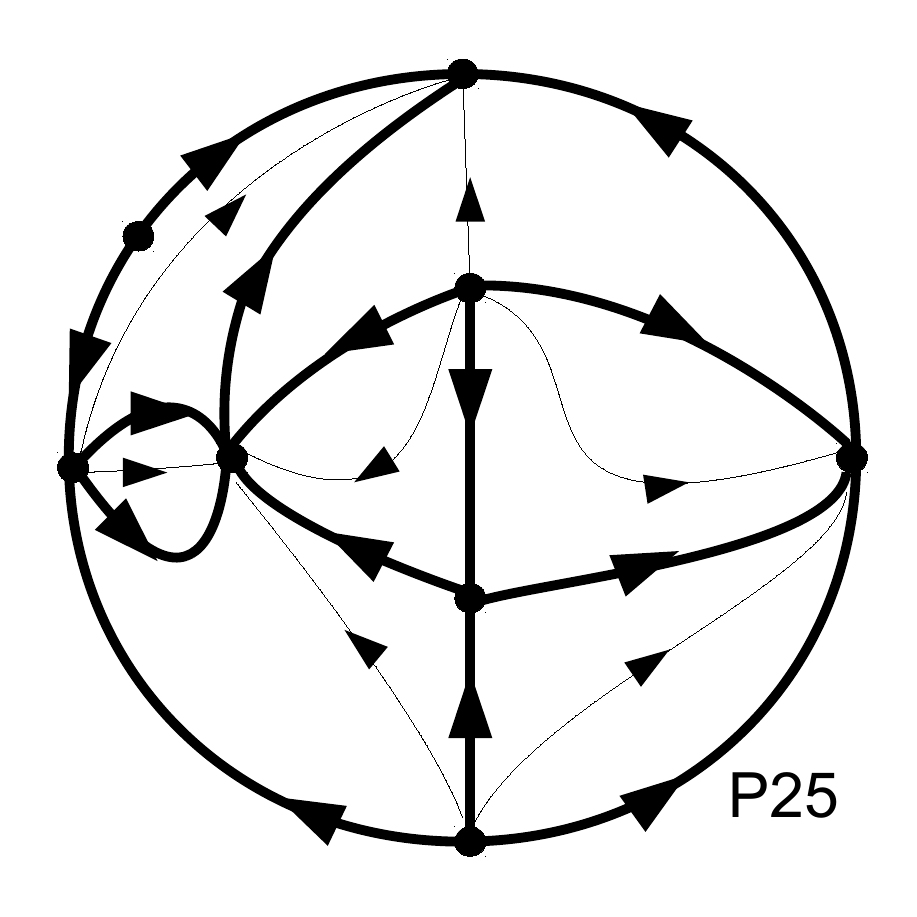}\end{minipage}
    \begin{minipage}[t]{2.7cm}\psfrag{b}{$b$}\centering\includegraphics[scale=.31]{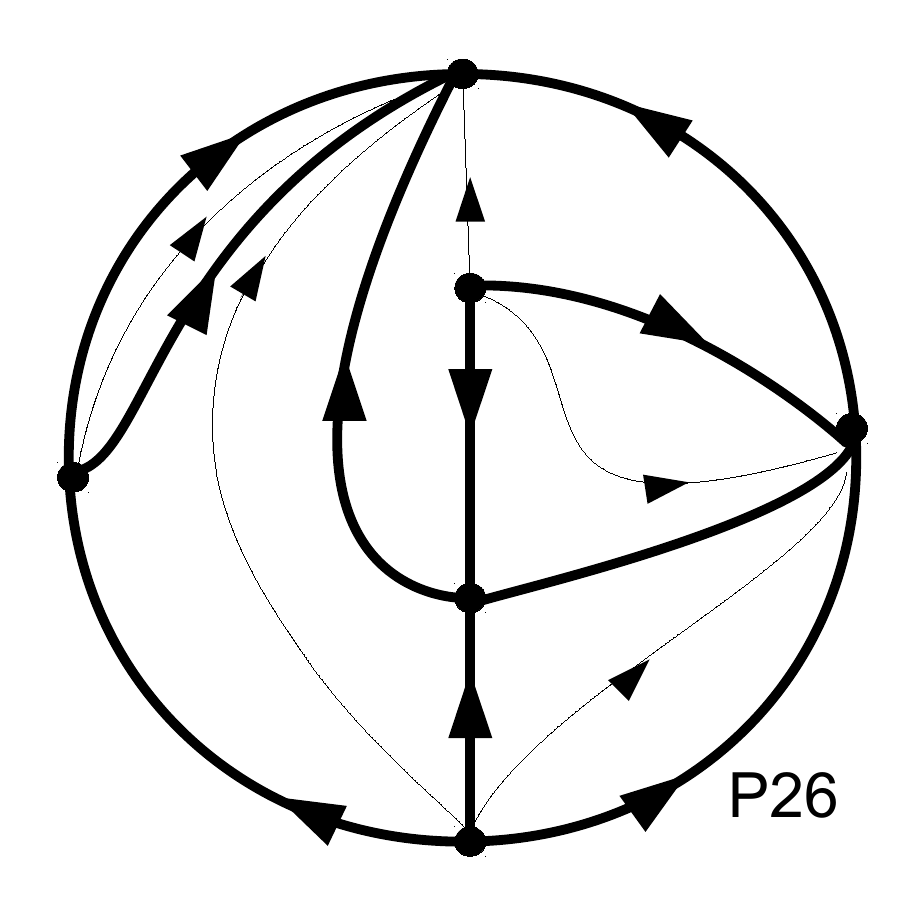}\end{minipage}
    \begin{minipage}[t]{2.7cm}\psfrag{c}{$c$}\centering\includegraphics[scale=.31]{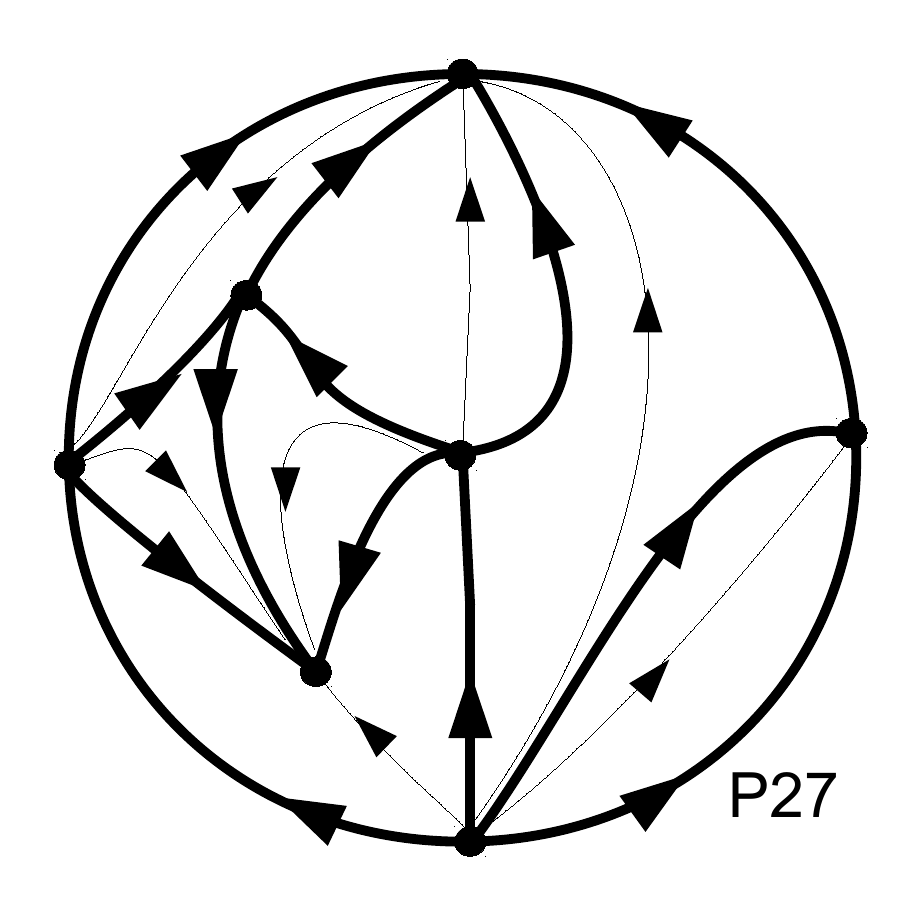}\end{minipage}
    \begin{minipage}[t]{2.7cm}\psfrag{d}{$d$}\centering\includegraphics[scale=.31]{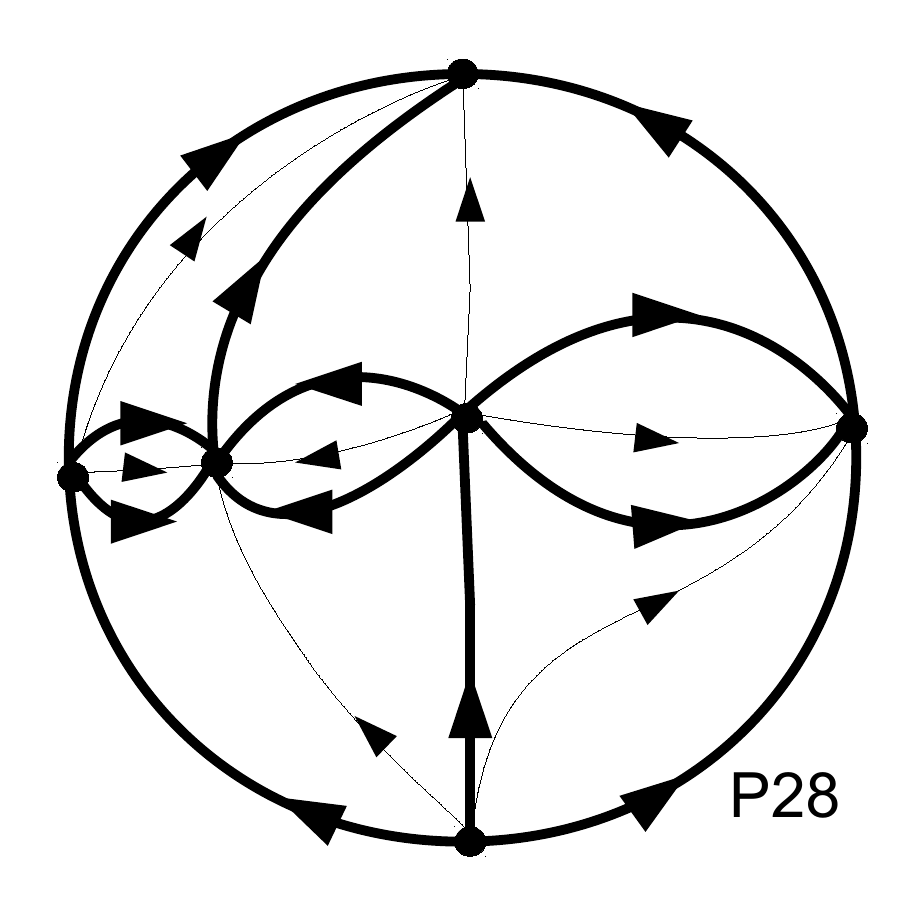}\end{minipage}
    \caption{\small Phase portraits of systems (1)
        the Poicar\'e disk.}\label{figura1-2}
\end{figure}
\newpage

\begin{figure}
    \begin{minipage}[t]{2.7cm} \psfrag{a}{$a$}\centering\includegraphics[scale=.31]{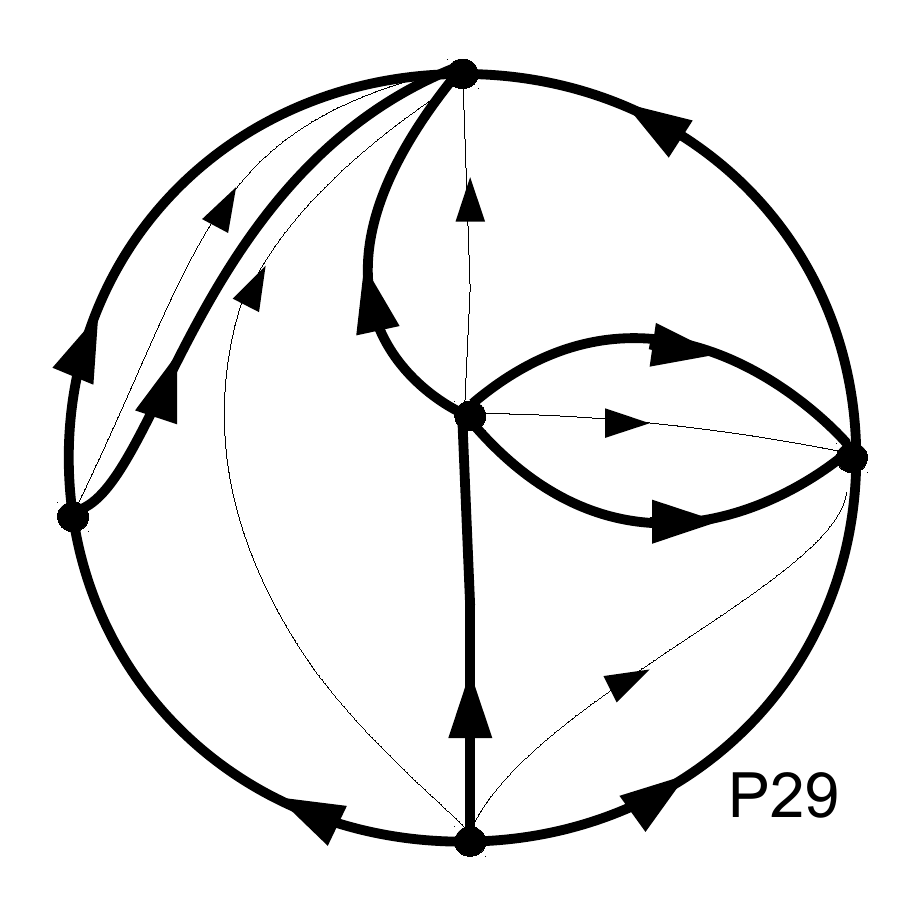}\end{minipage}
    \begin{minipage}[t]{2.7cm}\psfrag{b}{$b$}\centering\includegraphics[scale=.31]{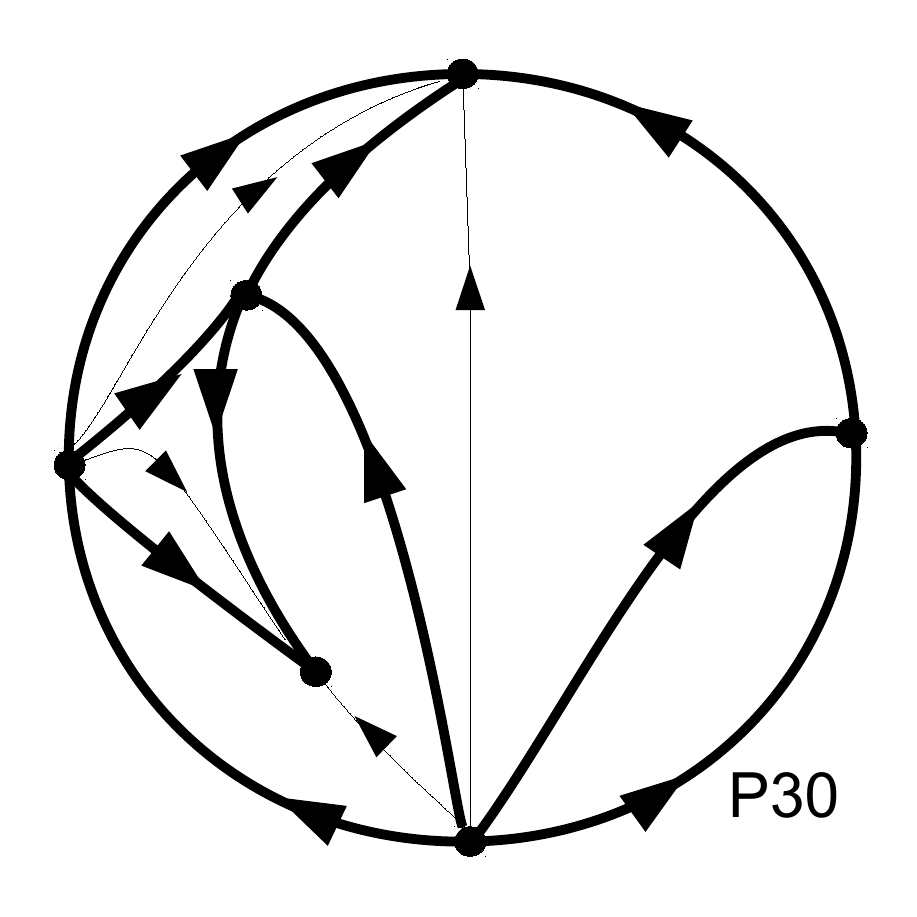}\end{minipage}
    \begin{minipage}[t]{2.7cm}\psfrag{c}{$c$}\centering\includegraphics[scale=.31]{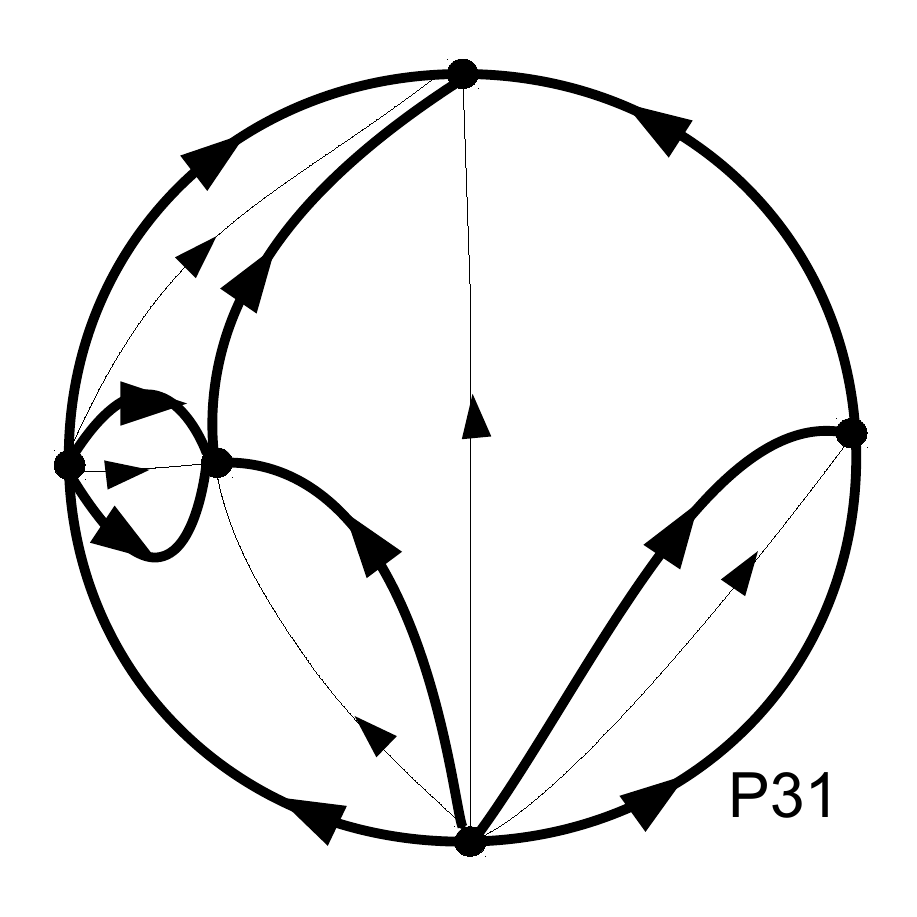}\end{minipage}
    \begin{minipage}[t]{2.7cm}\psfrag{d}{$d$}\centering\includegraphics[scale=.31]{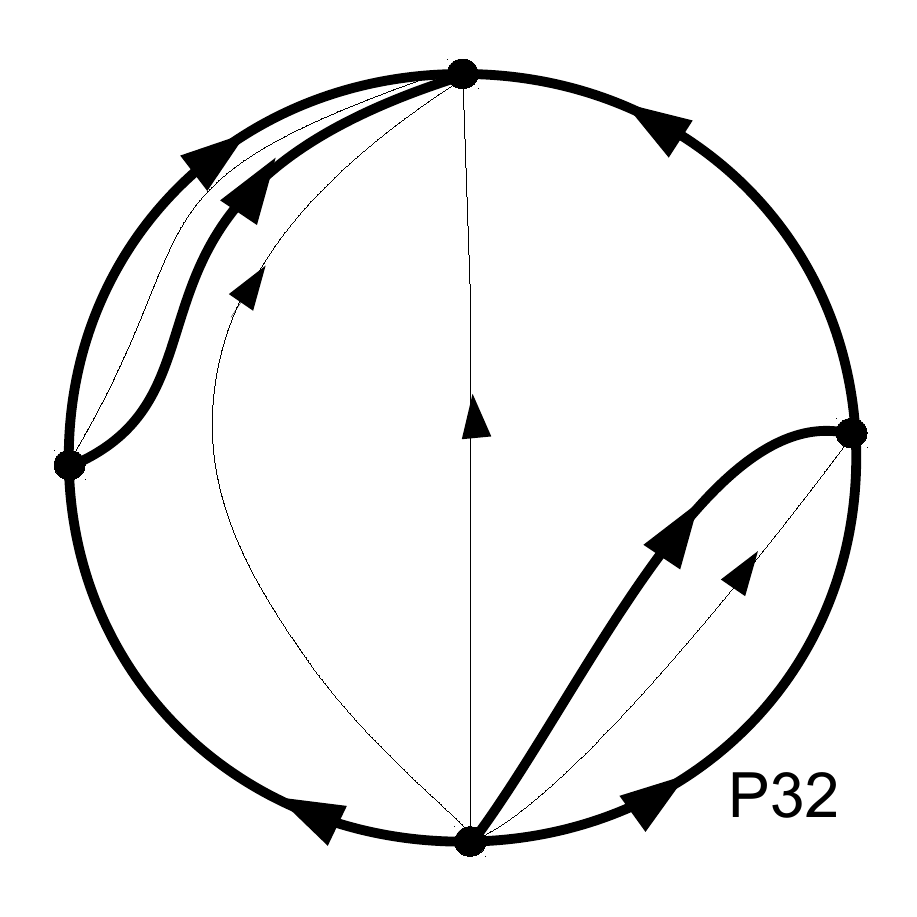}\end{minipage}

    \begin{minipage}[t]{2.7cm} \psfrag{a}{$a$}\centering\includegraphics[scale=.31]{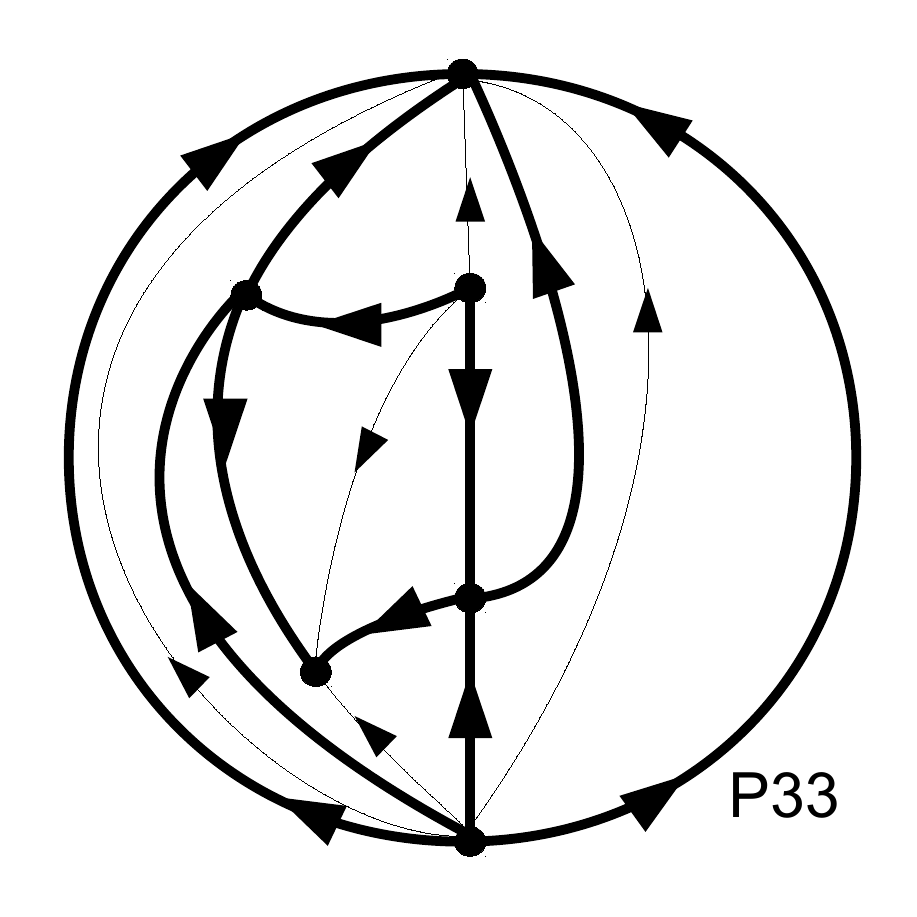}\end{minipage}
    \begin{minipage}[t]{2.7cm}\psfrag{b}{$b$}\centering\includegraphics[scale=.31]{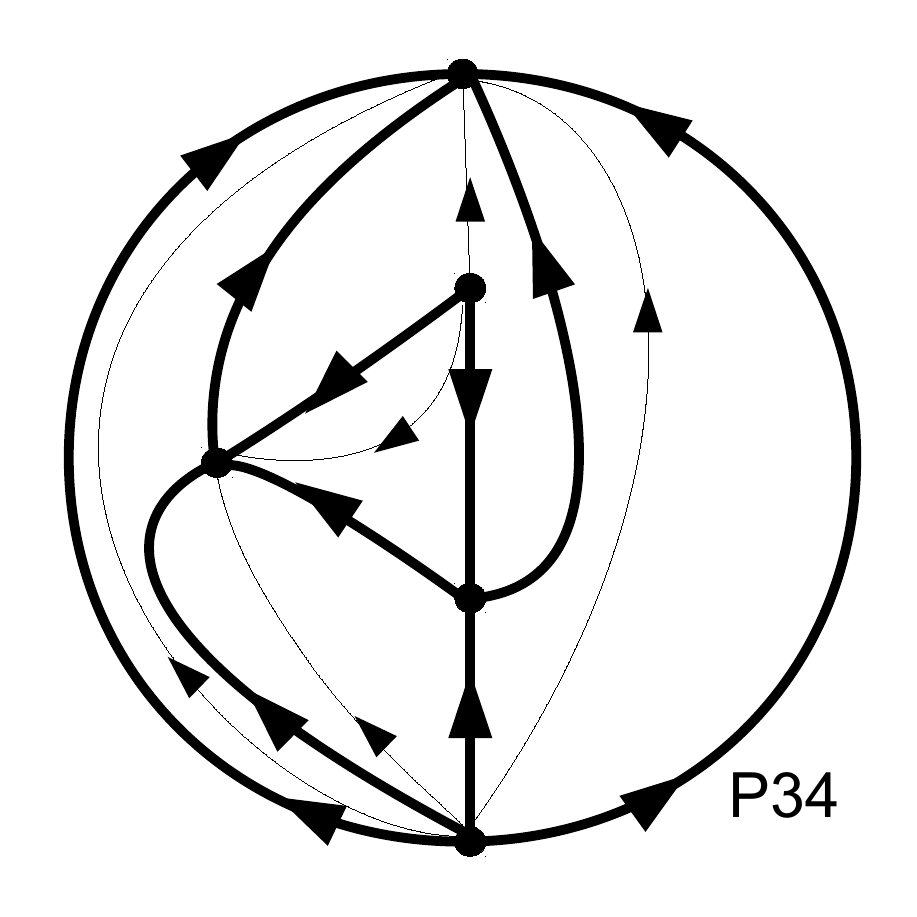}\end{minipage}
    \begin{minipage}[t]{2.7cm}\psfrag{c}{$c$}\centering\includegraphics[scale=.31]{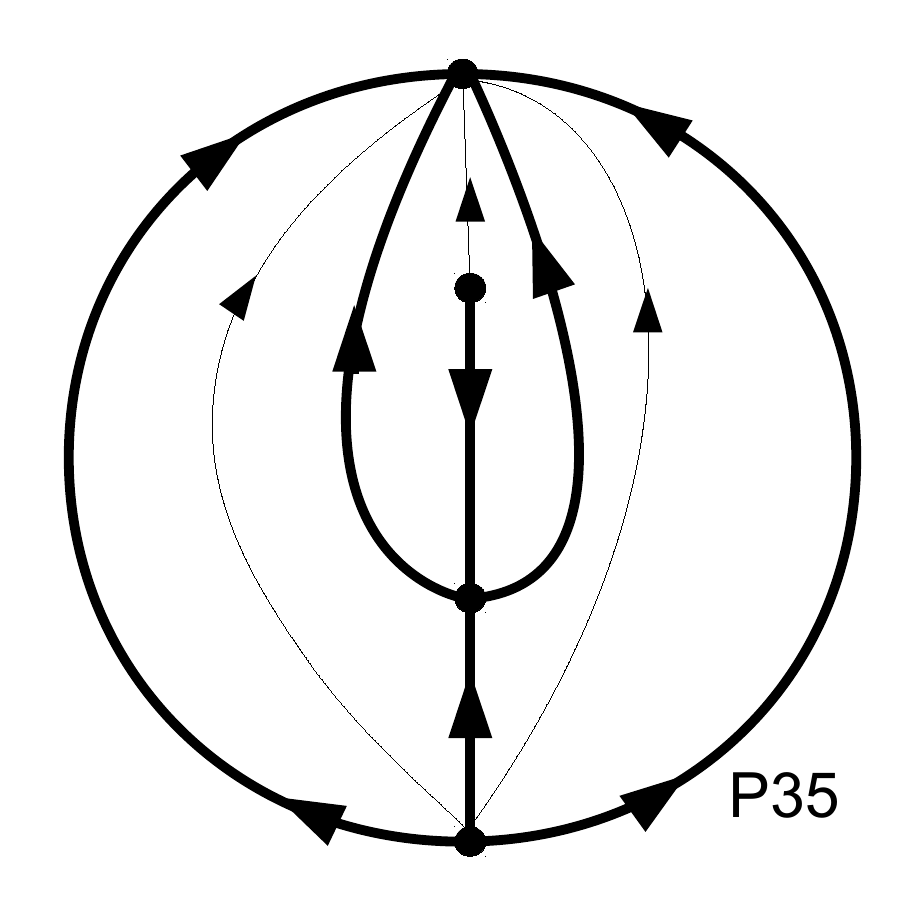}\end{minipage}
    \begin{minipage}[t]{2.7cm}\psfrag{d}{$d$}\centering\includegraphics[scale=.31]{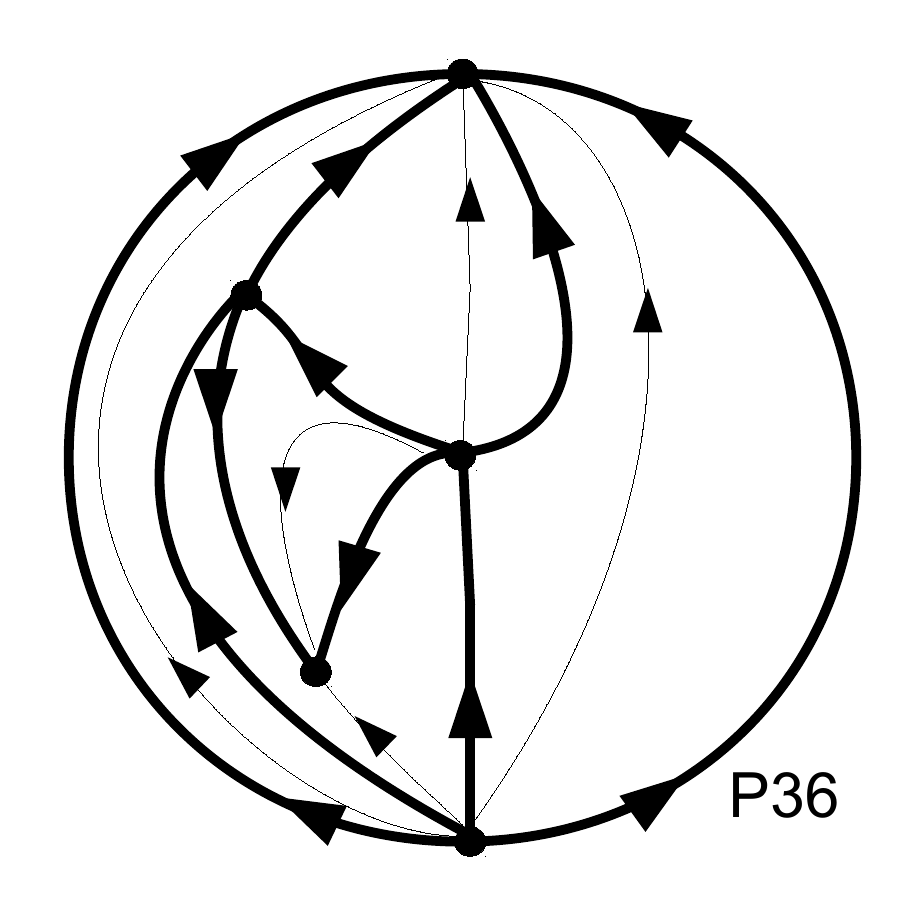}\end{minipage}

    \begin{minipage}[t]{2.7cm} \psfrag{a}{$a$}\centering\includegraphics[scale=.31]{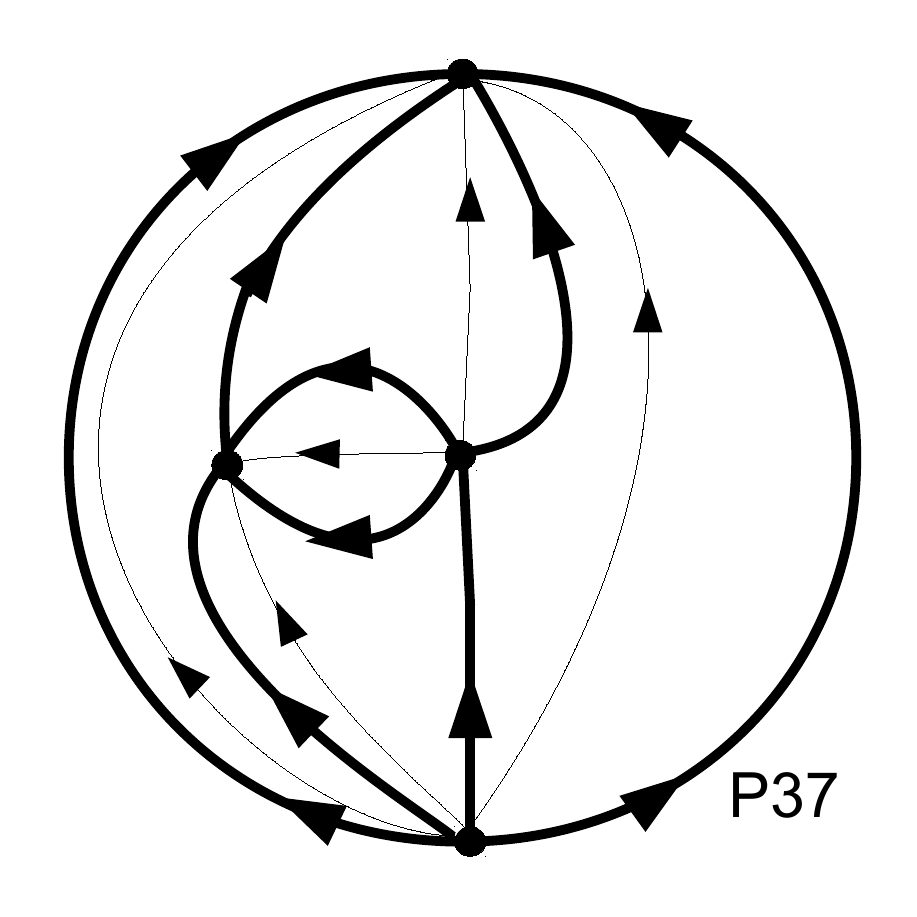}\end{minipage}
    \begin{minipage}[t]{2.7cm}\psfrag{b}{$b$}\centering\includegraphics[scale=.31]{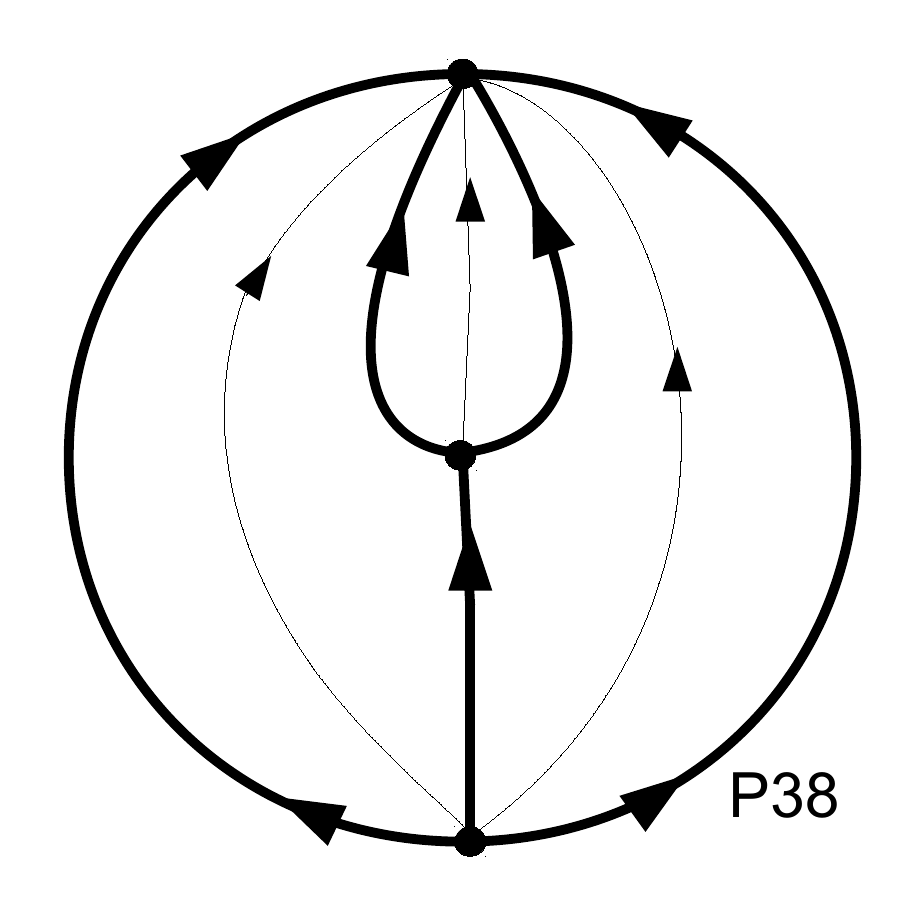}\end{minipage}
    \begin{minipage}[t]{2.7cm}\psfrag{c}{$c$}\centering\includegraphics[scale=.31]{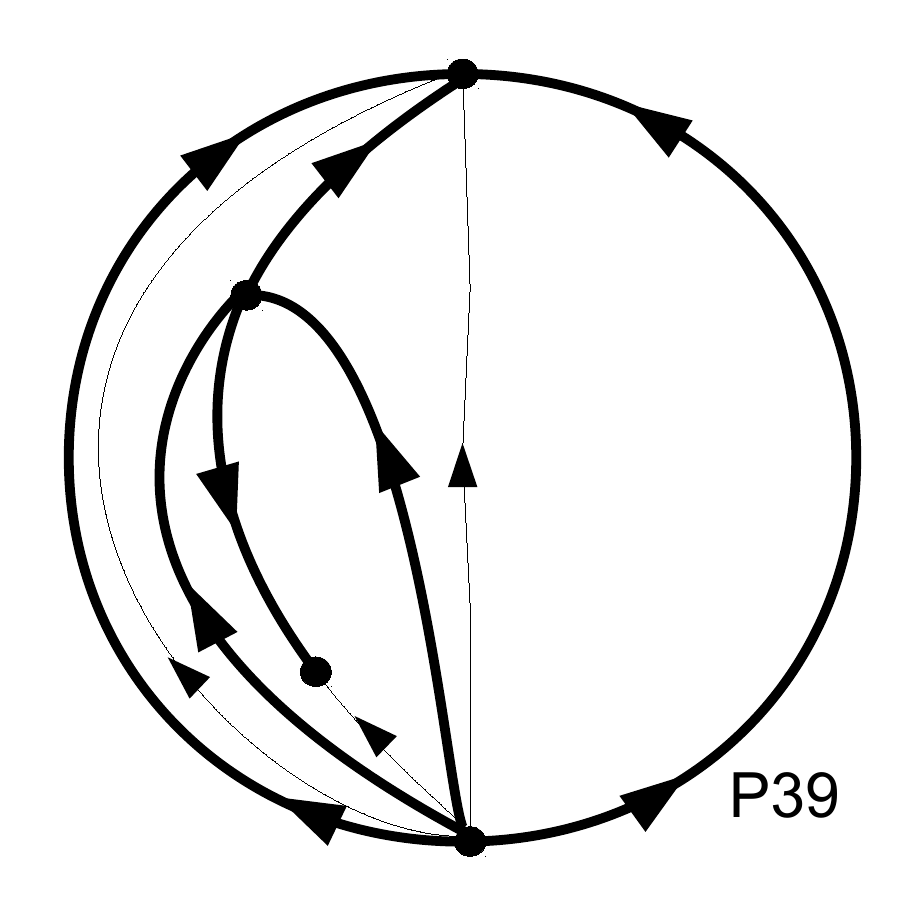}\end{minipage}
    \begin{minipage}[t]{2.7cm}\psfrag{d}{$d$}\centering\includegraphics[scale=.31]{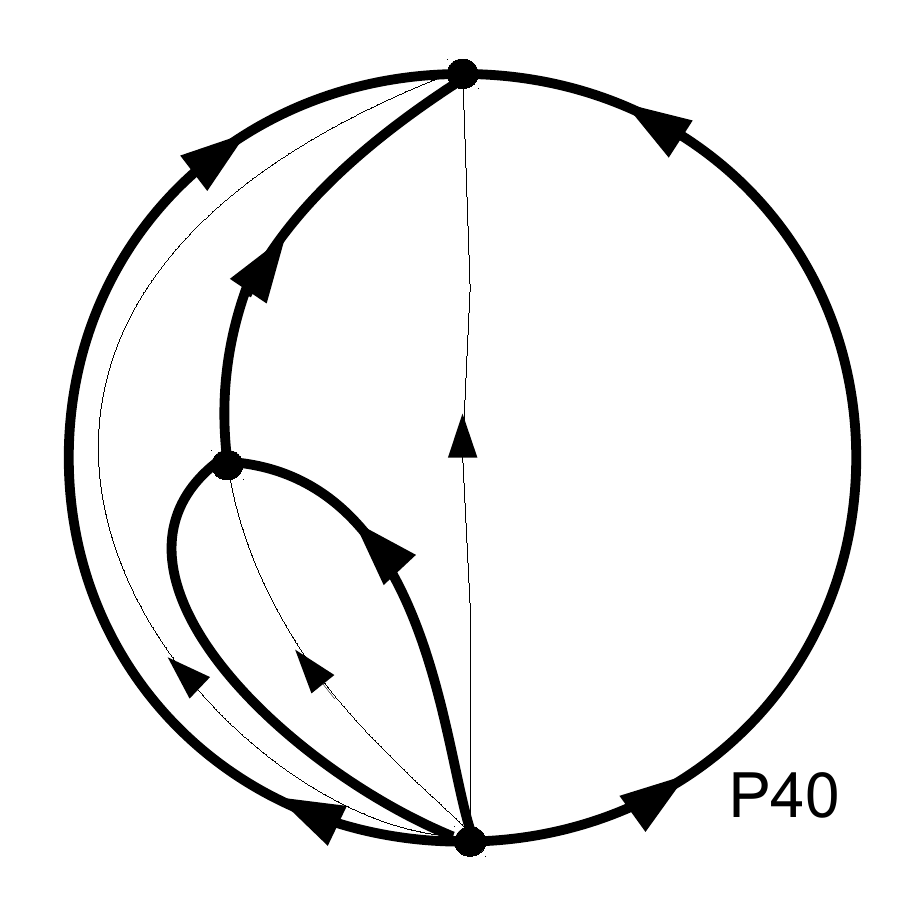}\end{minipage}

    \begin{minipage}[t]{2.7cm} \psfrag{a}{$a$}\centering\includegraphics[scale=.31]{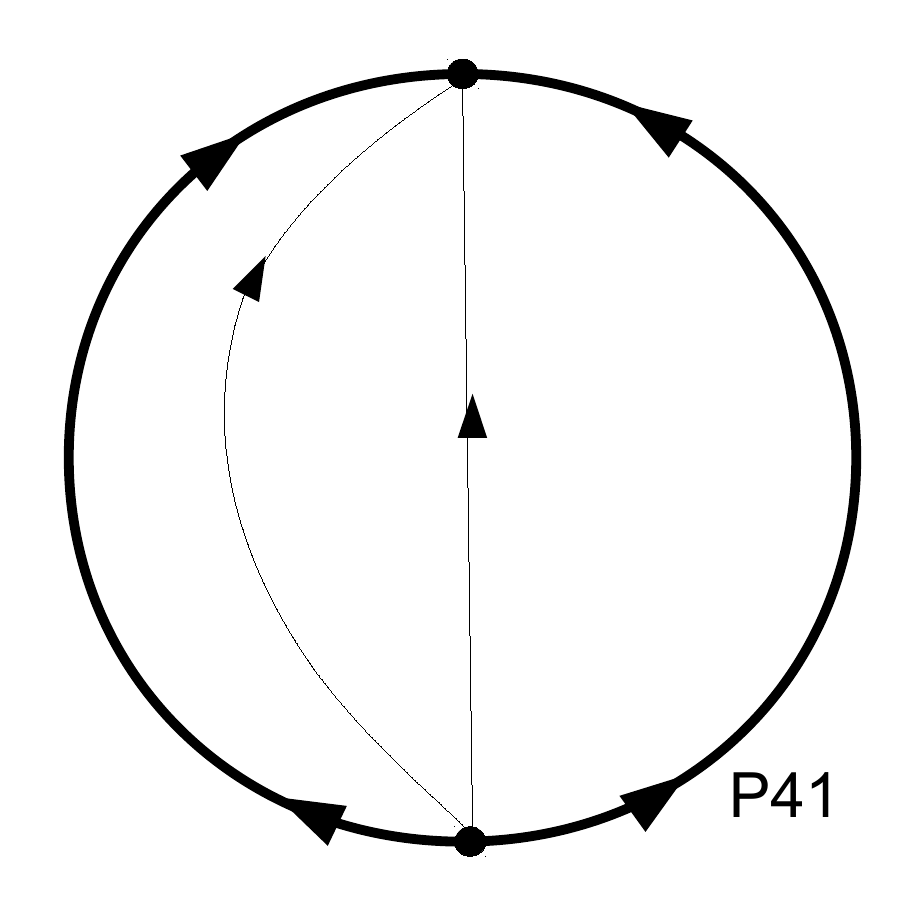}\end{minipage}
    \begin{minipage}[t]{2.7cm}\psfrag{b}{$b$}\centering\includegraphics[scale=.31]{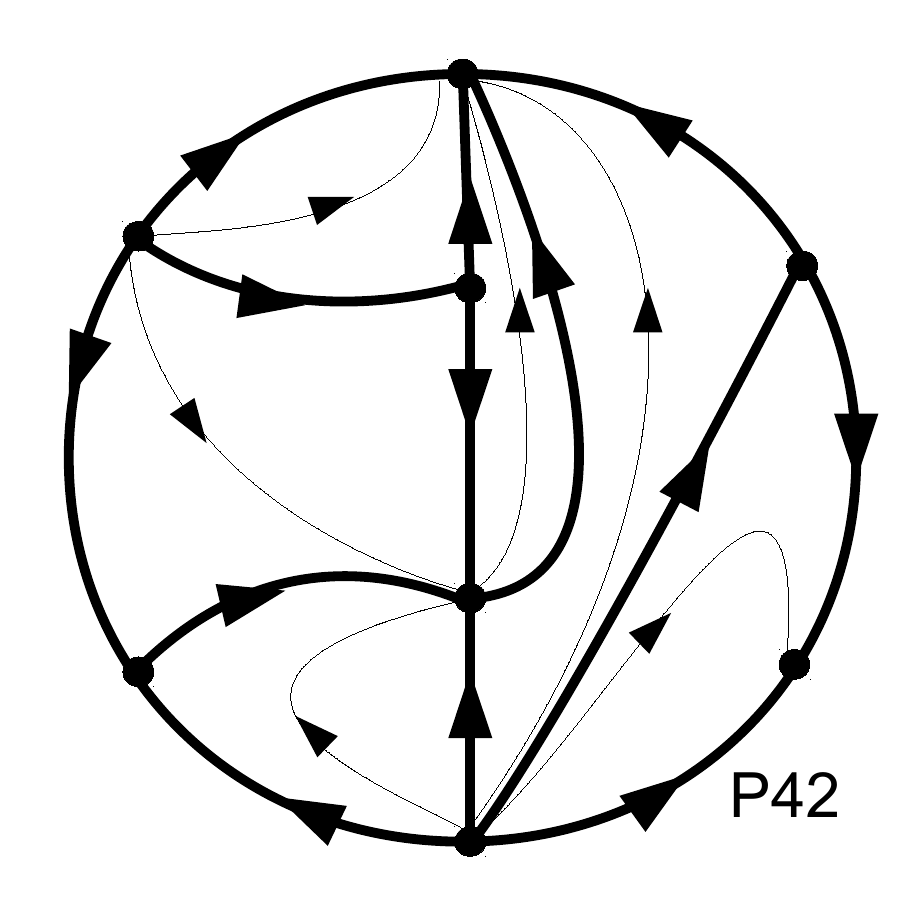}\end{minipage}
    \begin{minipage}[t]{2.7cm}\psfrag{c}{$c$}\centering\includegraphics[scale=.31]{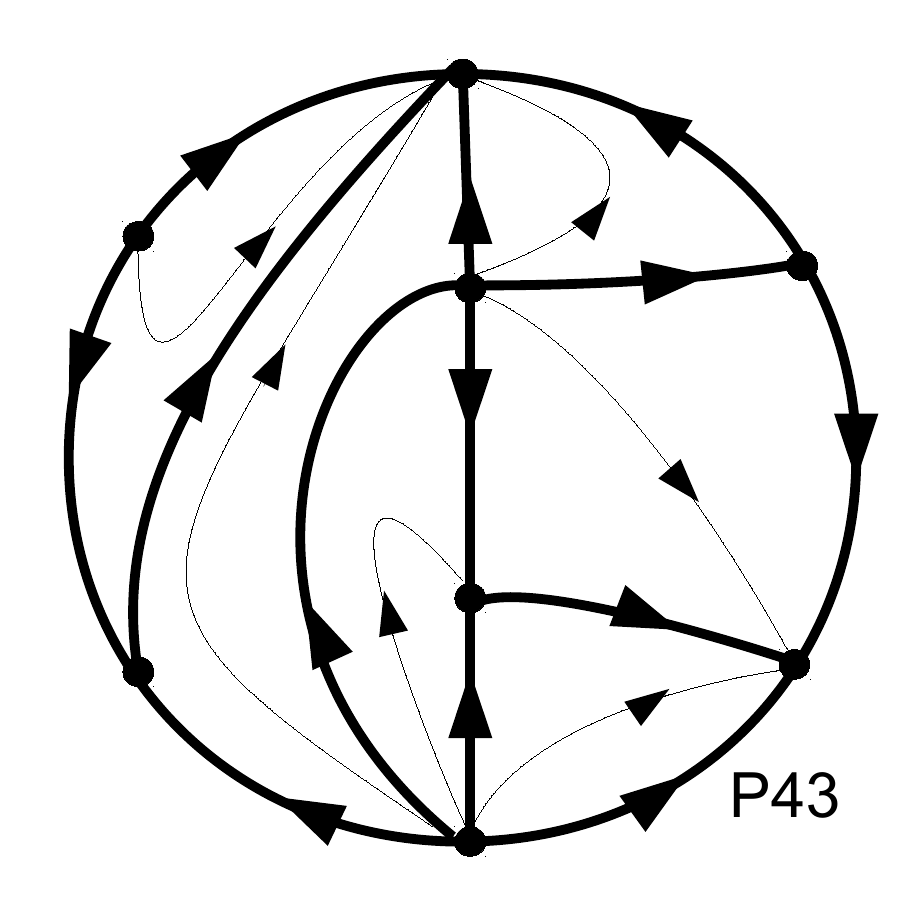}\end{minipage}
    \begin{minipage}[t]{2.7cm}\psfrag{d}{$d$}\centering\includegraphics[scale=.31]{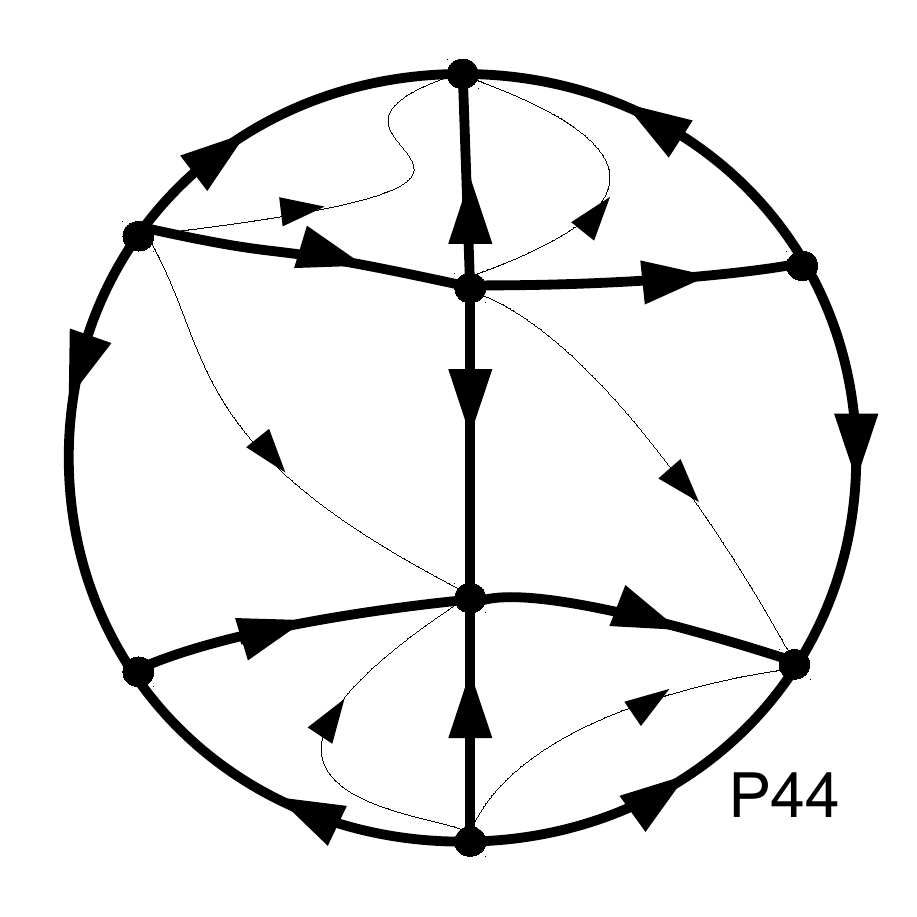}\end{minipage}

    \begin{minipage}[t]{2.7cm} \psfrag{a}{$a$}\centering\includegraphics[scale=.31]{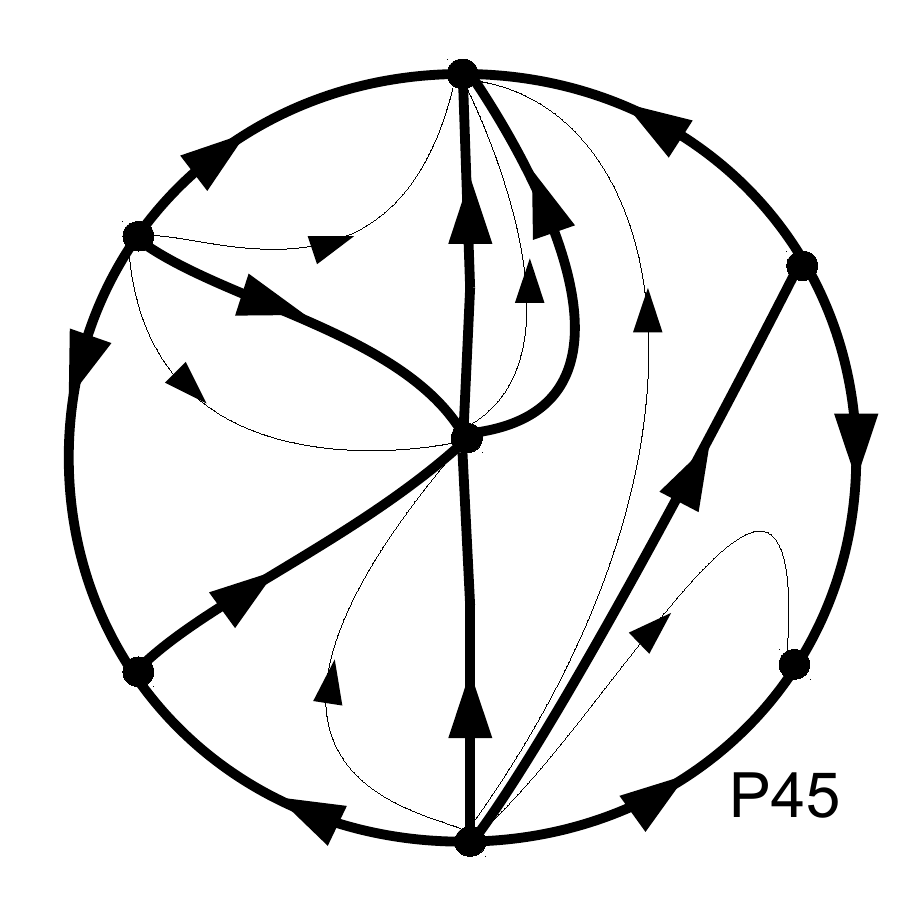}\end{minipage}
    \begin{minipage}[t]{2.7cm}\psfrag{b}{$b$}\centering\includegraphics[scale=.31]{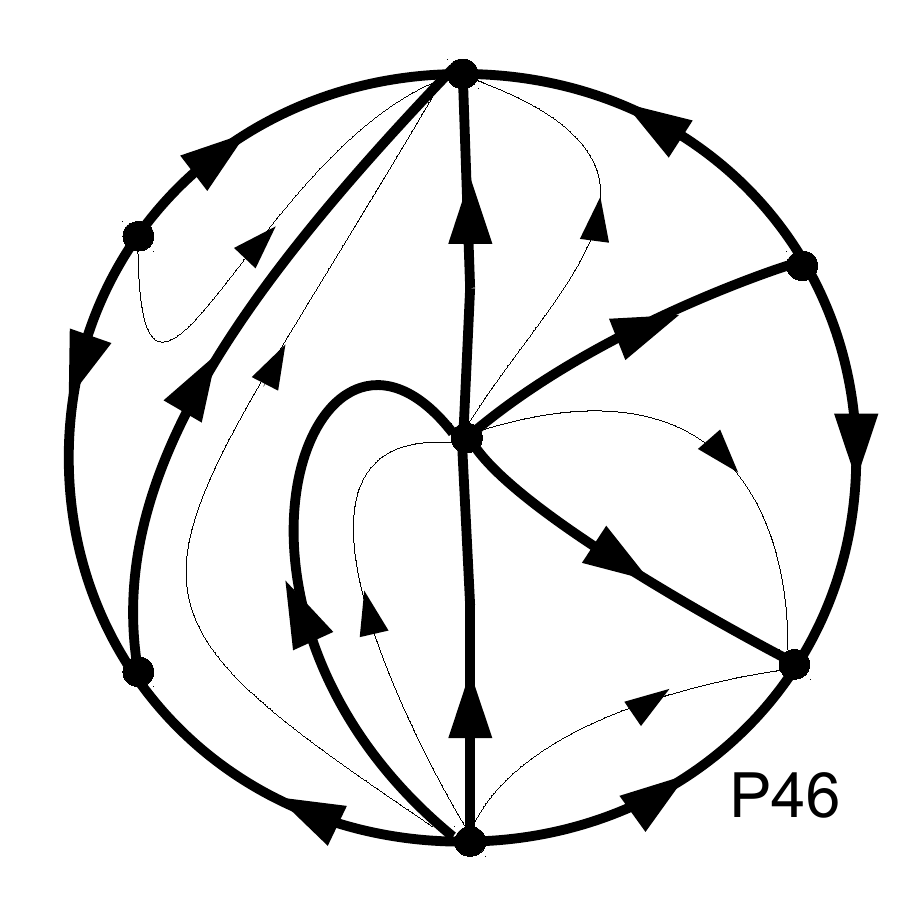}\end{minipage}
    \begin{minipage}[t]{2.7cm}\psfrag{c}{$c$}\centering\includegraphics[scale=.31]{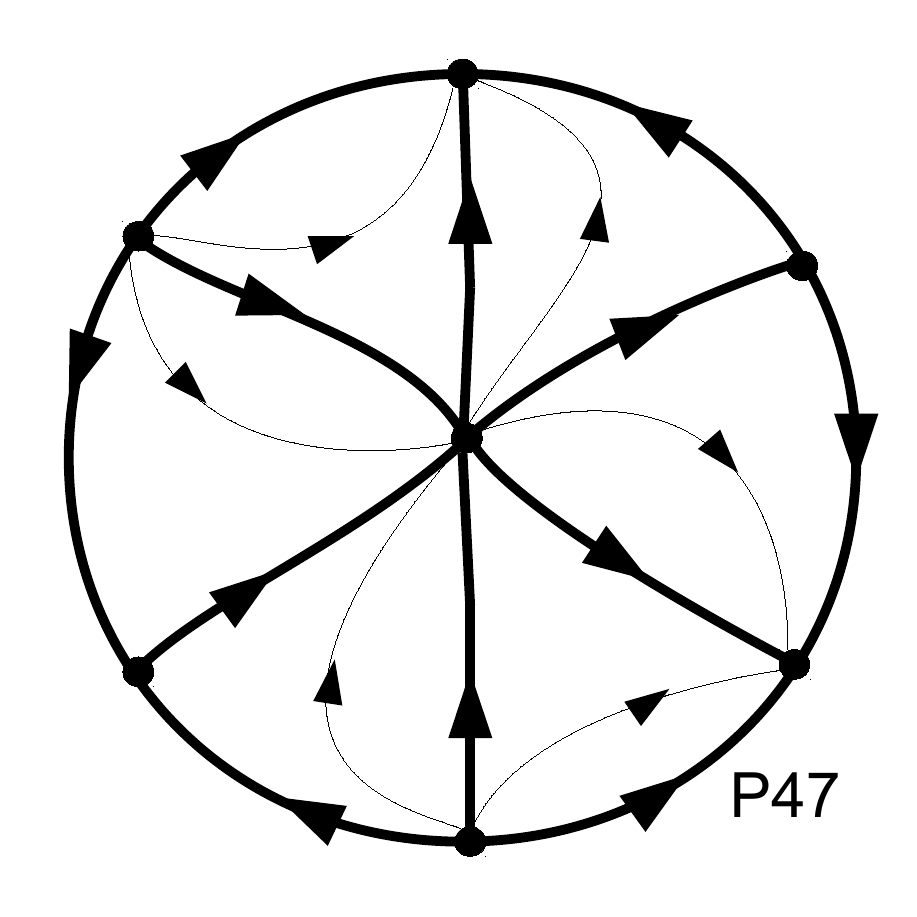}\end{minipage}
    \begin{minipage}[t]{2.7cm}\psfrag{d}{$d$}\centering\includegraphics[scale=.31]{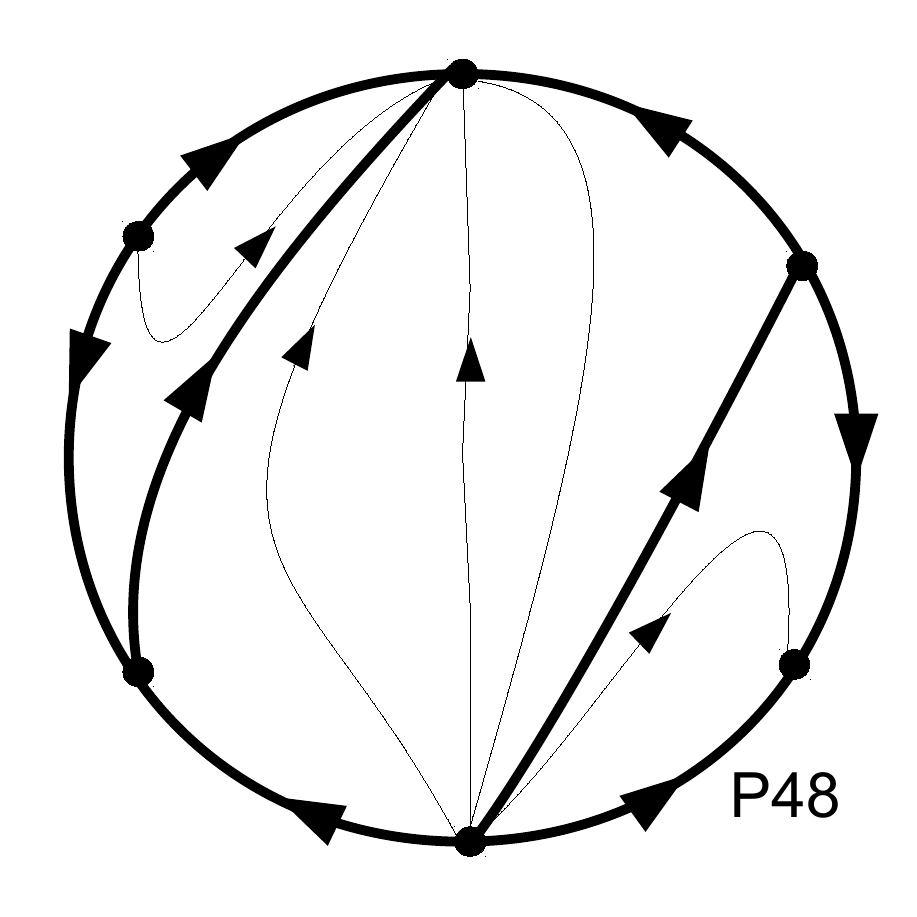}\end{minipage}

    \begin{minipage}[t]{2.7cm} \psfrag{a}{$a$}\centering\includegraphics[scale=.31]{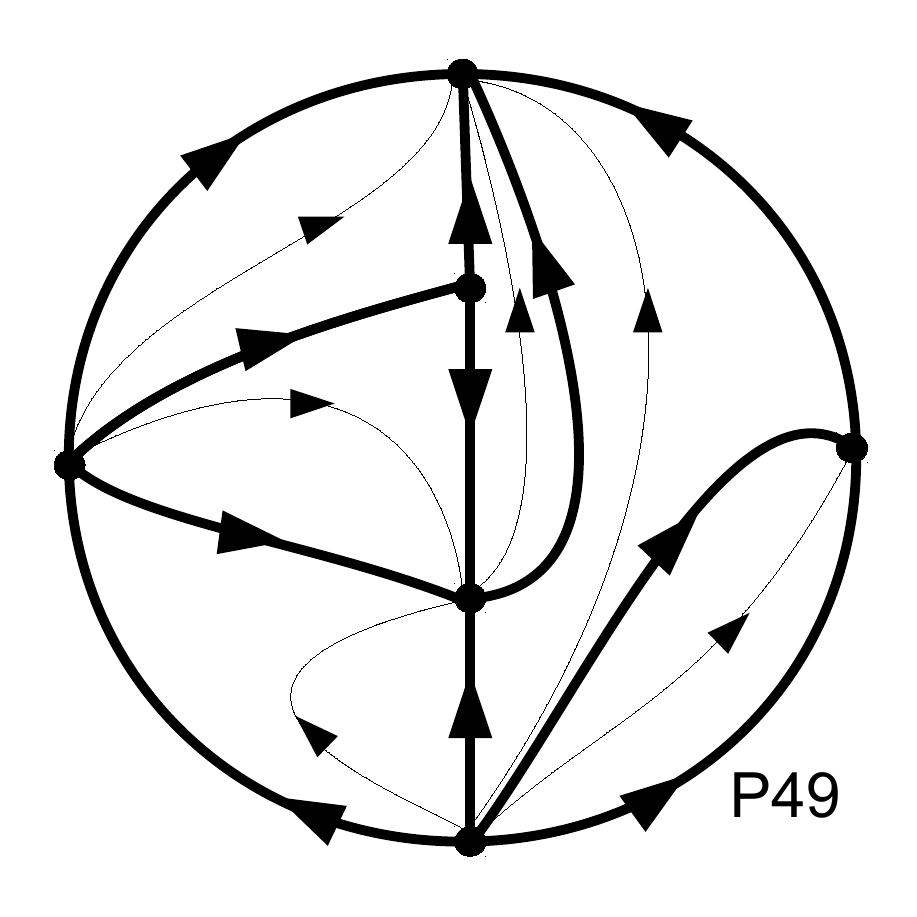}\end{minipage}
    \begin{minipage}[t]{2.7cm}\psfrag{b}{$b$}\centering\includegraphics[scale=.31]{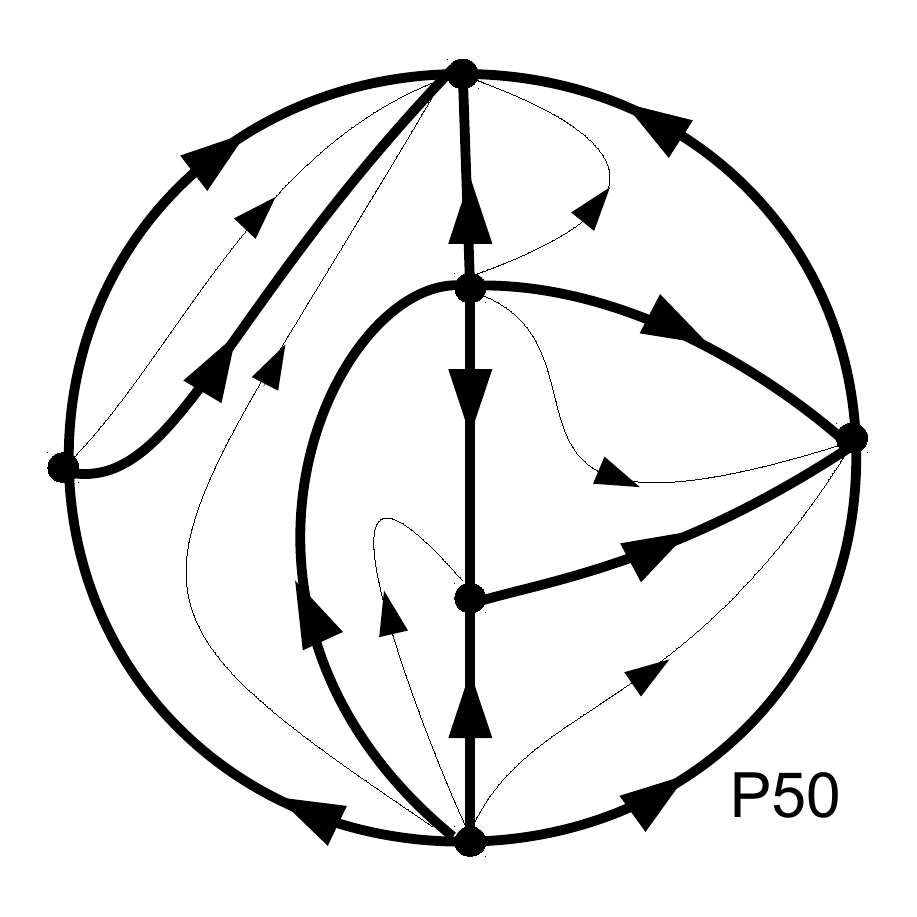}\end{minipage}
    \begin{minipage}[t]{2.7cm}\psfrag{c}{$c$}\centering\includegraphics[scale=.31]{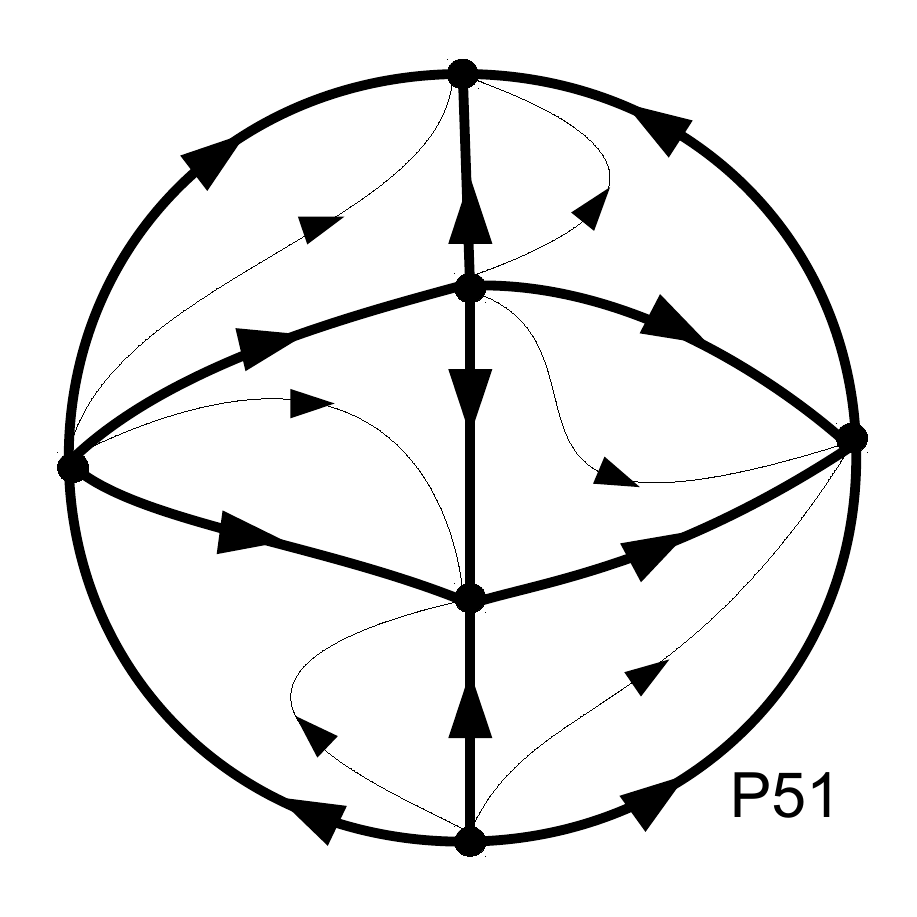}\end{minipage}
    \begin{minipage}[t]{2.7cm}\psfrag{d}{$d$}\centering\includegraphics[scale=.31]{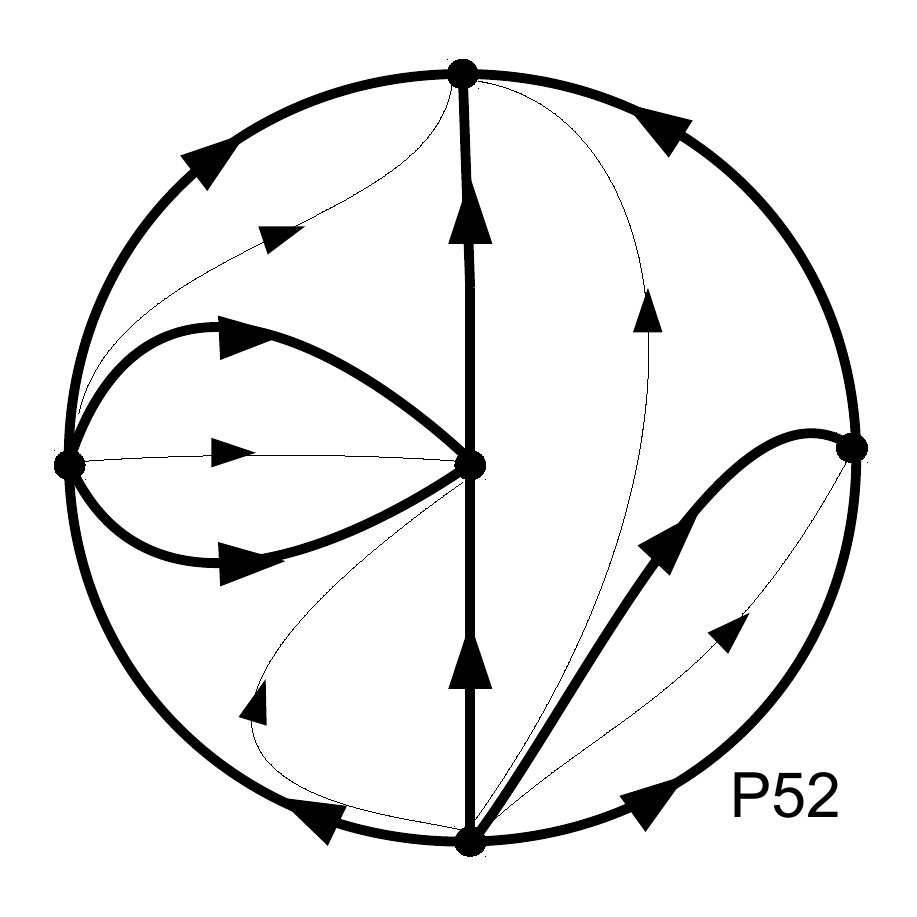}\end{minipage}

    \begin{minipage}[t]{2.7cm} \psfrag{a}{$a$}\centering\includegraphics[scale=.31]{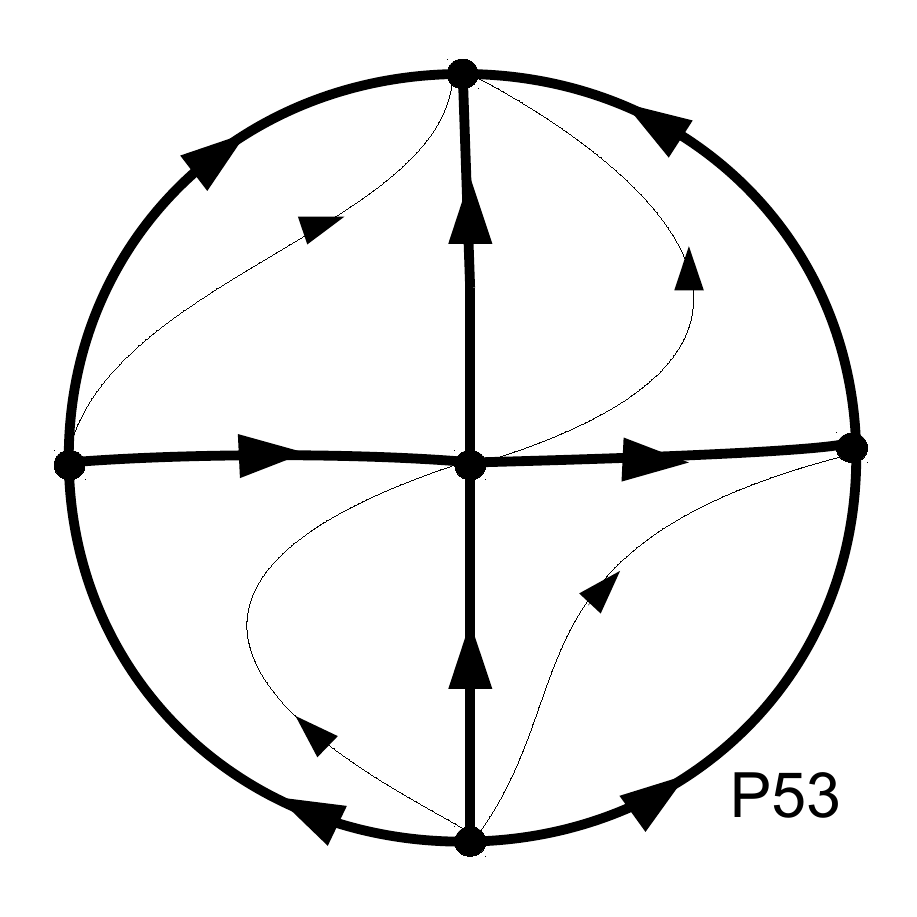}\end{minipage}
    \begin{minipage}[t]{2.7cm}\psfrag{b}{$b$}\centering\includegraphics[scale=.31]{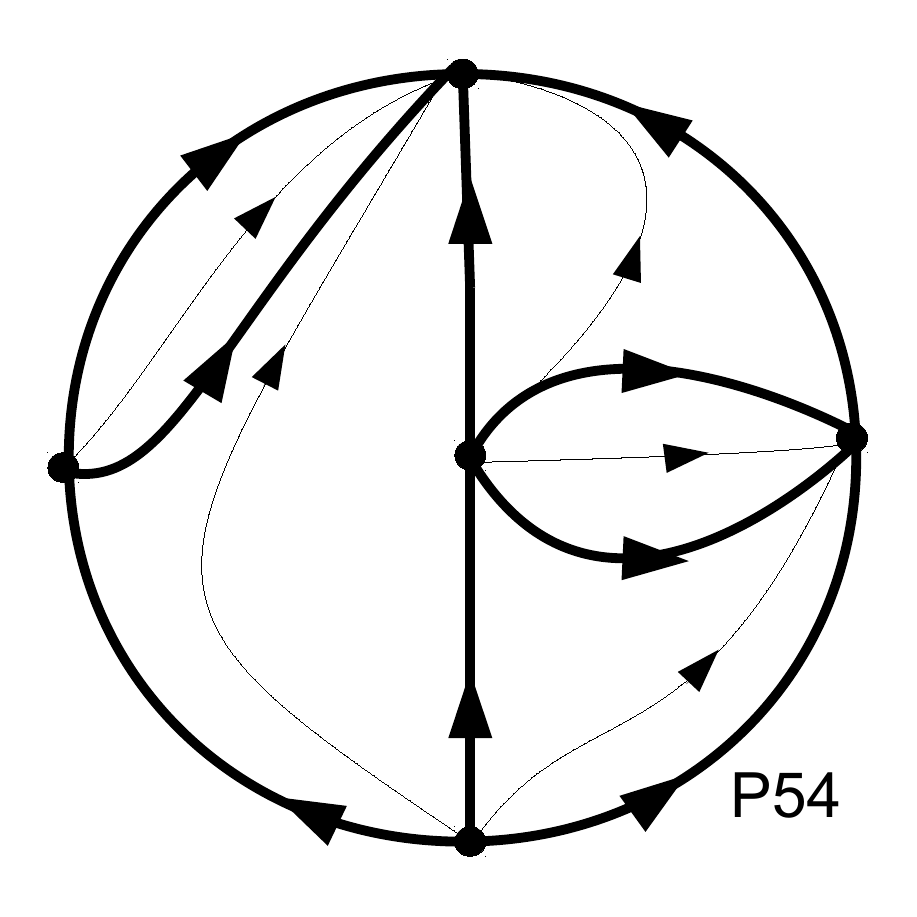}\end{minipage}
    \begin{minipage}[t]{2.7cm}\psfrag{c}{$c$}\centering\includegraphics[scale=.31]{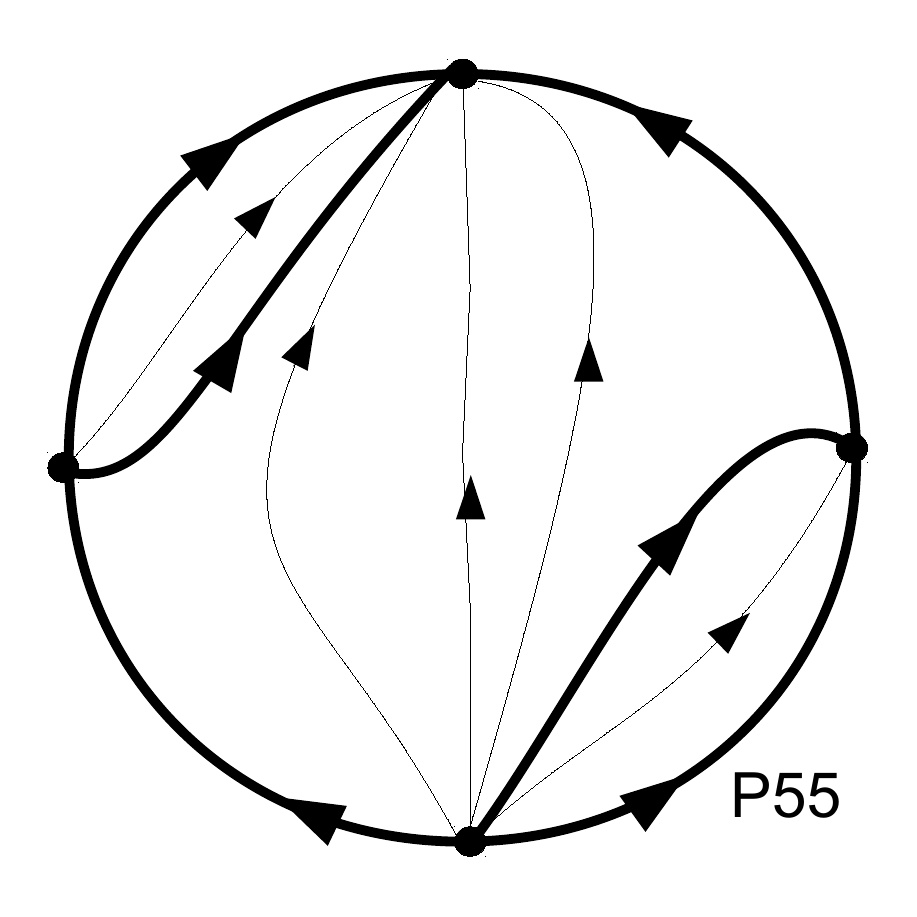}\end{minipage}
    \begin{minipage}[t]{2.7cm}\psfrag{d}{$d$}\centering\includegraphics[scale=.31]{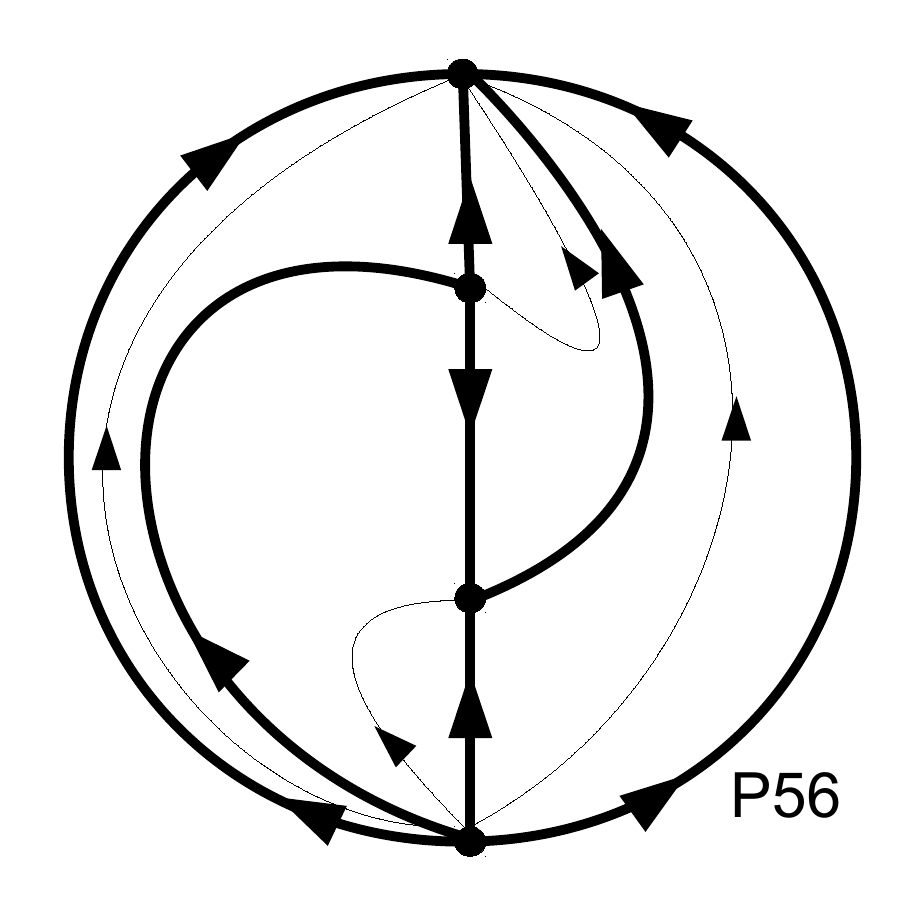}\end{minipage}

    \caption{\small Continuation of phase portraits of systems (1) in
        the Poicar\'e disk.}\label{figura1-3}
\end{figure}
\newpage

\begin{figure}
    \psfrag{A1}{A}

    \begin{minipage}[t]{2.7cm} \psfrag{a}{$a$}\centering\includegraphics[scale=.31]{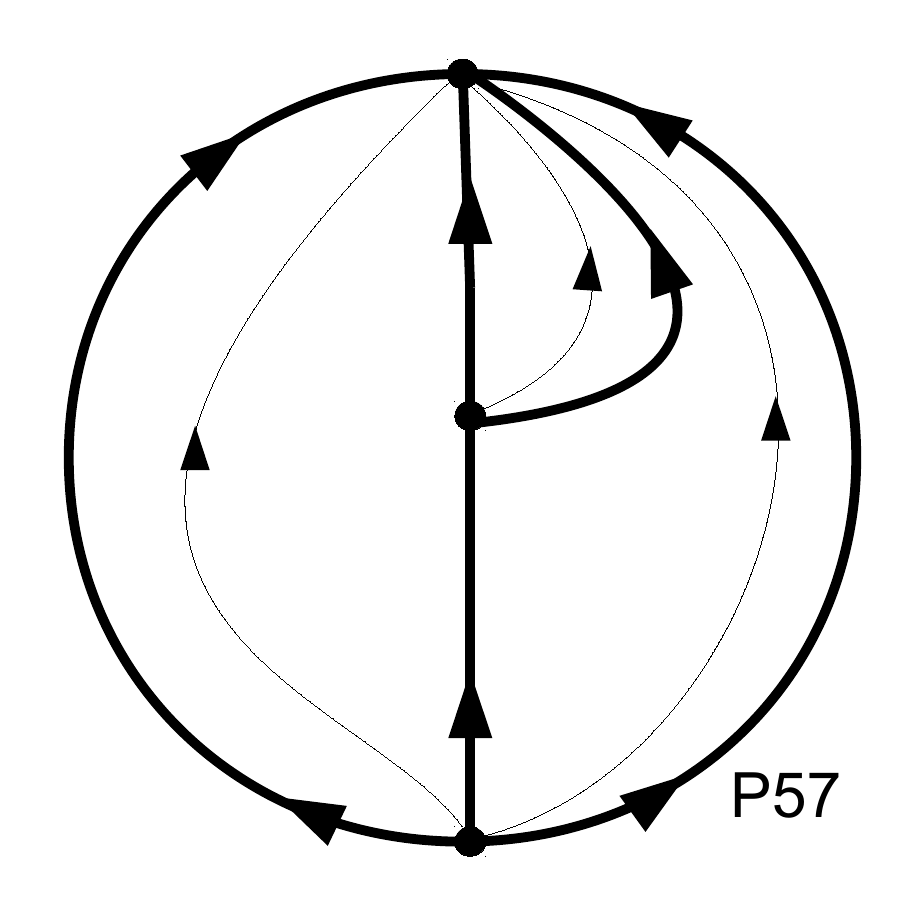}\end{minipage}
    \begin{minipage}[t]{2.7cm}\psfrag{b}{$b$}\centering\includegraphics[scale=.31]{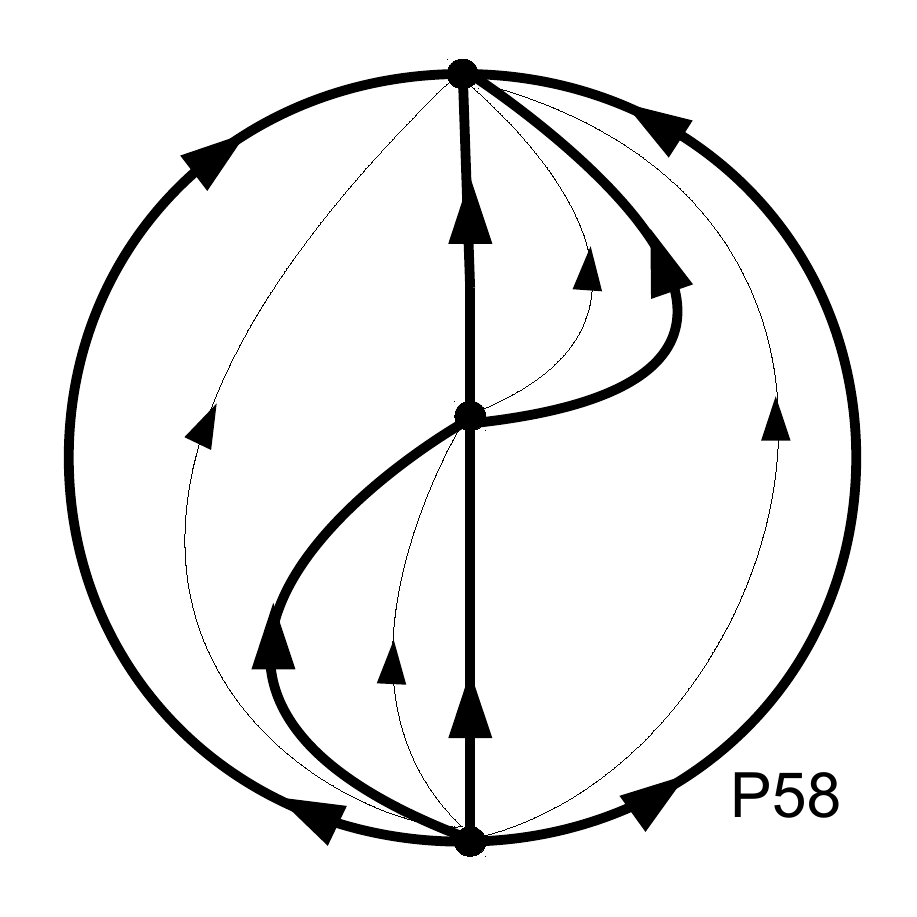}\end{minipage}
    \begin{minipage}[t]{2.7cm}\psfrag{c}{$c$}\centering\includegraphics[scale=.31]{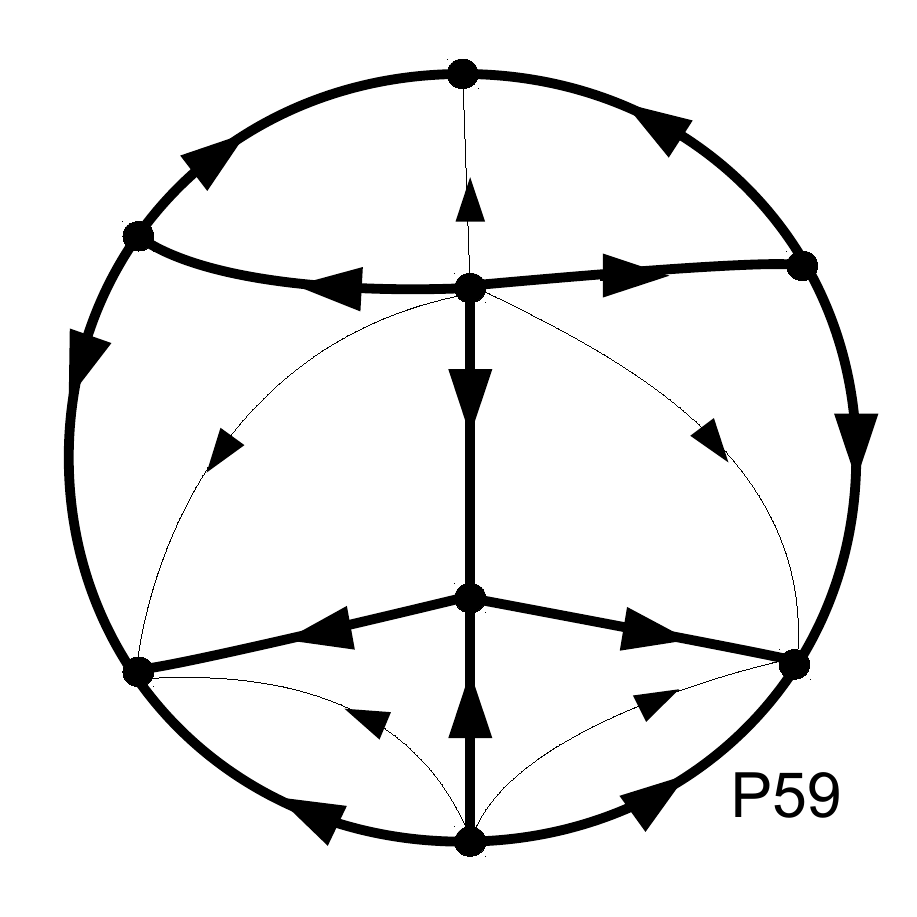}\end{minipage}
    \begin{minipage}[t]{2.7cm}\psfrag{d}{$d$}\centering\includegraphics[scale=.31]{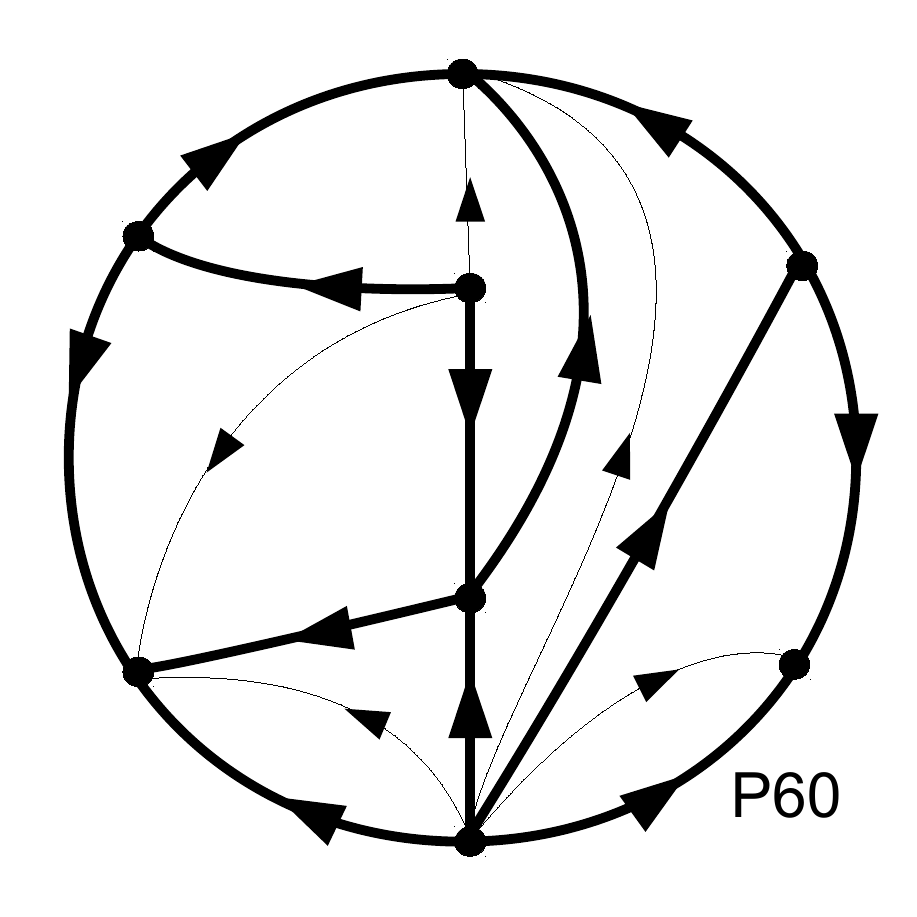}\end{minipage}

    \begin{minipage}[t]{2.7cm} \psfrag{a}{$a$}\centering\includegraphics[scale=.31]{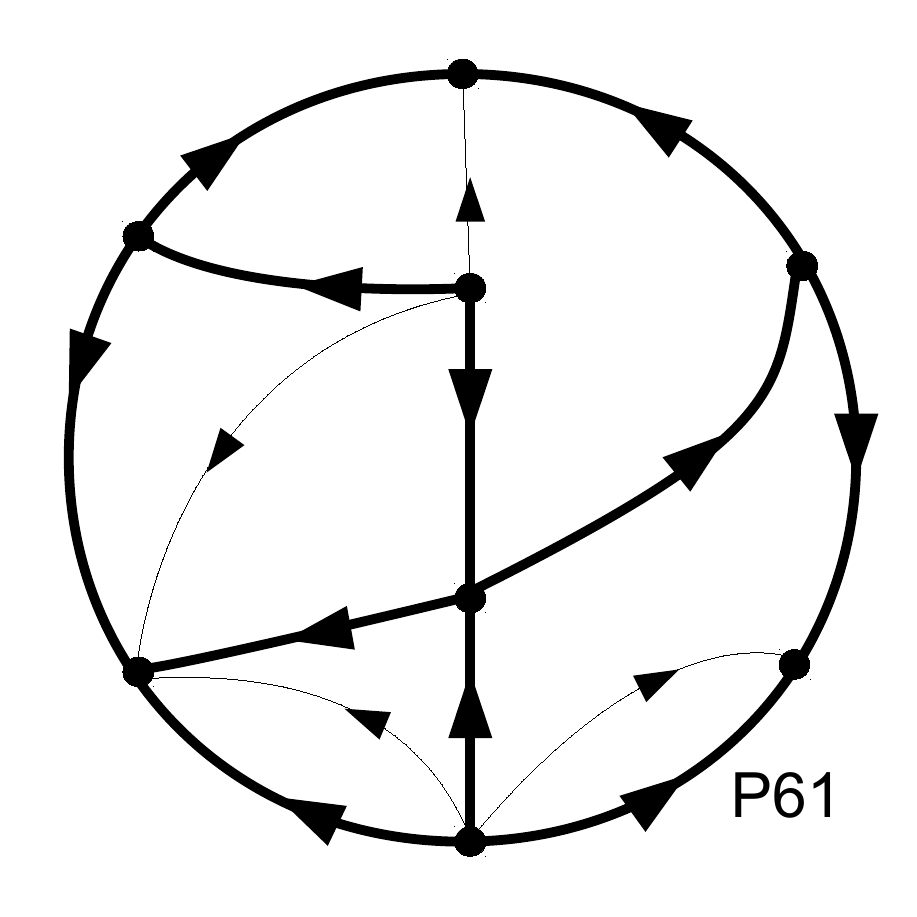}\end{minipage}
    \begin{minipage}[t]{2.7cm}\psfrag{b}{$b$}\centering\includegraphics[scale=.31]{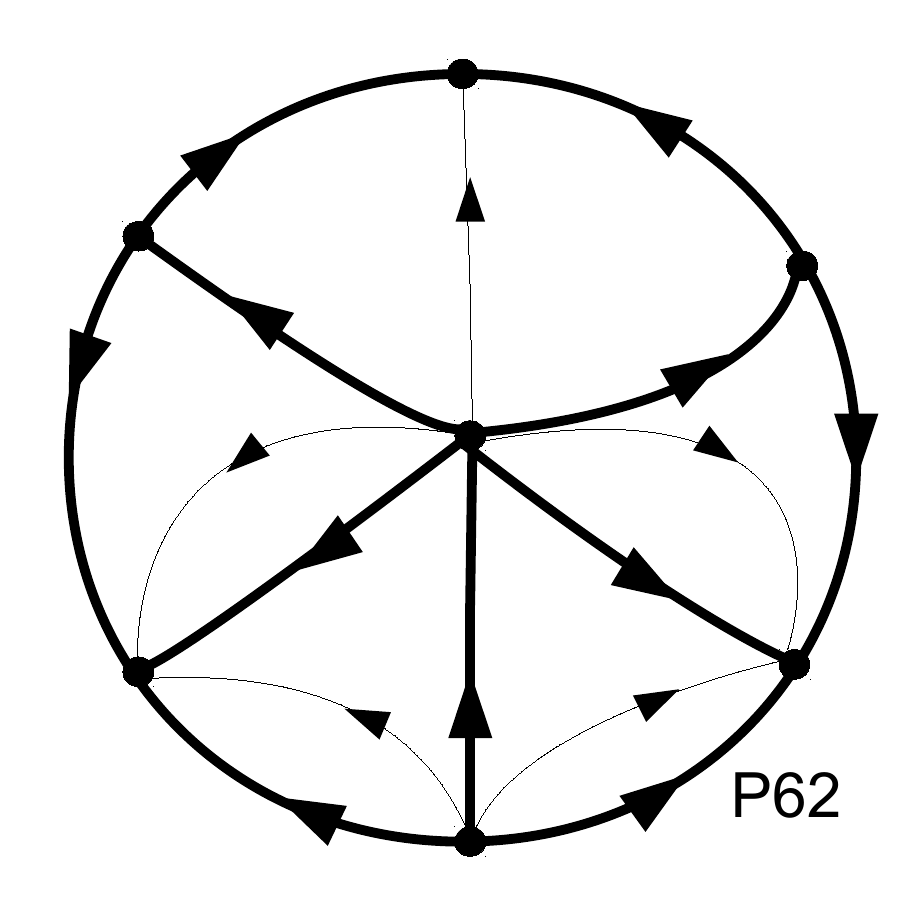}\end{minipage}
    \begin{minipage}[t]{2.7cm}\psfrag{c}{$c$}\centering\includegraphics[scale=.31]{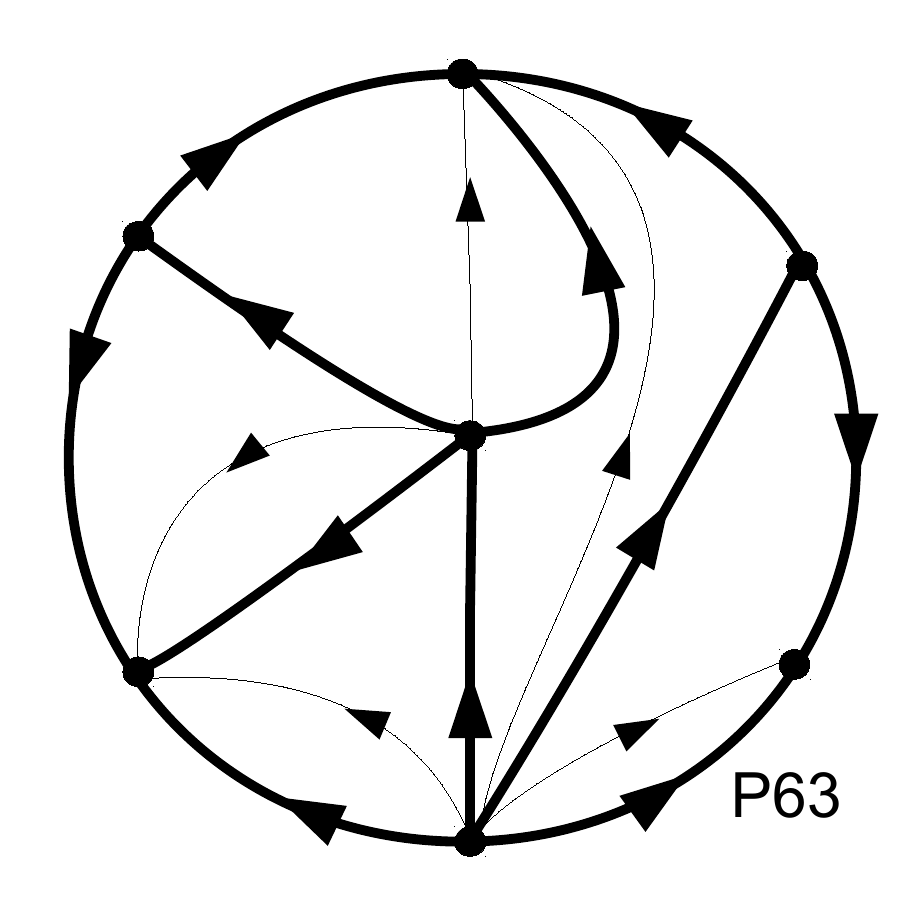}\end{minipage}
    \begin{minipage}[t]{2.7cm}\psfrag{d}{$d$}\centering\includegraphics[scale=.31]{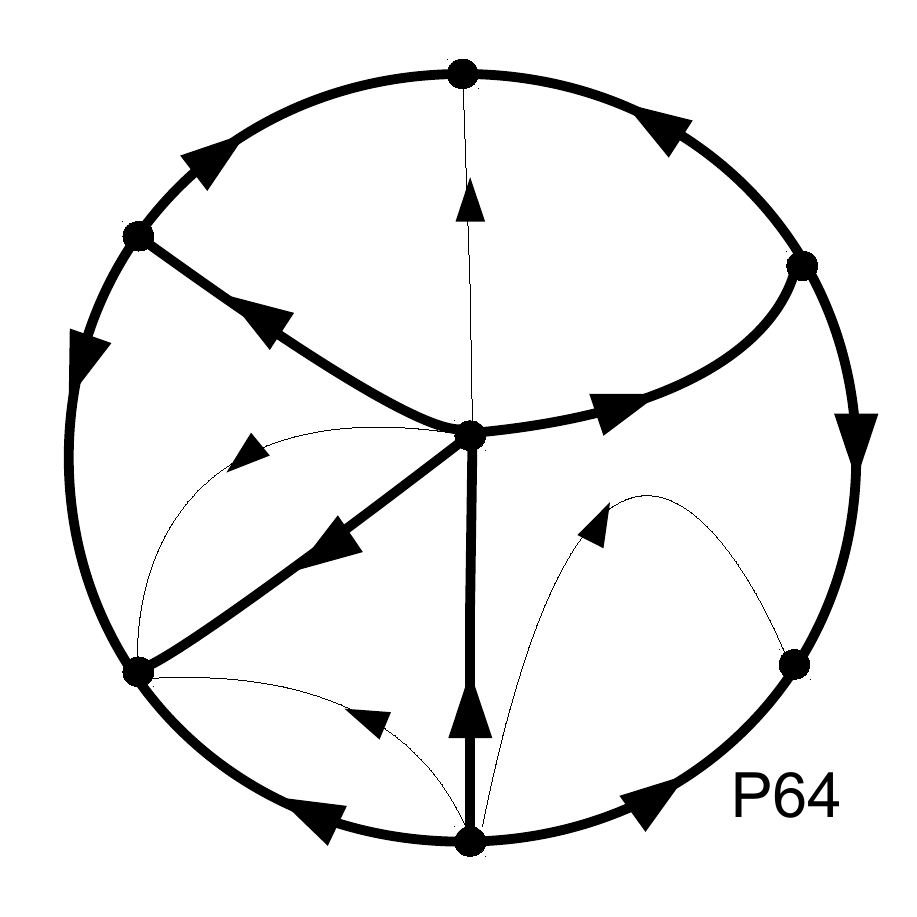}\end{minipage}

    \begin{minipage}[t]{2.7cm} \psfrag{a}{$a$}\centering\includegraphics[scale=.31]{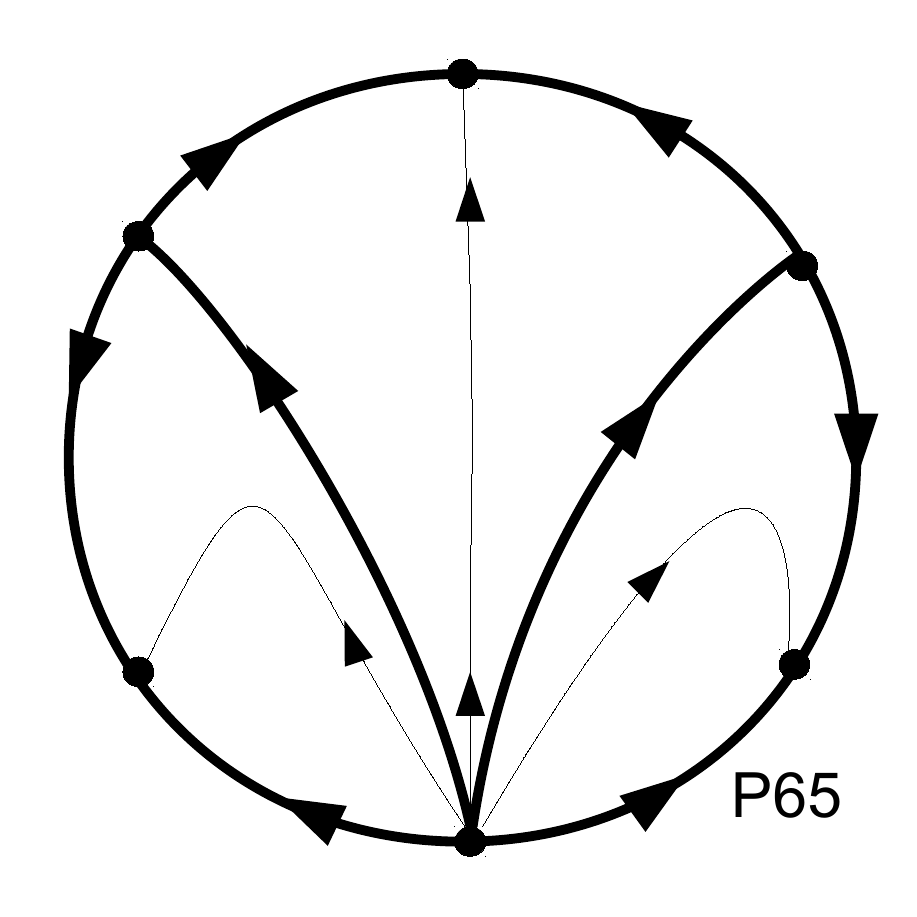}\end{minipage}
    \begin{minipage}[t]{2.7cm}\psfrag{b}{$b$}\centering\includegraphics[scale=.31]{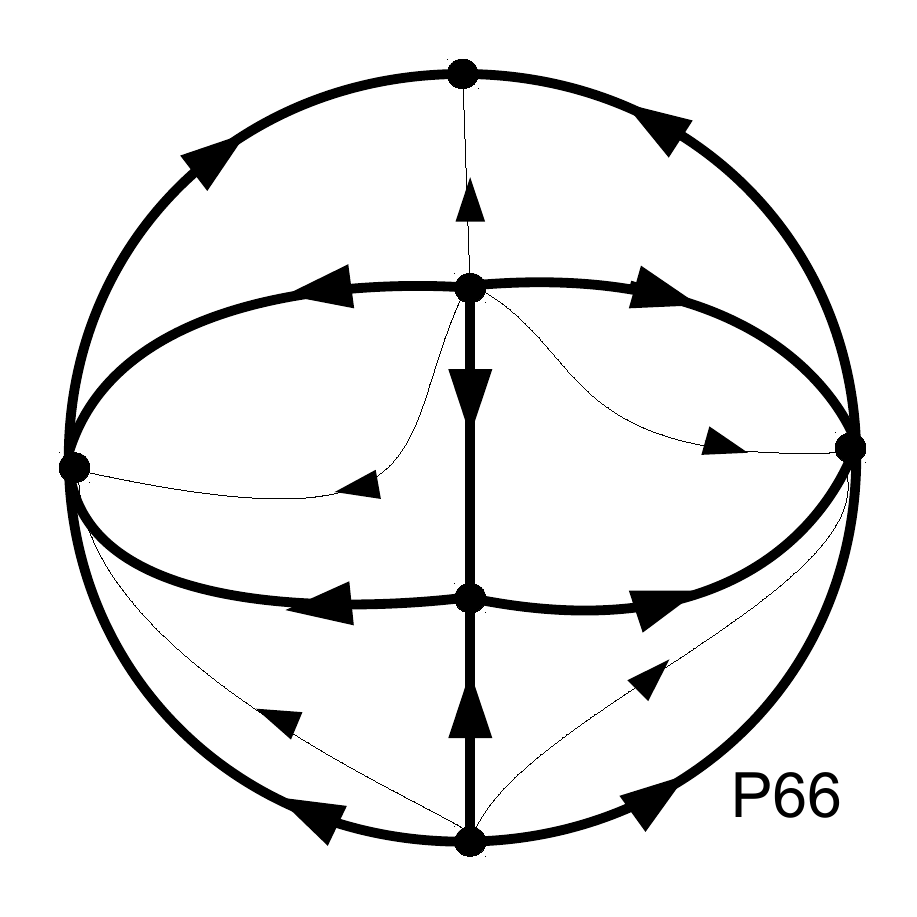}\end{minipage}
    \begin{minipage}[t]{2.7cm}\psfrag{c}{$c$}\centering\includegraphics[scale=.31]{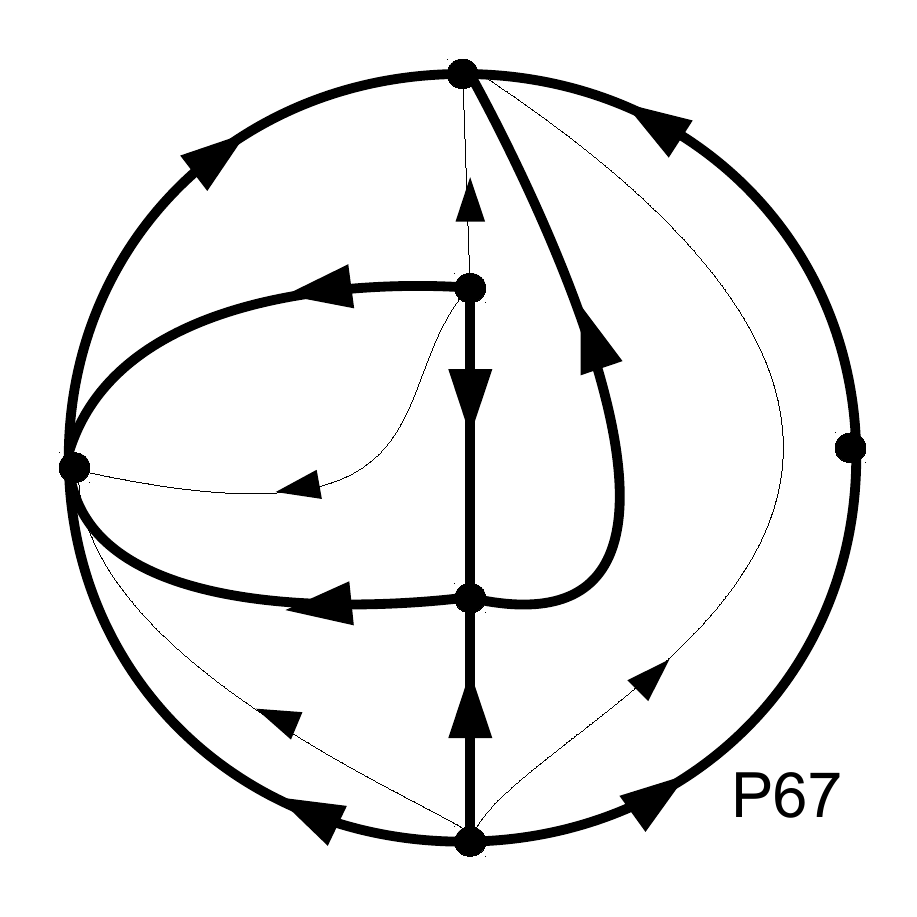}\end{minipage}
    \begin{minipage}[t]{2.7cm}\psfrag{d}{$d$}\centering\includegraphics[scale=.31]{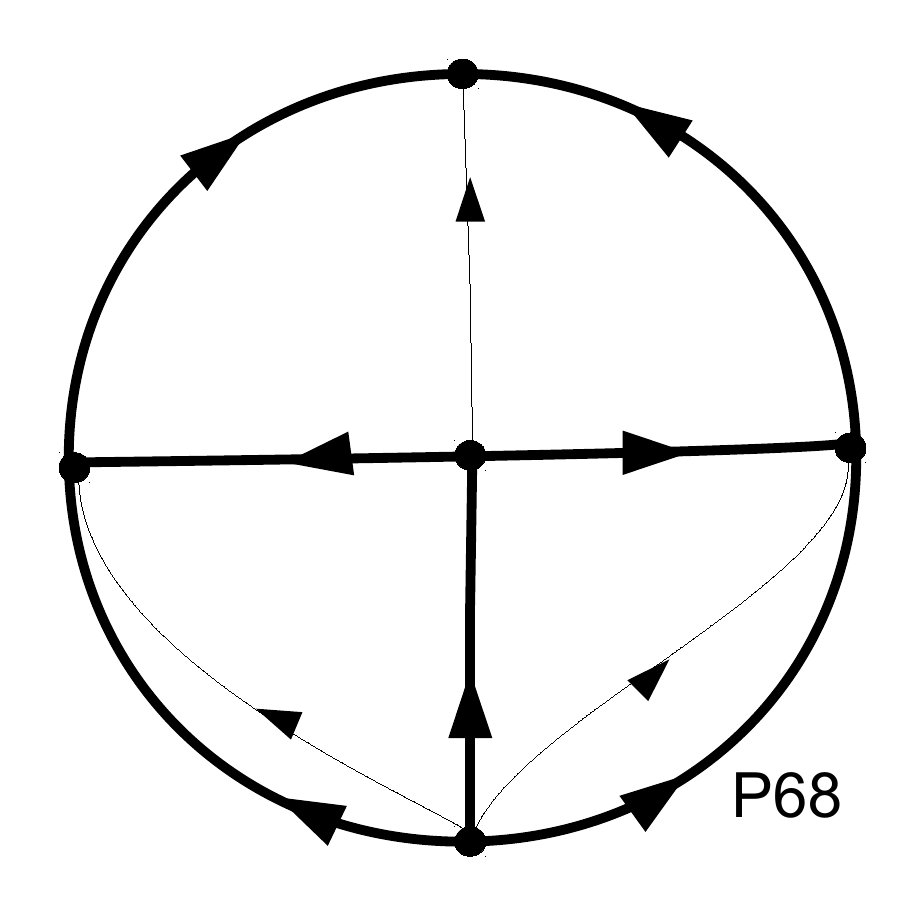}\end{minipage}

    \begin{minipage}[t]{2.7cm} \psfrag{a}{$a$}\centering\includegraphics[scale=.31]{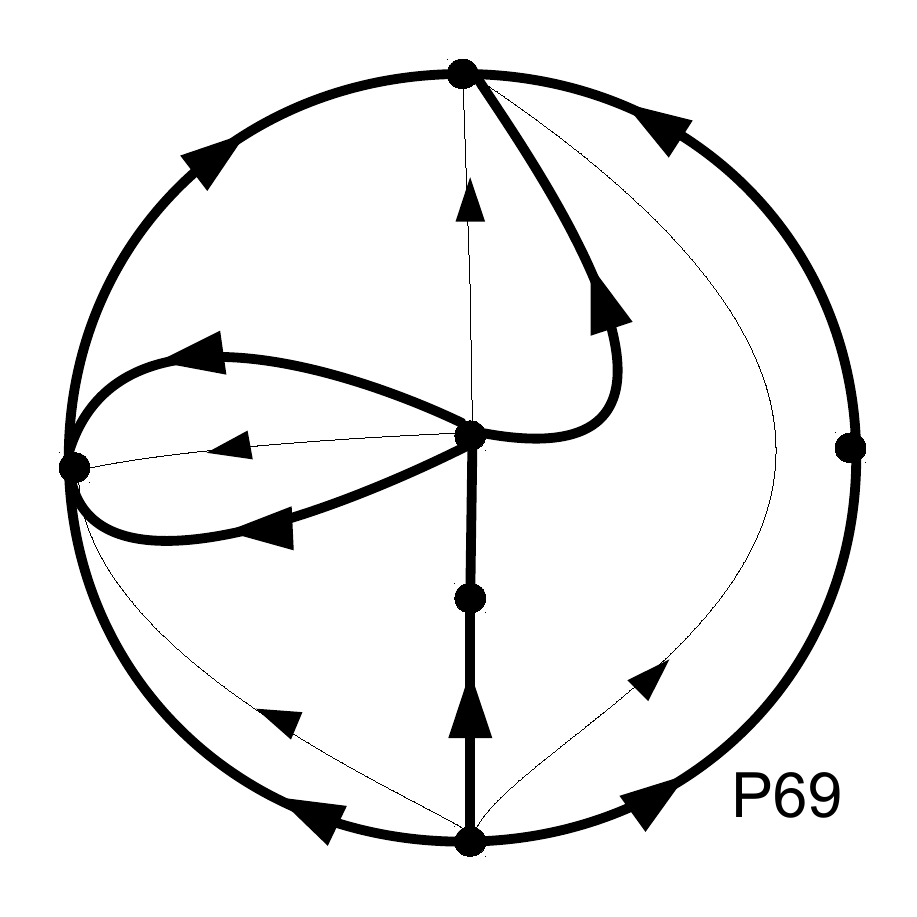}\end{minipage}
    \begin{minipage}[t]{2.7cm}\psfrag{b}{$b$}\centering\includegraphics[scale=.31]{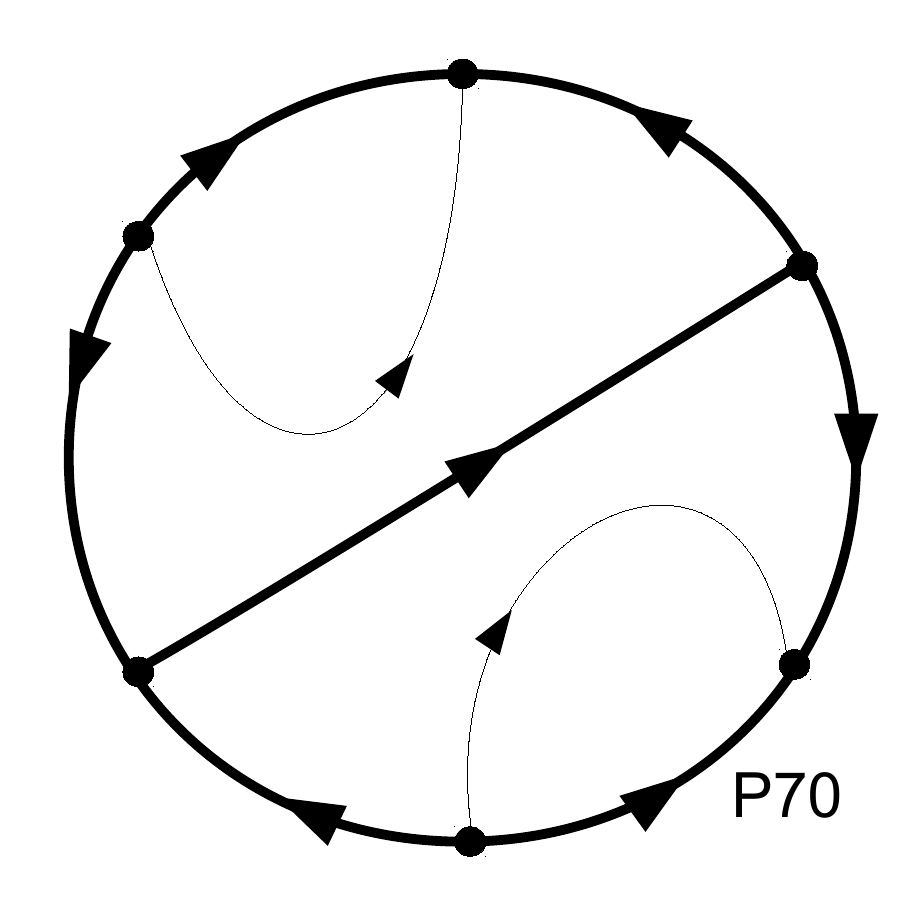}\end{minipage}
    \begin{minipage}[t]{2.7cm}\psfrag{c}{$c$}\centering\includegraphics[scale=.31]{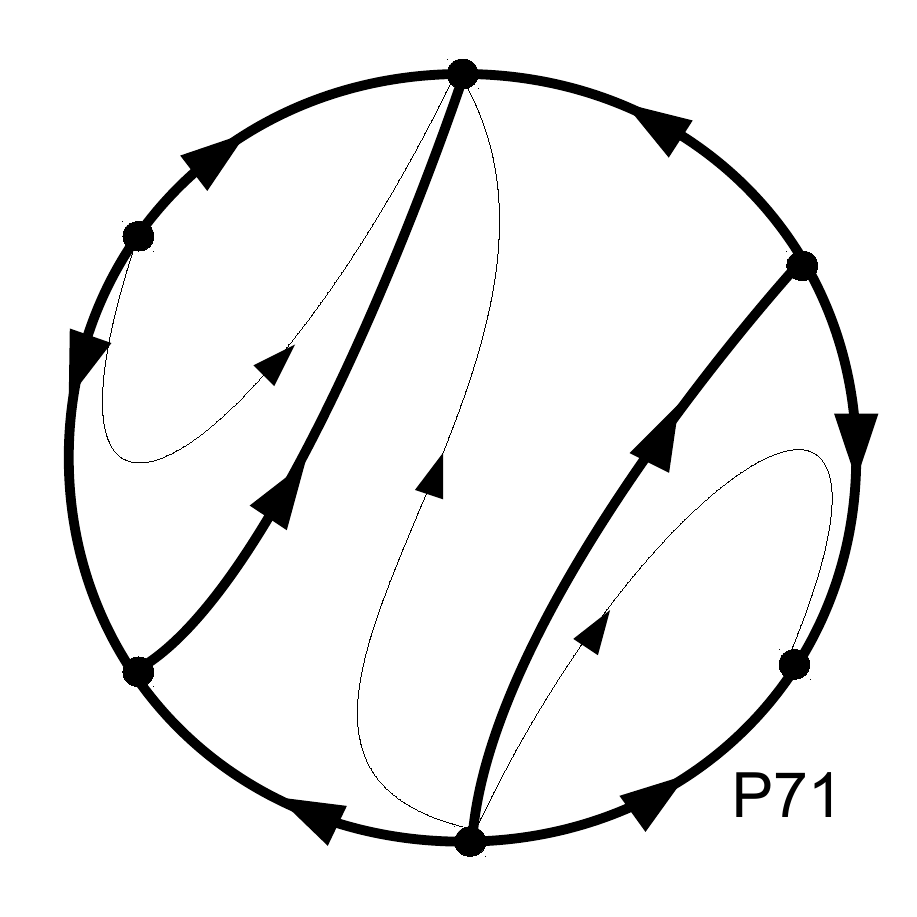}\end{minipage}
    \begin{minipage}[t]{2.7cm}\psfrag{d}{$d$}\centering\includegraphics[scale=.31]{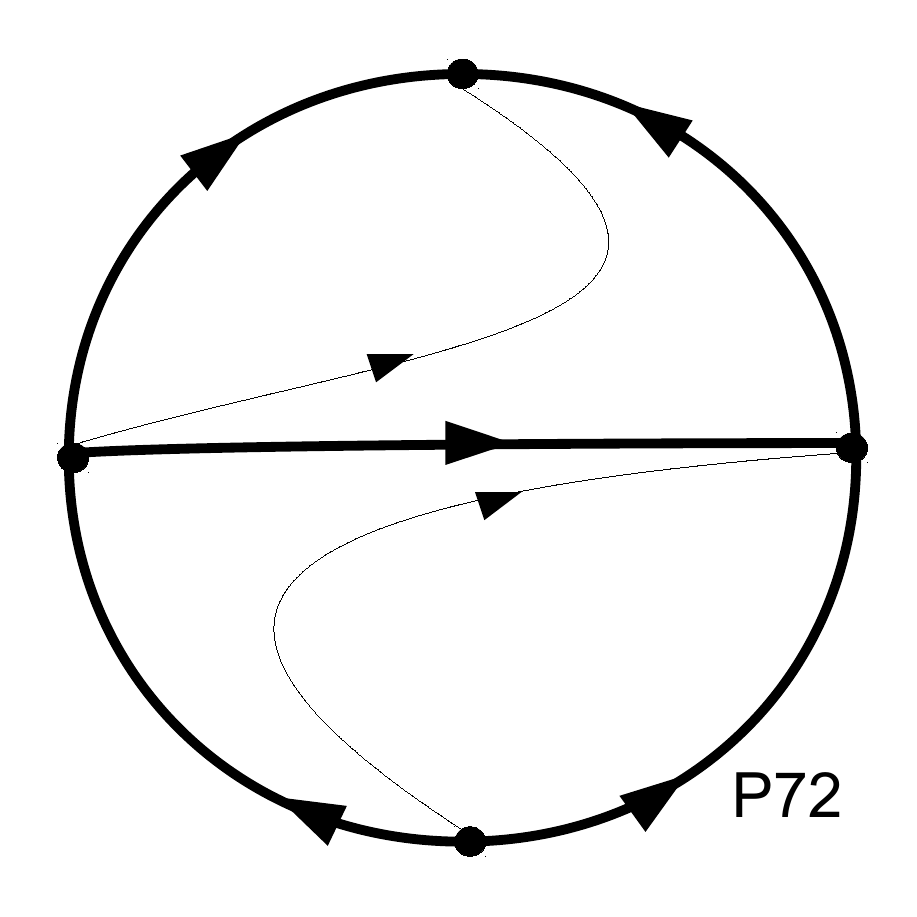}\end{minipage}

    \begin{minipage}[t]{2.7cm} \psfrag{a}{$a$}\centering\includegraphics[scale=.31]{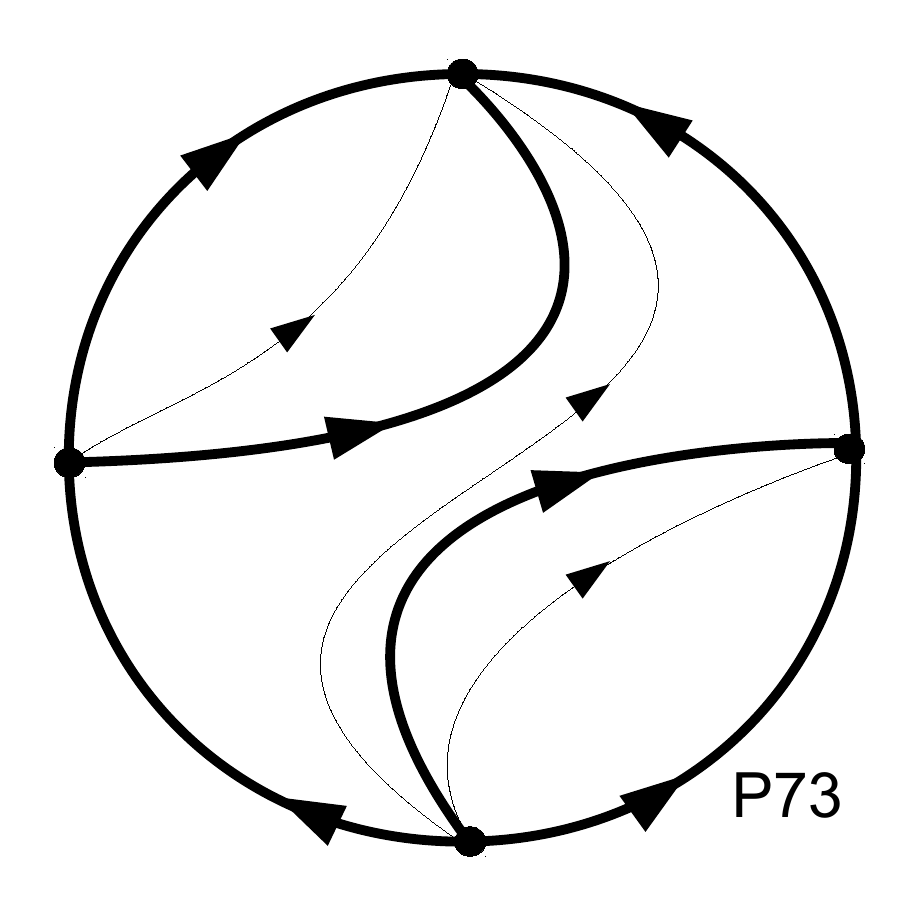}\end{minipage}
    \begin{minipage}[t]{2.7cm}\psfrag{b}{$b$}\centering\includegraphics[scale=.31]{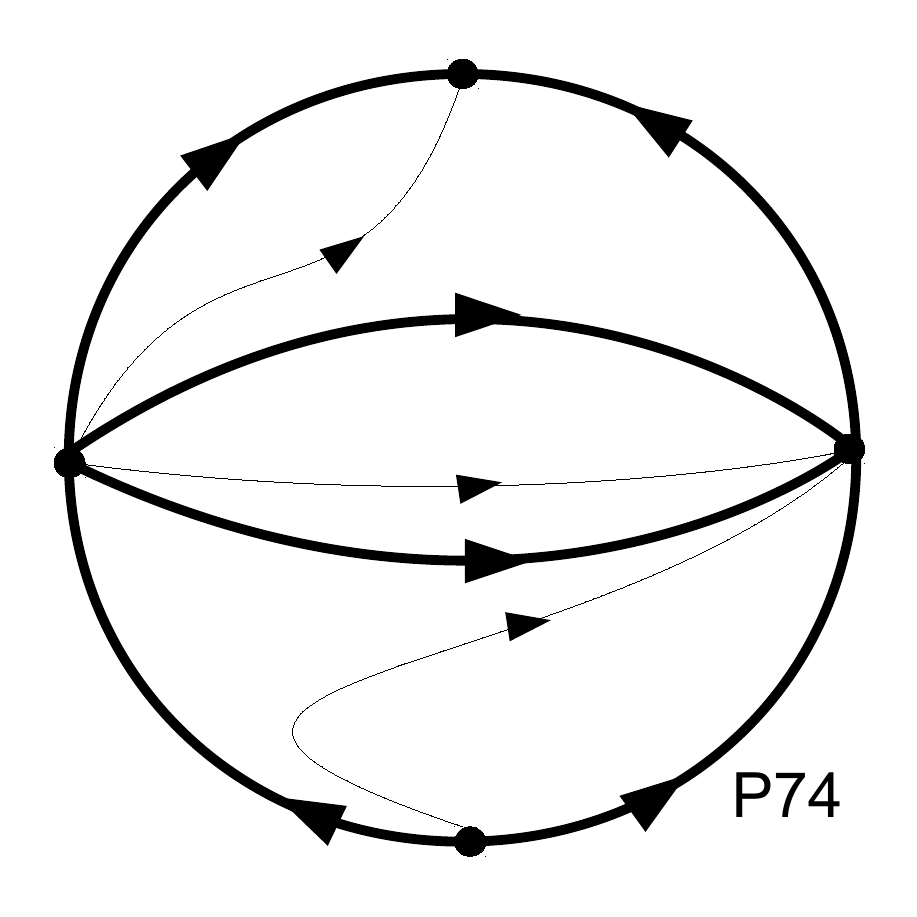}\end{minipage}

    \caption{\small Continuation of the phase portraits of systems (1) in
        the Poicar\'e disk.}\label{figura1-4}
\end{figure}

\newpage

\section{Finite equilibrium points}\label{s2}

We start this section with the proof of Proposition \ref{p0}.\\

\noindent\textit {Proof of Proposition \ref{p0}}. Since $ \a_2(x)$ is a polynomial of degree at most $2$, we have,
using a rescaling of the time if necessary,
\[\begin{array}{ll}
\dot{x}=(x-r)(x-s) \quad \mbox{with} \quad r\neq s,\\\dot{x}=(x-r)^2, \\ \dot{x}=(x-r), \\
\dot{x}=1, \\ \dot{x}=(x-r)^2+s^2 \quad \mbox{with} \quad s\neq 0.\end{array}\]

If $\dot{x}=(x-r)(x-s)$, $r\neq s$,  considering the change of coordinates
 $$x_1=\frac{x-r}{r-s},\quad y_1=cy \quad \mbox{and}\quad T=(r-s)t,$$
we get a system $(i)$. If $\dot{x}=(x-r)^n$, $n=1,2,$   considering the change of coordinates
 $x_1=x-r,\quad y_1=cy, $ we get systems $(ii)$ for $n=2$ and systems $(iii)$ for $n=1$.
If $\dot{x}=1$,   considering the change of coordinates
$x_1=x$ and $y_1=cy,$
we get a systems $(iv)$.
If $\dot{x}=(x-r)^2+s^2$,  considering the change of coordinates
$x_1=(x-r)/ s$, $y_1=cy$ and $T=st$,
we get a system $(v)$.\bbox

\begin{proposition}\label{p1}
The finite equilibrium points of  the Riccati quadratic polynomial differential system \eqref{eq1} are described below.
\begin{itemize}
\item [(a)]  Systems (i) have at most $4$ equilibria which can be either a saddle, or a stable or an unstable node, or a saddle--node.
\item [(b)]  Systems (ii) have at most $2$ equilibria which can be either a saddle--node either semi--hyperbolic  or nilpotent.
\item [(c)]  Systems (iii) have at most $2$ equilibria which can be either a saddle or an unstable node, or a saddle-node.
\item [(d)]  Systems (iv) and (v) have  no finite equilibria.
\end{itemize}
\end{proposition}

\begin{proof}

\noindent\textit {Systems (i):} Consider the Riccati quadratic polynomial differential systems
\begin{equation}\label{sisf1}
\dot{x} = x(x+1), \quad  \dot{y}= y^2+(ax+b) y + c x^2+ d x + e.
\end{equation}

The equilibrium  points of system \eqref{sisf1} are
\begin{center}$(x_1,y_1)=\left (0,-\dfrac{b +\sqrt{\Delta_{F_1}}}{2 }\right )$, $(x_2,y_2)=\left(0,-\dfrac{b -\sqrt{\Delta_{F_1}}}{2 }\right),$
\end{center}
\begin{center}$(x_3,y_3)=\left(-1,-\dfrac{-a+b +\sqrt{\Delta_{F_2}}}{2 }\right)$, $(x_4,y_4)=\left(-1,-\dfrac{-a+b -\sqrt{\Delta_{F_2}}}{2 }\right),$
\end{center}
where $\Delta_{F_1} $ and $\Delta_{F_2}$ are given by \eqref{deltas}.

The eigenvalues of  the Jacobian matrix of system \eqref{sisf1}
evaluated at $(x_i, y_i)$  are $(1, (-1)^{i}\sqrt{\Delta_{F_1}})$
for  $i=1,2$, and $(-1, (-1)^{i}\sqrt{\Delta_{F_2}})$ for  $i=3,4$, when they exist.\\
From the classification of the hyperbolic and semi-hyperbolic equilibrium points (see for instance Theorems 2.18 and 2.19 of \cite{DLA}),
we have  the following (when the equilibrium point is not hyperbolic we mention this fact explicitly).
\begin{itemize}

 \item[(i)] If $\Delta_{F_1}>0$ and  $\Delta_{F_2}>0$, system \eqref{sisf1} has  two saddles,
a stable  node and an unstable node.

 \item[(ii)]If $\Delta_{F_1}>0$ and  $\Delta_{F_2}=0$, system \eqref{sisf1} has a saddle,
a stable node and a semi--hyperbolic saddle--node.

 \item[(iii)]If $\Delta_{F_1}>0$ and  $\Delta_{F_2}<0$, system \eqref{sisf1} has  a saddle and
a stable node.

 \item[(iv)]If $\Delta_{F_1}=0$ and  $\Delta_{F_2}>0$, system \eqref{sisf1} has a saddle,
an unstable node and a semi--hyperbolic saddle--node.

 \item[(v)]If $\Delta_{F_1}=0$ and  $\Delta_{F_2}=0$, system \eqref{sisf1} has two semi--hyperbolic saddle--nodes.

 \item[(vi)]If $\Delta_{F_1}=0$ and  $\Delta_{F_2}<0$ system \eqref{sisf1} has one semi--hyperbolic saddle--node.

 \item[(vii)]If $\Delta_{F_1}<0$ and  $\Delta_{F_2}>0$,  system \eqref{sisf1} has a saddle and
an unstable node.

 \item[(viii)]If $\Delta_{F_1}<0$ and  $\Delta_{F_2}=0$, system \eqref{sisf1} has one semi--hyperbolic saddle--node.

 \item[(ix)]If $\Delta_{F_1}<0$ and  $\Delta_{F_2}<0$, system \eqref{sisf1}  has not  equilibria.
\end{itemize}

\noindent\textit {Systems (ii):} Consider the  Riccati  quadratic polynomial differential systems
\begin{equation}\label{sisf2} \dot{x} = x^2,\quad \dot{y}=y^2+(a x+b) y + cx^2+ dx + e.\\
\end{equation}
We have that the finite equilibrium points of  system \eqref{sisf2} are
\begin{equation} \label{singii} (x_1,y_1)=\left (0,-\dfrac{b +\sqrt{\Delta_{F_1}}}{2 }\right ),  (x_2,y_2)=\left(0,-\dfrac{b -\sqrt{\Delta_{F_1}}}{2 }\right),
\end{equation}
where  $\Delta_{F_1} $ is given by \eqref{deltas}.

The eigenvalues of  the Jacobian matrix of system \eqref{sisf2}
evaluated at $(x_i, y_i)$ for all $i=1, 2$ are $0$ and $(-1)^{i}\sqrt{\Delta_{F_1}}$. Then we have\\

\begin{itemize}
 \item[(i)] If $\Delta_{F_1}>0$ systems \eqref{sisf2} have two semi--hyperbolic saddle--nodes.
   \item[(ii)]If $\Delta_{F_1}=0$ then systems \eqref{sisf2} have one nilpotent saddle--node equilibrium point.
   \item[(iii)]  If $\Delta_{F_1}<0$ system \eqref{sisf2} has not equilibrium points.
 \end{itemize}

\noindent\textit {Systems (iii):} Consider the Riccati quadratic polynomial differential systems
\begin{equation}\label{sisf3}
\dot{x} = x,\quad \dot{y}=y^2+(ax+b) y + cx^2+ d x + e.
\end{equation}
 The equilibrium points of systems \eqref{sisf3} are given by \eqref{singii}.
Then system \eqref{sisf3}
has $0, 1$ or  $2$ equilibrium points if  $\Delta_{F_1}$ is negative, zero or positive, respectively.
The eigenvalues of  the Jacobian matrix of system \eqref{sisf3}
evaluated at $(x_i, y_i)$ for $i=1,2$ are $1$ and $(-1)^{i}\sqrt{\Delta_{F_1}}$. Thus we have:

\begin{itemize}

\item[(i)] If  $\Delta_{F_1}>0$ systems \eqref{sisf3} have  a saddle and an unstable node.

\item[(ii)] If  $\Delta_{F_1}=0$ systems \eqref{sisf3} have  a semi--hyperbolic saddle--node.

\item[(iii)] If $\Delta_{F_1}<0$ system \eqref{sisf3} has no equilibria.
\end{itemize}

\noindent\textit {Systems (iv) and (v):}
These systems are  chordal quadratic systems, or  quadratic system without
finite singularities.

\end{proof}

\section{Infinite equilibrium points}\label{s3}

For a complete description of the Poincar\'{e} compactification method we refer to chapter 5 of \cite{DLA}.
In what follows we remember some formulas.\\

Consider  a polynomial differential system  in $\mathbb{R}^2$ with degree 2.
\begin{equation}
\dot{x}= P(x,y), \quad \dot{y}= Q(x,y)
\end{equation}
or equivalently its associated polynomial vector field $X=(P,Q)$.
As we said before, any polynomial differential system can be
extended to an analytic differential system on a closed
disk of radius one centered at their origin of coordinates, whose interior is diffeomorphic to $\R^2$ and
its boundary, the circle $\s^1$, plays the role of
the infinity.  \\
We consider 4 open charts covering the disk  $\D$:

$$\phi_1: \R^2  \longrightarrow U_1,\quad   \phi_1(x,y)=(1/v, u/v) ,$$
$$\phi_2: \R^2  \longrightarrow U_2,\quad   \phi_1(x,y)=(u/v, 1/v) $$
and
$$\psi_k: \R^2  \longrightarrow V_k, \quad  \psi_k(x,y)=-\phi_k(x,y), \quad k=1,2$$
with
$$U_1=\{(u,v)\in \D: u^2+v^2\leq 1 \quad \mbox{and}\quad  u> 0\},$$
$$U_2=\{(u,v)\in \D: u^2+v^2\leq 1 \quad \mbox{and}\quad  v> 0\},$$
$$V_1=\{(u,v)\in \D: u^2+v^2\leq 1 \quad \mbox{and}\quad  u<0\},$$
$$V_2=\{(u,v)\in \D: u^2+v^2\leq 1 \quad \mbox{and}\quad  v< 0\}.$$

The Poincar\'{e} compactification is denoted by $p(X)$.
The expression of $p(X)$  in the chart $U_1$ is
\begin{equation}\label{u1}
\dot{u}=v^2(-uP+Q),\quad \dot{v}=-v^3P,
\end{equation}
where  $P$ and $Q$ are
evaluated at $(1/v,u/v)$.% Moreover in all these local charts the points $(u,v)$ of the infinity have at $v=0$.\\

The expression of $p(X)$  in the chart $U_2$ is
\begin{equation}\label{u2}
\dot{u}=v^2(P-uQ),\quad \dot{v}=-v^3Q,
\end{equation}
where  $P$ and $Q$ are
evaluated at $(u/v,1/v)$. Moreover in all these local charts the points $(u,v)$ of the infinity have its coordinate $v=0$.\\

The expression for the extend differential system in the local chart $V_i$,
$i=1,2$ is the same as in $U_i$ multiplied by $-1$.

\begin{proposition}\label{p2}
On the circle of the infinity, for any systems of Proposition \ref{p0} the origin of $U_2$, denoted by $n$,
is an attracting node and the origin of $V_2$, denoted by $s$, is a repelling node of the Riccati quadratic
 polynomial differential system \eqref{eq1}. Moreover, the remaining infinite equilibrium
 points are described below.
\begin{itemize}
\item[(a)] For systems (i), (ii) and (v) three situations can occur.
\begin{itemize}
\item $4$ equilibrium points being $2$ saddles, $1$ attracting node  and $1$ repelling node;
\item $2$ equilibrium points being $2$ saddle-nodes;
\item The only equilibria are $n$ and $s$.
\end{itemize}
\item [(b)] For systems (iii) three situations can occur.
\begin{itemize}
\item $4$ equilibrium points being $4$ nilpotent saddle-nodes;
\item $2$ equilibrium points being $2$  semi--hyperbolic saddle-nodes;
\item The only equilibria are $n$ and $s$.
\end{itemize}
\item [(c)] For systems  (iv) three situations can occur.
\begin{itemize}
\item $4$ equilibrium points being $2$ semi-hyperbolic saddles, $1$ semi-hyperbolic attracting node  and $1$ semi-hyperbolic repelling node;
\item $2$ equilibrium points being $2$ semi-hyperbolic saddle-nodes;
\item The only equilibria are $n$ and $s$.
\end{itemize}
\end{itemize}
\end{proposition}

\begin{proof}
\noindent\textit {Systems (i):}
First we analyze the phase portrait in the local  chart $U_1$.
The expression of the system  in this chart is
 \begin{equation} \label{sisi1}
\dot{u} =v ((b-1)u + d) + e v^2+ p(u),  \quad  \dot{v}=-(v + v^2),
\end{equation}
where $p(u)= u^2+(a-1) u+c$.

Note that $(u_0,0)$ is an infinite equilibrium point of  \eqref{sisi1} if, and only if, $p(u_0)=0.$
System \eqref {sisi1} has $0,1$ or $2$ two infinite equilibrium points:
\begin{equation}\label {Si} S_i=\left(\dfrac{1-a +(-1)^i\sqrt{\Delta_{I_1}}}{2 },0\right),\end{equation}
for $i=1,2$, where $\Delta_{I_1}$ is given \eqref{deltas}.

The eigenvalues of the Jacobian matrix of  system \eqref{sisi1} are $-1$ and $(-1)^{i}\sqrt{\Delta_{I_1}}$. Thus we have:

\begin{itemize}
\item[(i)] If  $\Delta_{I_1}>0$  systems \eqref {sisi1} have  a saddle and a stable node.

\item[(ii)] If  $\Delta_{I_1}=0$  systems \eqref {sisi1} have  a semi--hiperbolic saddle--node.

\item[(iii)] If $\Delta_{I_1}<0$  systems \eqref {sisi1} have no equilibrium points.
\end{itemize}

Now we analyze the phase portrait in the local  chart $U_2$,
we need to the study the origin of $U_2$, the others infinite
singularity ahead, have been studied in the local chart $U_1$.
The expression of the system  in this chart is
\begin{equation} \label{sisi11}
\begin{array}{lcl}\dot{u} =v( v ue- u (d u+b-1))+q(u),  \\ \dot{v}=-v \left(1+a u+cu^2\right)-v^2 (b+d u)-e v^3,\\
\end{array}
\end{equation}
where $q(u)=-u(1+ (a-1)u+c u^2).$

The eigenvalues of the Jacobian matrix at the origin of $U_2$ of system  \eqref{sisi11} are  $-1$ and $-1 $. Therefore system
\eqref {sisi11} has a stable node at $(0,0)$.

Thus, the equilibrium points of system \eqref{eq1}, system $(i)$, on the circle
$\sss^1$ are classified as follows.

\begin{figure}[!htb]
    \begin{overpic}[width=4cm]{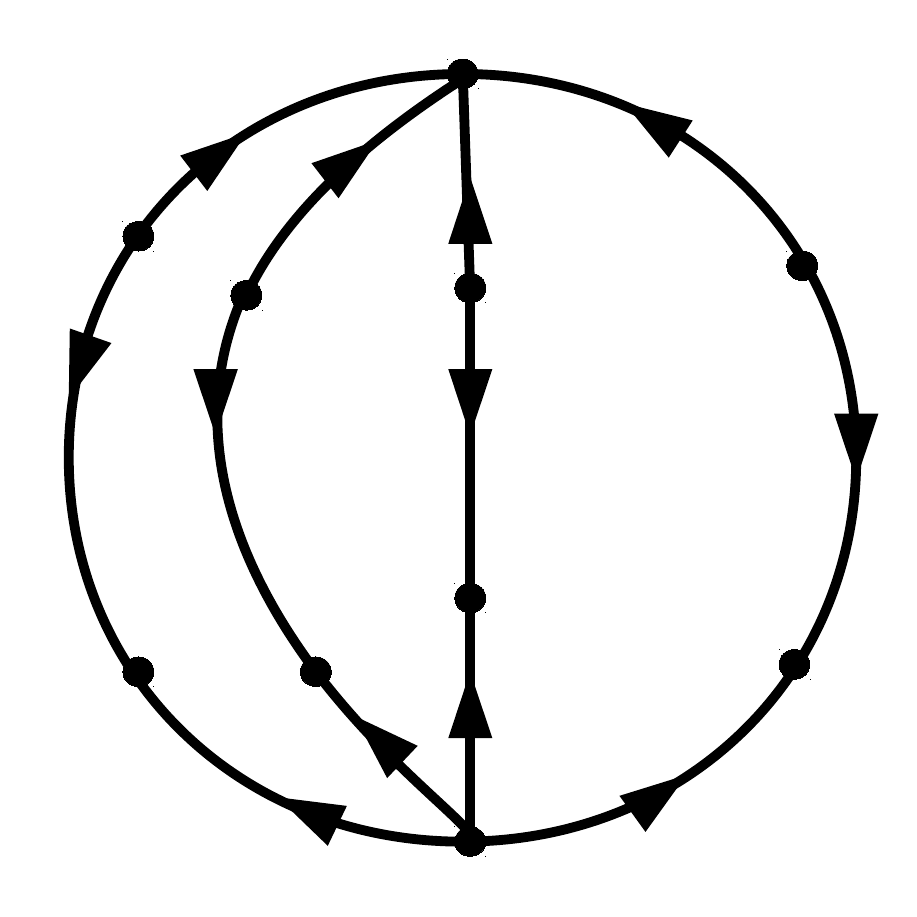}
        %\begin{overpic}[grid,tics=10,width=4cm]{inicial}
        \put(55,70){$q_1$}\put(30,60){$p_1$}\put(50,95){$n$}\put(90,70){$u_1$}
        \put(90,25){$u_2$}\put(50,0){$s$}\put(5,20){$v_1$}\put(5,75){$v_2$}\put(55,30){$q_2$}\put(35,30){$p_2$}
    \end{overpic}
\label{figinicial}\caption{Finite and infinite equilibrium of system \eqref{eq1}, systems (i).}
\end{figure}
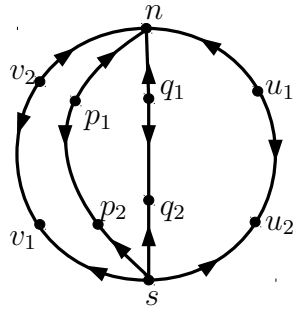

\begin{itemize}
\item [(a)] If $\Delta_{I_1}> 0$ system \eqref{eq1} has $6$ equilibrium points.
\begin{itemize}
\item $2$ saddles: $u_1$  and $v_1$ diametrically opposed to $u_1$;
\item $2$ attracting nodes: $u_2$ and $n$ the origin of $U_2$);
\item $2$ repelling nodes: $v_2$ diametrically opposed to $u_2$
and $s$  the origin of $V_2$ diametrically opposed to $n$.
\end{itemize}
\item[(b)]If $\Delta_{I_1}= 0$ system \eqref{eq1} has $4$ equilibrium points.
\begin{itemize}
\item $2$ saddle-node: $u_{12}$ and $v_{12}$ (diametrically opposed to $u_{12}$;
\item $1$ attracting node: $n$;
\item $1$ repelling node: $s$.
\end{itemize}
\item[(c)] If $\Delta_{I_1}<0$ system \eqref{eq1} has $2$ equilibrium points.
\begin{itemize}
\item $1$ attracting node: $n$;
\item $1$ repelling node: $s$.
\end{itemize}
\end{itemize}

\noindent\textit {Systems (ii):}
The expression of the system  in the local chart $U_1$ is
\begin{equation} \label{sisi2}
\dot{u} =v ((d + b u) + e v)+ p(u),  \quad \dot{v}=-v ,\end{equation}
where $p(u)= u^2+(a-1) u+c$, and in the local chart $U_2$ is

\begin{equation} \label{sisi22}
\begin{array}{lcl}\dot{u} =-v( u (d+b u)+ v u e)+q(u),  \\ \dot{v}=-v \left(c+d u+c u^2\right)-v^2 (b+du)- e v^3,\\
\end{array}
\end{equation}
where $q(u)=-u(1+(a-1) u+c u^2).$
The equilibrium  point at infinity and their classification are exactly the same of system (i).

\noindent\textit {Systems (iii):}
The expression of this system in the local chart $U_1$ is
\begin{equation} \label{sisi3}
\begin{array}{lcl}\dot{u} = v (u (b-1)+d+ev)+p(u),  \\ \dot{v}=-v^2,\\
\end{array}
\end{equation}
where $p(u)= u^2+a u+c$. System \eqref{sisi3} has $0,1$ or $2$ equilibrium points.
 $$S_i=\left(\dfrac{-a+(-1)^i\sqrt{\Delta_{I_2}}}{2 },0\right)$$
for $i=1,2$, where $\Delta_{I_2}$ is given by \eqref{deltas}. The eigenvalues of the Jacobian
matrix of system \eqref{sisi3} are $0$ and $(-1)^{i}\sqrt{\Delta_{I_2}}$. Thus we have:
\begin{itemize}
\item[(i)] If  $\Delta_{I_2}>0$ systems \eqref{sisi3} have two nilpotent saddle--nodes.
\item[(ii)] If  $\Delta_{I_2}=0$ systems \eqref{sisi3} have a saddle--node with both eigenvalues being zero.
\item[(iii)] If $\Delta_{I_2}<0$ systems \eqref{sisi3} have no equilibrium points.\end{itemize}

The expression of the system in the local chart $U_2$ is
\begin{equation} \label{sisi33}
\begin{array}{lcl}\dot{u} =v( -v (e u)- u (-1+b+d u))+q(u),  \\ \dot{v}=-v \left(1+au+cu^2\right)-v^2 (b+du)-e v^3,\\
\end{array}
\end{equation}
where $q(u)=-u(1+(a-1)u+cu^2).$
The equilibrium  points at infinity and their classification are exactly the same of systems (i).\\

In summay, the equilibrium points of system \eqref{eq1}, system $(iii)$, on the circle
$\sss^1$ are classified as follows.

\begin{itemize}
\item [(a)] If $\Delta_{I_2}> 0$ system \eqref{eq1} has $6$ equilibrium points.
\begin{itemize}
\item $4$ saddle-nodes: $u_1$ , $v_1$ diametrically opposed to $u_1$,
$u_2$ and  $v_2$ diametrically opposed to $u_2$;
\item $1$ attracting node:  $n$ ;
\item $1$ repelling node:  $s$ diametrically opposed to $n$.
\end{itemize}
\item[(b)]If $\Delta_{I_2}= 0$ system \eqref{eq1} has $4$ equilibrium points.
\begin{itemize}
\item $2$ saddle-node: $u_{12}$  and $v_{12}$
diametrically opposed to $u_{12}$;
\item $1$ attracting node: $n$;
\item $1$ repelling node: $s$.
\end{itemize}
\item[(c)] If $\Delta_{I_2}<0$ system \eqref{eq1} has $2$ equilibrium points.
\begin{itemize}
\item $1$ attracting node: $n$;
\item $1$ repelling node: $s$.
\end{itemize}
\end{itemize}

\noindent\textit {Systems (iv):}
The expression of the system in the local chart $U_1$ is
\begin{equation} \label{sisi4}
\dot{u} =v(d+ b u + (e - u) v)+ p(u),  \quad \dot{v}=-v^3,
\end{equation}
where $p(u)=u^2+a u+ c.$ System \eqref{sisi4} has $0,1$ or $2$ equilibrium points.

$$S_i=\left(\dfrac{-a +(-1)^i\sqrt{\Delta_{I_2}}}{2},0\right)$$
for $i=1,2$, where $\Delta_{I_2}$  is given by \eqref{deltas}. The eigenvalues of the Jacobian matrix of system \eqref{sisi4}
 are $0$ and $(-1)^{i}\sqrt{\Delta_{I_2}}$. Thus we have:
\begin{itemize}
\item[(i)] If  $\Delta_{I_2}>0$  systems \eqref{sisi4} have a semi-hyperbolic stable node and a semi-hyperbolic  saddle.
\item[(ii)] If  $\Delta_{I_2}=0$  systems \eqref{sisi4} have a semi-hyperbolic saddle-node.
\item[(iii)] If $\Delta_{I_2}<0$  systems \eqref{sisi4} have no equilibrium points.
\end{itemize}

The expression of the system  in the local chart $U_2$ is
\begin{equation} \label{sisi44}
\begin{array}{lcl}\dot{u} =v( v (1-e u)- u (b+d u))+q(u),  \\ \dot{v}=-v \left(1+a u+cu^2\right)-v^2 (b+d u)-e v^3,\\
\end{array}
\end{equation}
where $q(u)=-u(1+(a-1) u+c u^2).$
The equilibrium  points at infinity and their classification are exactly the same of systems (i).\\
In short, the equilibrium points of system \eqref{eq1}, systems $(iv)$, on the circle
$\sss^1$ are classified as follows.

\begin{itemize}
\item [(a)] If $\Delta_{I_2}> 0$ system \eqref{eq1} has $6$ equilibrium points.
\begin{itemize}
\item $2$ semi-hyperbolic saddles: $u_1$ and $v_1$ diametrically opposed to $u_1$;
\item $1$ semi-hyperbolic attracting node: $u_2$;
\item $1$ attracting node:  $n$ ;
\item $1$ semi-hyperbolic repelling node: $v_2$ diametrically opposed to $u_2$;
\item $1$ repelling node: $s$ diametrically opposed to $n$.
\end{itemize}
\item[(b)]If $\Delta_{I_2}= 0$ system \eqref{eq1} has $4$ equilibrium points.
\begin{itemize}
\item $2$ semi-hyperbolic  saddle-nodes: $u_{12}$  and $v_{12}$ diametrically opposed to $u_{12}$;
\item $1$ attracting node:  $n$ ;
\item $1$repelling node:  $s$ diametrically opposed to $n$.
\end{itemize}
\item[(c)] If $\Delta_{I_2}<0$ system \eqref{eq1} has $2$ equilibrium points.
\begin{itemize}
\item $1$ attracting node: $n$;
\item $1$ repelling node: $s$.
\end{itemize}
\end{itemize}

\noindent\textit {Systems (v):}
The expression of the system in the local chart $ U_1$ is
\begin{equation} \label{sisi5}
\begin{array}{lcl}\dot{u}=v((d+b u)+v (e- u))+ p(u)  \\ \dot{v}=-(v + v^3).\\
\end{array}
\end{equation}
where $p(u)= u^2+(a-1) u+c.$ The equilibrium  points at infinity and their classification are exactly the same than of systems (i).

The expression of the system in the local chart $U_2$ is
\begin{equation} \label{sisi55}
\begin{array}{lcl}\dot{u} =v( v (1-e u)- u (b+d u))+q(u),  \\ \dot{v}=-v \left(1+a u+cu^2\right)-v^2 (b+d u)-e v^3,\\
\end{array}
\end{equation}
where $q(u)=-u(1+(a-1) u+c u^2-u).$
The origin and its classification is exactly the same than of systems (i).

\end{proof}

\section{Proof  of Theorem \ref{mainteo}} \label{s4}

We start this section considering the  Tables $1,...,5,$  one for each of the
possible Riccati systems.  In each table, we list the conditions  about the parameters and indicate the
possible phase portraits.

 \subsection{Proof of Theorem \ref{mainteo}}
We remember the notation introduced in previous sections
  \[\Delta_{F_1} =b^2-4 e, \quad
\Delta_{F_2} =(b-a)^2-4  (c-d+e),\]\[ \Delta_{I_1}= (a-1)^2-4 c\quad \mbox{ and} \quad \Delta_{I_2}= a^2-4  c.\]

 \subsubsection{Proof of Theorem \ref{mainteo} -- System (i)} We begin the proof considering the assumptions of
the first row of Table \ref{t6}.
These systems have 4 finite equilibrium $p_1, p_2, q_1$, $q_2$
and 6 infinite equilibrium $ n, s, u_1, u_2, v_1, v_ 2$, according
to sections \ref{s2} and \ref{s3}, see Figure \eqref{figinicial}.

Let $r_1$  be the straight line joining $v_1, p_1 $ and $u_1$, and $r_2$ be the straight line joining $v_1, q_2 $ and $u_1$:
\[ r_1= y-u_1x-k_1=0,\quad  r_2=  y-u_1x-k_2=0, \]
where  \[k_1=  \frac{1}{2}(1-b+\sqrt{(a-1)^2-4c}+\sqrt{(a-b)^2-4(c-d+e)})\]
and \[k_2=  \frac{1}{2}(-b-\sqrt{b^2-4c}).\]

We analyze the position of $q_1$ with respect to $r_1$
and the position of $p_2$ with respect to $r_ 2$. We have four possibilities.

\begin{figure}[!htb]
    \begin{minipage}[t]{2.7cm}\psfrag{a}{$a$}\centering\includegraphics[scale=.31]{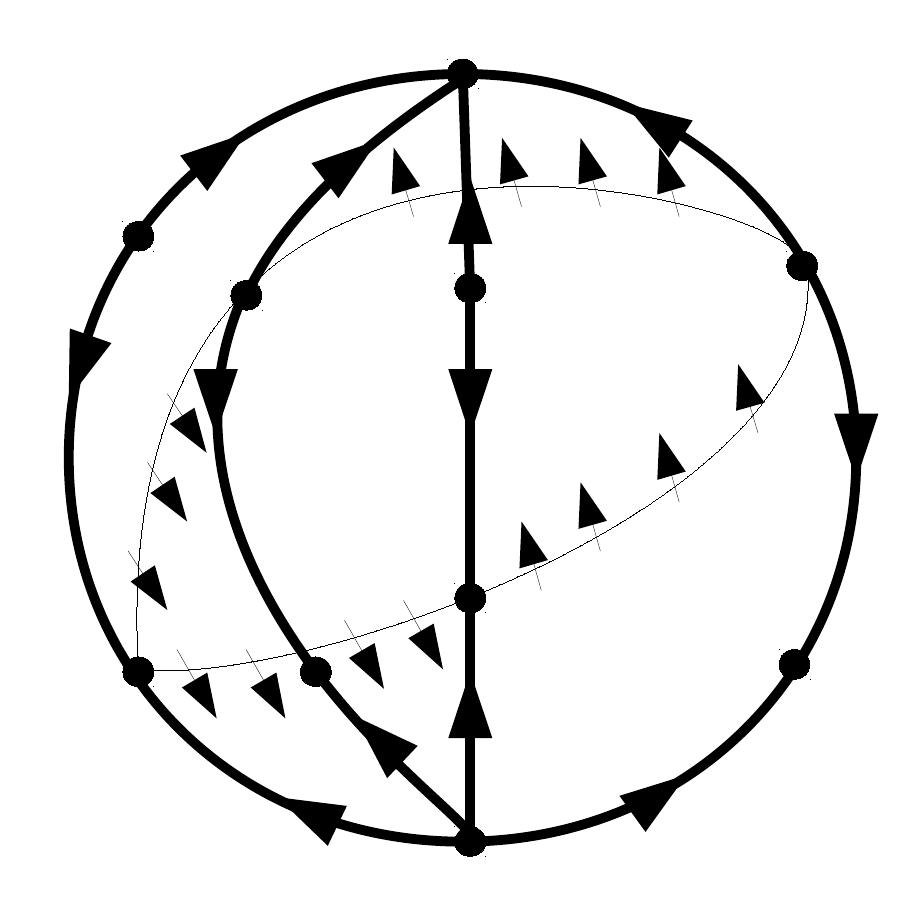}\end{minipage}
    \begin{minipage}[t]{2.7cm}\psfrag{b}{$b$}\centering\includegraphics[scale=.31]{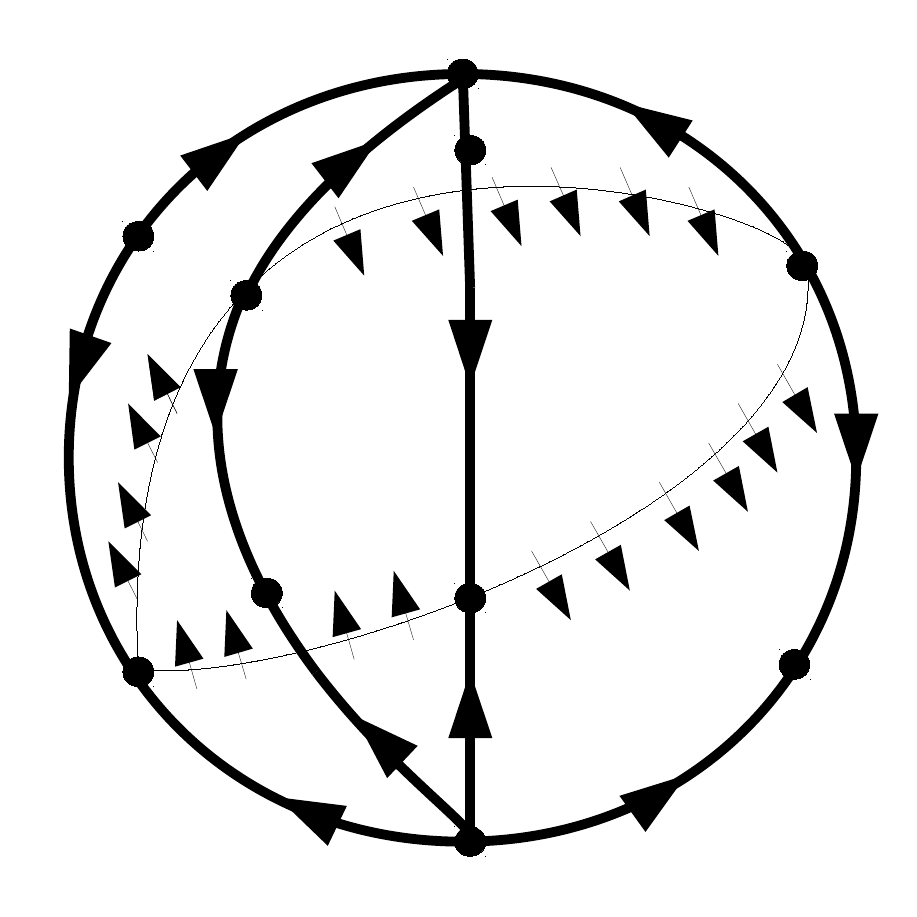}\end{minipage}
    \begin{minipage}[t]{2.7cm}\psfrag{c}{$c$}\centering\includegraphics[scale=.31]{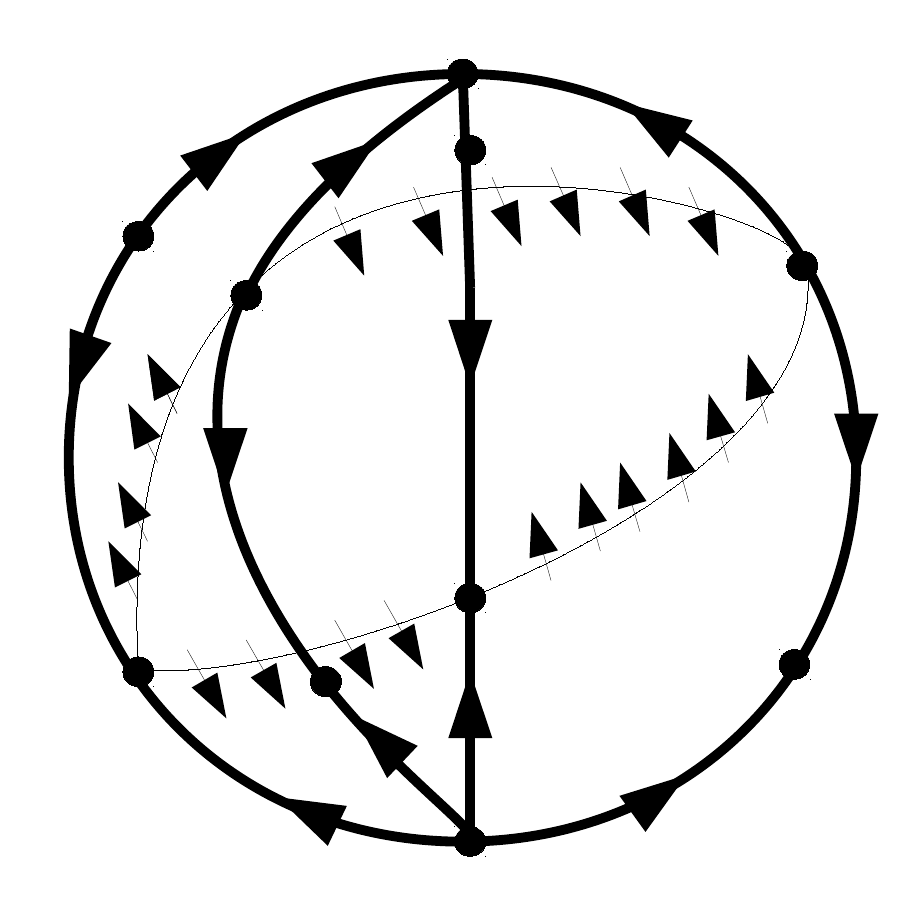}\end{minipage}
    \begin{minipage}[t]{2.7cm}\psfrag{d}{$d$}\centering\includegraphics[scale=.31]{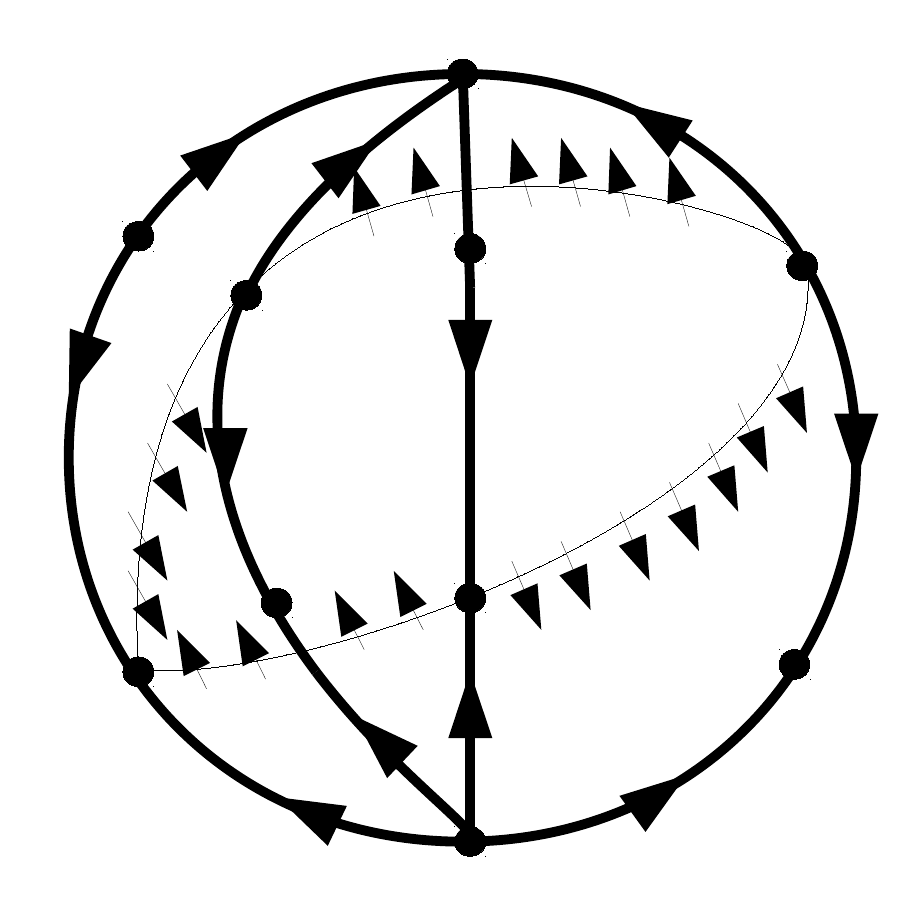}\end{minipage}
    \caption{\small Straight lines $r_1$ and $r_2$ and the directions of the vector field $X(x,y)=(\a_2(x),ky^2+\b_1(x) y + \g_2(x))$.} \label{figura2}
\end{figure}

Assume the first possibility. By Lemma \ref{L2} (see Appendix) the vector field $X(x,y)=(\a_2(x),ky^2+\b_1(x) y + \g_2(x))$ has only the equilibrium $p_1$
as a contact point with $r_1$, and the equilibrium $q_2$ as  a contact point
with $r_2$. Thus $p_1$ divides $r_1$ into two
semi-straight lines and we have the direction of the field downward
between $v_2$ and $p_1$ and upward between $p_1$ and $u_1$.
In fact this is due to the fact that  the repelling node is below the line
$r_1$, and there is a trajectory with $\alpha$-limit $q_1$ and $\omega$-limit $n$.
Similarly we concluded that  $q_2$ divides $r_2$ into two
semi-straight lines and we have the direction of the vector field downward
between $v_2$ and $q_2$ and upward between $q_2$ and $u_1$. Thus the only way to complete the phase portrait is shown in figure $P1.$\\

In the second case and in an analogous way, we conclude that the phase portrait is shown in figure $P2.$\
The third case does not occur, because the conditions $ r_1(q_1)>0$ and $ r_2(p_2)<0$ will never be satisfied at the same time.
In the fourth case we concluded that $p_1$ divides $r_1$ into two
semi-straight lines and  the direction of the field is downward
between $v_2$ and $p_1$ and upward between $p_1$ and $u_1$. Moreover
$q_2$ divides $r_2$ into two semi-straight lines and  the direction
of the field is upward between $v_2$ and $q_2$ and downward between
$q_2$ and $u_1$. There are three possibilities  to complete the phase portrait.
To analyze this case we consider the straight line $S: y=m x+ n$  joining  $p_1$ and $q_2$.  The coeficients
are
 \[ m = \frac{\pi_2(q_2)- \pi_2(p_1)}{\pi_1(q_2)- \pi_1(p_1)} = (-a - \sqrt{\Delta_{F_1}}- \sqrt{\Delta_{F_2}} )/2  \quad \mbox{and}\quad
 n = (-b - \sqrt{\Delta_{I_1}})/2\]
 where $\pi_1(x,y)= x$ and $\pi_2(x,y)= y$. We analyze how the straight line $S$ reaches the infinite. If
$ -a - \sqrt{\Delta_{F_1}}- \sqrt{\Delta_{F_2}} < 1-a- \sqrt{\Delta_{I_1}}$, then
$ u_2$ is above $S$, and the only possibility to complete the phase portrait is shown in figure $P3.$
If $ -a - \sqrt{\Delta_{F_1}}- \sqrt{\Delta_{F_2}} > 1-a- \sqrt{\Delta_{I_1}}$
then $ u_2 $ is below $S,$ and the phase portrait is shown in figure $P5.$
Finally, if $ -a - \sqrt{\Delta_{F_1}}- \sqrt{\Delta_{F_2}} = 1-a- \sqrt{\Delta_{I_1}}$
then $ u_2 $ belong to $S,$ the  phase portrait is shown in figure $P4.$\\
Now we explicit the parameter values for each phase portrait.
\begin{itemize}

\item   $P_1$: $(a, b, c, d, e) = (0,0,0,3.75, -0.25)$.
\item   $P_2$: $(a, b, c, d, e) = (0,0,0,-3.75, -4)$.
\item   $P_3$: $(a, b, c, d, e) =(0,0,-0,75,-0.75,-0.25)$.
\item   $P_4$: $(a, b, c, d, e) = (0,0,-2,-2, -0.25)$.
\item   $P_5$: $(a, b, c, d, e) =(0,0,-3.75,-3.75,-0.25)$.
\end{itemize}

Assume the conditions in the second row of Table \ref{t6}. systems (i) have 3 finite equilbria $p_{1,2}, q_1$, $q_2$
and 6 infinite equilibria $ n, s, u_1, u_2, v_1, v_ 2$. Note that $p_{1,2}$ comes from the collision of $p_1$ and $ p_2$
(these equilibria exist when we assume the conditions of the first row of Table \ref{t6}) when $\Delta_{F_2}\rightarrow0$.
Consequently systems (i) have at most five phase portraits which are obtained from the 5 possible phase portraits
of row $1$ of Table \ref{t6}. Applying Lemma 7, we can see that effectively only the 4 phase portraits
listed in row 2 of Table \ref{t6} occur. Next we explicit the parameter values for each phase portrait.

\begin{itemize}

\item  $P_6$: $(a, b, c, d, e) = (0,0,0,-3, -3)$.
\item  $P_7$: $(a, b, c, d, e) = (0,0,-1,-2, -1)$.
\item  $P_9$: $(a, b, c, d, e) =(0,0,-29,-30,-1)$.
\item  $P_8$: We cannot explicit a choice of $(a, b, c, d, e)$. However its
existence follows from continuity when we pass from  the phase portraits $P_7$ to $P_9$.

\end{itemize}

The analysis of the phase portraits for the conditions listed in the other rows of  Table \ref{t6}
is analogous to the one that we did above.  We will only give an example for each phase portrait.

\begin{itemize}

\item   $P_{10}$: $(a, b, c, d, e) = (0,0,0,-1, -0.25)$.
\item  $P_{11}$: $(a, b, c, d, e) = (0,0,-1,10, 0)$.
\item  $P_{12}$: $(a, b, c, d, e) =(2,0,-1,-1,0)$.
\item  $P_{13}$: $(a, b, c, d, e) =(1,0,-1,-1,0)$.
\item  $P_{14}$: $(a, b, c, d, e) =(-2,0,-1,-1,0)$.
\item  $P_{15}$: $(a, b, c, d, e) =(4,0,2,-2,0)$.
\item  $P_{16}$: $(a, b, c, d, e) =(1,0,-1,-1.25,0)$.
\item  $P_{17}$: $(a, b, c, d, e) =(0,0,-1,-10,0)$.
\item  $P_{18}$: $(a, b, c, d, e) =(0,1,0,1,0.75)$.
\item  $P_{19}$: $(a, b, c, d, e) =(0,0,-0.75,0.25,1)$.
\item  $P_{20}$: $(a, b, c, d, e) =(1,0,-1,-1.25,1)$.
\item  $P_{21}$: $(a, b, c, d, e) =(1,2,0,4.75,0)$.
\item  $P_{22}$: $(a, b, c, d, e) =(2,0,0.25,-9.75,-10)$.
\item  $P_{23}$: $(a, b, c, d, e) =(2,0,0.25,-0.75,-1)$.
\item  $P_{24}$: $(a, b, c, d, e) =(0,0,0.25,-0.75,-1)$.
\item  $P_{25}$: $(a, b, c, d, e) =(1,1,0,0.2,0.2)$.
\item  $P_{26}$: $(a, b, c, d, e) =(0,0,0.25,-1.75,-1)$.
\item  $P_{27}$: $(a, b, c, d, e) =(0,0,0.25,1.25,0)$.
\item  $P_{28}$: $(a, b, c, d, e) =(1,1,0,0.25,0.25)$.
\item  $P_{29}$: $(a, b, c, d, e) =(1,0,0,-1,0)$.
\item  $P_{30}$: $(a, b, c, d, e) =(0,0,0.25,2.25,1)$.
\item  $P_{31}$: $(a, b, c, d, e) =(0,0,0.25,1.25,1)$.
\item  $P_{32}$: $(a, b, c, d, e) =(0,0,0.25,0.25,1)$.
\item  $P_{33}$: $(a, b, c, d, e) =(0,0,1.25,1.25,-1)$.
\item  $P_{34}$: $(a, b, c, d, e) =(1,1,1,1.2,0.2)$.
\item  $P_{35}$: $(a, b, c, d, e) =(1,1,2,0,0.2)$.
\item  $P_{36}$: $(a, b, c, d, e) =(1,2,1,2,1)$.
\item  $P_{37}$: $(a, b, c, d, e) =(2,0,2,1,0)$.
\item  $P_{38}$: $(a, b, c, d, e) =(1,0,1,-1,0)$.
\item  $P_{39}$: $(a, b, c, d, e) =(1,0,1,2,1)$.
\item  $P_{40}$: $(a, b, c, d, e) =(1,0,1,1.75,1)$.
\item  $P_{41}$: $(a, b, c, d, e) =(0,0,1.25,1.25,1)$.

\end{itemize}

\subsubsection{Proof of Theorem \ref{mainteo} --System (ii)}
The phase portraits listed in row 1 of Table \ref{t2} are obtained from  row 1 of Table \ref{t6}.
Note that system (ii) has only $ x = 0 $ as an invariant vertical line, which comes when
the two straight lines $x=0, x=-1$ of system (i) collide at $x=0$.
Thus we consider the phase portraits represented in the figures $P1, P2, P3, P4$ and $P5$, excluding what occurs in
the strip $ -1\leq x \leq 0$. This reduces the possible phase portraits to $P42, P43$ and $P44$
obtained from $P1, P2$ and $P3$ respectively. Note that no new configurations can be obtained from $P4$
 and $P5$ because the phase portraits are equal in the complement of the strip $ -1 \leq x \leq 0 $.
 The possibilities listed in the other rows of Table \ref{t2} are obtained in a similar way.
 Below we list values of the parameters that realize each one of the possible phase portraits.\\

 \begin{itemize}

\item   $P_{42}$: $(a, b, c, d, e) = (1,1,-1,4, -1)$.
\item  $P_{43}$: $(a, b, c, d, e) = (1,1,-1,-4, -1)$.
\item  $P_{44}$: $(a, b, c, d, e) =(1,1,-1,0,-1)$.
\item  $P_{45}$: $(a, b, c, d, e) =(1,1,-1,4,0.25)$.
\item  $P_{46}$: $(a, b, c, d, e) =(1,1,-1,0,0.25)$.
\item  $P_{47}$: $(a, b, c, d, e) =(2,1,-1,1,0.25)$.
\item  $P_{48}$: $(a, b, c, d, e) =(1,2,-1,0,2)$.
\item  $P_{49}$: $(a, b, c, d, e) =(1,1,0,4,-1)$.
\item  $P_{50}$: $(a, b, c, d, e) =(1,1,0,-2,-1)$.
\item  $P_{51}$: $(a, b, c, d, e) =(1,1,0,0,-1)$.
\item  $P_{52}$: $(a, b, c, d, e) =(1,1,0,4,0.25)$.
\item  $P_{53}$: $(a, b, c, d, e) =(2,1,0.25,1,0.25)$.
\item  $P_{54}$: $(a, b, c, d, e) =(1,1,0,-2,0.25)$.
\item  $P_{55}$: $(a, b, c, d, e) =(1,1,0,0,1)$.
\item  $P_{56}$: $(a, b, c, d, e) =(1,1,1,0,0.2)$.
\item  $P_{57}$: $(a, b, c, d, e) =(1,1,1,0,0.25)$.
\item  $P_{58}$: $(a, b, c, d, e) =(2,1,0.3,1,0.25)$.

\end{itemize}

\subsubsection{Proof of Theorem \ref{mainteo} --System (iii)} If  $\Delta_{I_2}>0$ and
$\Delta_{F_I}>0$, corresponding to the case considered in
the first row of Table \ref{t3}, systems (iii)
have 2 finite equilibria $q_1$, $q_2$
and 6 infinite equilibria $ n, s, u_1, u_2, v_1, v_ 2$, according
to sections \ref{s2} and \ref{s3}.

We consider  the straight line $r$
joining $v_2, q_1 $ and $u_2$. Applying Lemma \eqref{L2} we can prove that the following configurations cannot occur:
\begin{itemize}
\item [(a)] both unstable separatrix of $ q_2 $ have $\omega$-limit  $n$;
\item [(b)] the left hand side of unstable separatrix of $ q_2 $  has $ \omega $-limit $n$ and the
right hand side separatrix of $ q_2 $  has $ \omega $-limit  $ u_1 $;
\item [(c)] the left hand side of unstable separatrix of  $ q_2 $ has $ \omega $-limit  $v_2 $ and the
right hand side separatrix of $ q_2 $  has $ \omega $-limit  $n$;
\item [(d)] the left hand side of unstable separatrix of  $ q_2 $ has $ \omega $-limit  $v_2 $ and the
right hand side separatrix of $ q_2 $ has $ \omega $-limit  $u_1$;
\item [(e)] the left hand side of unstable separatrix of $ q_2 $ has $ \omega $-limit  $u_1$ and the
right hand side separatrix of $ q_2 $ has $\omega $-limit  $v_2$.
\end{itemize}

Taking into account this previous informative the only possible phase portraits are  $ P_{59},P_{60} $ and $ P_{60}$ remain.
The other lines of Table \ref {t3} are similarly analyzed.
Below we list the parameter values that realize each one of the possible phase portraits.\\

 \begin{itemize}

\item   $P_{59}$: $(a, b, c, d, e) = (1,2,0.2,1,0.2)$.
\item  $P_{60}$: $(a, b, c, d, e) = (1,1,0.2,2,0.2)$.
\item  $P_{61}$: $(a, b, c, d, e) =(1,1,0.2,* *,0.2)$.
\item  $P_{62}$: $(a, b, c, d, e) =(2,2,0.2,1,1)$.
\item  $P_{63}$: $(a, b, c, d, e) =(1,2,0.2,1,1)$.
\item  $P_{64}$: $(a, b, c, d, e) =(* *,2,0.2,1,1$.
\item  $P_{65}$: $(a, b, c, d, e) =(8,2,2,1,5)$.
\item  $P_{66}$: $(a, b, c, d, e) =(2,2,1,1,0.2)$.
\item  $P_{67}$: $(a, b, c, d, e) =(2,1,1,1,0.2)$.
\item  $P_{68}$: $(a, b, c, d, e) =(2,2,1,1,1)$.
\item  $P_{69}$: $(a, b, c, d, e) =(1,2,0.25,1,1)$.

\end{itemize}

\subsubsection{Proof of Theorem \ref{mainteo} --Systems (iv) and (v)} The classification given in Tables 4 and 5
follows directly from the analysis of singularities at infinity. We list a parameter value that realize each phase portrait.

\begin{itemize}

\item  $P_{70}$: $(a, b, c, d, e) = (1,1,0,0,0)$.
\item  $P_{71}$: $(a, b, c, d, e) =(1,1,0,0,1)$.
\item  $P_{72}$: $(a, b, c, d, e) = (1,1,0,0,0)$.
\item  $P_{73}$: $(a, b, c, d, e) =(1,1,0,0,1)$.
\item  $P_{74}$: $(a, b, c, d, e) =(1,1,0,0,-1)$.

\end{itemize}

\section{Appendix: Semi-hyperbolic equilibrium points}\label{a1}

The following two lemmas are very useful in the proofs and they proved in Chapter 11 of \cite {Ye}.

\begin{lemma}\label{L1}
If the straight line passing through two singular points $S_1$ and
$S_2$ of a quadratic system is not an integral line, then it must be
formed by three open line segments without contact points $\overline
{\infty S_1}$, $\overline {S_1 S_2}$ and $\overline {S_2\infty}$.
Moreover the trajectories cross $\overline {\infty S_1}$ and
$\overline {S_2\infty}$ in one direction, and cross $\overline {S_1
S_2}$ in the opposite direction.
\end{lemma}

\begin{lemma}\label{L2}
The straight line connecting one finite singular point and a pair of
infinite singular points in a quadratic system is either formed by
trajectories or it is a line with exactly one contact point. This
contact point is the finite singular point. For the latter case the
flow goes in different directions on each half--line.
\end{lemma}

\end{document}